\documentclass{article}

\usepackage{xurl}
\usepackage[english]{babel}
\usepackage[T1]{fontenc}
\usepackage{
    amsmath,
    amssymb,
    amsthm,
    mathtools,
    mathrsfs,
    hyperref,
    fullpage,
    listings,
    xcolor,
    enumitem,
    stmaryrd,
    multicol,
    array,
    mdframed,
    multirow,
    multicol,
    caption,
    subcaption,
    orcidlink
}
\usetikzlibrary{decorations.markings}
\usepackage[table]{xcolor}
\newcommand{\legendbox}[1]{
  \raisebox{0.2ex}{
    \fcolorbox{black}{#1}{\rule{0pt}{0.3em}\rule{0.8em}{0pt}}
  }
}

\theoremstyle{definition}
\newtheorem{Def}{Definition}
\theoremstyle{plain}
\newtheorem{Thm}[Def]{Theorem}
\newtheorem{Prop}[Def]{Proposition}
\newtheorem{Corol}[Def]{Corollary}
\newtheorem{Lemma}[Def]{Lemma}

\newtheorem{Rq}[Def]{Remark}

\newcommand*\samethanks[1][\value{footnote}]{\footnotemark[#1]}
\author{Eimear Byrne\thanks{School of Mathematics and Statistics, University College Dublin, \texttt{ebyrne@ucd.ie} \orcidlink{0000-0002-1857-0365} and \texttt{lucien.francois@ucd.ie} \orcidlink{0009-0006-4810-6068}}, Lucien François\samethanks[1]}

\title{The multilinear forms Cayley graph and the eigenvalue method for tensor codes.}
\date{}

\newcommand{\C}{\mathbb{C}}
\newcommand{\CCC}{\mathcal{C}}%
\newcommand{\DDD}{\mathcal{D}}%
\newcommand{\E}{\mathbb{E}}
\newcommand{\F}{\mathbb{F}}
\newcommand{\FFF}{\mathcal{F}}%
\newcommand{\III}{\mathcal{I}}%
\newcommand{\N}{\mathbb{N}}
\newcommand{\Nnn}{\mathfrak{n}}%
\newcommand{\R}{\mathbb{R}}
\newcommand{\RR}{\mathscr{R}}
\renewcommand{\S}{\mathbb{S}}
\newcommand{\VVV}{\mathcal{V}}%
\newcommand{\Z}{\mathbb{Z}}

\DeclareMathOperator{\Bil}{Bil}
\DeclareMathOperator{\Cay}{Cay}

\DeclareMathOperator{\codim}{codim}

\DeclareMathOperator{\col}{col}
\DeclareMathOperator{\colsp}{colsp}

\DeclareMathOperator{\diag}{diag}
\DeclareMathOperator{\Endo}{End}

\DeclareMathOperator{\GL}{GL}

\DeclareMathOperator{\Hom}{Hom}

\DeclareMathOperator{\im}{im}

\DeclareMathOperator{\krank}{k-rank}

\DeclareMathOperator{\modemult}{m}

\DeclareMathOperator{\rank}{rank}

\DeclareMathOperator{\row}{row}
\DeclareMathOperator{\rowsp}{rowsp}
\DeclareMathOperator{\Span}{Span}

\DeclareMathOperator{\Stab}{Stab}
\DeclareMathOperator{\slicesp}{ss}

\DeclareMathOperator{\Tr}{Tr}

\newcommand{\gbc}[3]{\begin{bmatrix}#2\\#1\end{bmatrix}_{\!#3}}
\newcommand{\sgbc}[3]{\left[\begin{smallmatrix}#2\\#1\end{smallmatrix}\right]_{#3}}

\begin{document}

\maketitle
\begin{abstract}
    The connections between graph theory, and more generally association schemes, and coding theory were established by Delsarte for the Hamming metric and rank-metric codes. The ambient metric space of Hamming-metric codes and rank-metric codes can be seen as Cayley graphs generated by words of weight one. The metrics considered then coincide with the geodesic distances of these distance-regular graphs. We focus on a generalisation of this framework to the space of tensors over a finite field, endowed with the tensor rank as a metric. This space corresponds to the Cayley graph generated by rank-one tensors, which is not distance-regular for tensors of order at least $3$. We show that the spectrum of this graph has a recursive expression and depends on the possible intersections between tensor subspaces of large enough dimension and the Segre variety. The spectrum of this graph for $3$-order tensors can be expressed with the rank distribution of the rank-metric codes generated by these tensors. In particular, we obtain the complete spectrum of the graph for $2\times 3 \times 3$ tensors over any finite field. We apply this result to derive bounds on the dimension of tensor codes in the tensor rank metric using the eigenvalue method, and in particular the ratio-type bound.

\end{abstract}



\setcounter{tocdepth}{2}

\section{Introduction}

The application of graph theory and algebraic combinatorics to coding theory has been studied extensively throughout the past five decades. Delsarte laid the foundation of the theory of association schemes in coding theory with the introduction of the Hamming scheme and the bilinear forms scheme \cite{Delsarte1973AlgebraicAproachAssociationSchemes,DELSARTE1978226,DelsarteLevenshtein1998AssociationSchemesSurvey}. 
The fact that the scheme induced by a given metric is an association scheme relies on the distance-regularity of the distance graph of the ambient space. In both the Hamming and rank-metric case, the invariants of the corresponding association schemes yield MacWilliams identities. This was applied to produce a linear programming bound on the size of a code for a given minimum distance in the Hamming metric case. 

The eigenvalue method provides another approach to obtain similar results without the distance-regularity assumption, and has been applied to codes for the sum-rank metric, the Lee metric, and to alternating rank-metric codes; see \cite{AbiadPetersRavagnani2025EigenvalueMethod} and the references therein. In this approach, elements of the ambient space $V$ of a code over a finite field can be seen as the vertices of a graph for which the geodesic distance coincides with the distance endowed on $V$. This point of view allows one to obtain upper bounds on the size of a code of a given minimum distance using the known bounds on the independence numbers of the graph such as the inertia-type and ratio-type bounds \cite{AbiadCoutinhoFiol2019kindependancenumber}. There also exist linear programming and mixed-integer linear programming bounds on the independence numbers using the spectrum of the graph with the aforementioned bounds \cite{AbiadAlfranoRavagnani2025EVboundsAlternatingRMC,AbiadNeriReijnders2025EVboundsLeeCodes,AbiadPetersRavagnani2025EigenvalueMethod}. 

The (tensor) rank of an $m$-order tensor is a generalisation of the rank of a matrix and measures the algebraic complexity of the multilinear map associated to the tensor \cite{Burgisser1997ch14}. The rank of a tensor $x$ is the least number of simple tensors whose sum is equal to $x$. This determines a distance function on the space of all $m$-order tensors whereby the (tensor) rank distance between a pair of tensors is the rank of their difference. Tensor codes were introduced in \cite{RothMaximumArrayCodes} as generalisation of rank-metric codes, motivated in part by applications to criss-cross error-correction. They are subspaces or subsets of the vector-space of $m$-order tensors that has been endowed with the tensor rank as a metric; see \cite{ByrneCotardo2023TensorCodesInvariants,byrne2026decodingalgorithmstensorcodes,ROTHTensorCodesForRankMetric}. 
Such codes generalise both Hamming-metric codes, rank-metric codes and sum-rank metric codes. Indeed, there exist subspaces of the space of tensors with the tensor rank as a metric that are isometric to any given Hamming metric codes or rank-metric ambient space. Additionally, 
a (matrix) rank-metric code can be seen as a slice space, or contraction space, of a 3-tensor called its generator tensor \cite{ByrneNeriRavagnaniSheekeyTensorRepresentation2019}. This definition can be generalised to $m$-order tensor codes for any $m\geq 3$, as such a code can be seen as the slice space of an $(m+1)$-order tensor. 

Unlike rank-metric codes, the `main coding theory problem' is unsolved for tensor codes of higher order. That is, the largest cardinality of an $m$-order tensor code for a fixed minimum rank distance is not known in general. As described in \cite{AbiadPetersRavagnani2025EigenvalueMethod}, we construct a graph $\Gamma$ whose vertices are joined by an edge if the distance between them is $1$. 
In our setting, $\Gamma$ is the \emph{multilinear forms graph}. The vertices of $\Gamma$ are the distinct elements of $\F_q^{n_1\times \dots \times n_m}$ (the space of $m$-order tensors of size $n_1 \times \cdots \times n_m$ over the finite field $\F_q$), any two of which form an edge of $\Gamma$ if the rank distance between them is 1. In other words, $\Gamma$ is the Cayley graph generated by elements of rank one in $\F_q^{n_1\times \dots \times n_m}$ (i.e., the Segre variety of $\F_q^{n_1\times \dots \times n_m}$).
For $m = 2$, this corresponds to the distance-regular graph of the bilinear forms association scheme studied in \cite{DELSARTE1978226}. As we show later, for $m \geq 3$, this Cayley graph is not distance-regular and hence Delsarte's method does not directly apply. However, $\Gamma$ has a number of properties that make it amenable to the eigenvalue method described in \cite{AbiadPetersRavagnani2025EigenvalueMethod}, such as the fact the geodesic distance in $\Gamma$ coincides with the tensor rank distance.

In this work, we study the properties of the aforementioned Cayley graph $\Gamma$, with a particular focus on its eigenvalues.
Note that in \cite{Lovasz75SpectraOfGraphsWithTransitiveGroups} it is shown that each eigenvalue of a Cayley graph can be expressed as a character sum over the generating set. We exploit this fact to obtain expressions for the eigenvalues of the multilinear forms graph. We first provide an iterative formula for the computation of the eigenvalues of the graph as a function of the eigenvalues of the graph for lower order tensors; see Theorem~\ref{thm:eigenvaluesmultilinearformsgraphrecursive}. We then provide a dual expression of the same eigenvalues associated to an $m$-order tensor as an affine function of the number of rank-one tensors in the dual of one of its slice spaces; see Theorem~\ref{thm:eigenvaluesmultilineargraphdual}. This in turn gives a one-to-one correspondence between the eigenvalues of the multilinear forms graph on $\F_{q}^{n_1\times \dots \times n_m}$ and the cardinalities of the intersections between subspaces of large enough dimension in $\F_q^{n_1\times \dots \times n_{j-1}\times n_{j+1}\times \dots \times n_m}$ and the Segre variety of that space.

With these expressions, computing the exact eigenvalues and multiplicities of a multilinear forms graph with a given set of parameters remains a hard problem, as it requires at least to know all possible intersection cardinalities between the Segre variety and subspaces of $(m-1)$-order tensors, as well as the precise number of subspaces in each configuration. Using the known classification of $2\times 3 \times 3$ tensors up to equivalence \cite{LavrauwSheekey2015CanonicalFormsTensors} as well as the known classification of subspaces of $\F_q^{2\times 3}$ up to equivalence \cite{LavrauwSheekey2017ClassificationSubspacesFqr}, we provide the spectrum of the trilinear forms graph for $2\times 3 \times 3$ tensors over any finite field $\F_q$. The spectrum can be used to compute ratio-type bounds \cite{AbiadCoutinhoFiolNoguieraZeijlemaker2022Optimizationindependancenumber} on the cardinality of such tensor codes with a given minimum distance 

The outline of this paper is as follows. In Section \ref{sec:prerequisites}, we introduce the necessary background, definitions, and properties of graphs and tensors required for the rest of the paper. We describe tensors of order $m$. Additionally, we introduce a number of tools from graph theory, character theory, and coding theory that are relevant for the study of the graph eigenvalues and the application to coding theory. In Section \ref{sec:multilinearfromsgraph}, we introduce the notion of the multilinear forms graph as the Cayley graph of tensors over a finite field generated by the set of simple tensors. We show that the graph is not distance-regular for $m\geq 3$. We obtain a recursive expression for the eigenvalue associated to an $m$-order tensor as a function of the eigenvalues associated to the elements of any one of its slice spaces. We use this to give a direct expression of the eigenvalues of the trilinear forms graph. We show that the eigenvalue associated to an order $3$ tensor is determined by the parameters of the ambient space and an evaluation of the weight enumerator of any one of its slice spaces. Using the recursion developed earlier, for any integer $m\geq2$, we provide an expression of the eigenvalue of an $m$-order tensor as an affine function of the number of rank-one elements in any of its slice spaces. We conclude by describing the spectrum of the Cayley graph generated by the Segre variety of $\F_q^{n_1\times \dots \times n_m}$. In Section \ref{sec:eigenvalues233}, we provide the expression of the spectrum of the trilinear forms graph of $2\times 3 \times 3$ tensors using the known classification of tensors and subspaces up to equivalence. We then deduce bounds on the cardinality of tensor codes of this size by utilising the computation of the spectrum of the trilinear forms graph and the ratio-type bounds on the $2$- and $3$-independence numbers. Finally, in Section \ref{sec:bounds}, we conclude the paper by computing the ratio-type bounds for multilinear forms graph with low parameters using mixed-integer linear programming.

A complementary GitHub repository (\url{https://github.com/lucienfrancois/EigenvaluesMultilinearFormsGraph}) provides MAGMA and SageMath code, respectively, to compute the spectrum of multilinear forms graphs for small enough parameters and deduce ratio-type bounds on tensor codes with the same parameters. The latter is largely based on the code given in \cite{Peters2024BoundingCardinality}.

\section{Preliminaries}
\label{sec:prerequisites}
\subsection{Notation}
For the rest of the paper, unless stated otherwise, we adopt the following notation.
We denote by $\N$ the set of positive integers and by $\N_0$ the set of non-negative integers. If $a$ and $b$ are integers, we denote by $\llbracket a,b \rrbracket = [a,b] \cap \Z$.
We fix $p$ to be a prime number and let $q = p^\delta$ for some positive integer $\delta$. We denote by $\F_p$ the finite field with $p$ elements and by $\F_q$ its extension of degree $\delta$. For each integer $n \in \N$, we denote by $\{e_1^{(n)},\dots,e_n^{(n)}\}$ the canonical basis of the vector space $\F_{q}^n$. Denote by $\Tr : \F_q \to \F_p$ the field trace map of this field extension, that is, the $\F_p$-linear map given by $\Tr(y) = y + y^p + \cdots + y^{p^{\delta-1}}$ for each $y \in \F_q$. For $n$ and $k$ two non-negative integers, we denote by $\sgbc{k}{n}{q}$ the $q$-binomial coefficient, that is, the number of $k$-dimensional subspaces of $\F_q^n$, which is given by 
\[
    \gbc{k}{n}{q} = \left\{\begin{array}{ll}
         \displaystyle \prod_{s = 1}^k \frac{q^{n-s+1}-1}{q^s-1}&\text{if }n\geq k\geq 0,  \\[0.5cm]
         0& \text{otherwise.}
    \end{array} \right.
\] 
We let $m$ denote a positive integer and let $\mathbf n = (n_1,\dots,n_m)$ be a tuple of positive integers of length $m$. We will use the following notation for the subsequences of $\mathbf n$. For each integer $j \in \llbracket 1,m\rrbracket$, we denote by $\mathbf n(j) = (n_1,\dots,n_{j-1},n_{j+1},\dots,n_m)$ the tuple obtained by removing the $j$-th entry of $\mathbf{n}$. Likewise, for each integer $j,\ell \in \llbracket 1,m\rrbracket$ such that $j < \ell$, we denote the tuple obtained after removing both $j$-th and $\ell$-th entries of $\mathbf n$ by $\mathbf n(j,\ell) = \mathbf n(\ell,j)= (n_1,\dots,n_{j-1},n_{j+1},\dots,\dots,n_{\ell-1},n_{\ell+1},\dots,n_m)$. 
Finally, we let $\R_k[X]$ denote the ring of univariate real polynomials of degree at most $k$.

\subsection{Tensors and tensor rank}

In this section, we will define the notions associated to tensors that will be of use in the following sections of the paper. For further details, we refer the reader to \cite{Burgisser1997ch14,ByrneNeriRavagnaniSheekeyTensorRepresentation2019,KRUSKAL197795,Lang2002TensorProduct}.

For each sequence of positive integers $\mathbf{n} = (n_1,\dots,n_m)$, we denote by $\F_q^{\mathbf{n}}$ the $n_1n_2\cdots n_m$-dimensional vector-space of $n_1 \times \cdots \times n_m$ tensors over $\F_q$ given by
\[
    \F_q^{\mathbf{n}} = \F_q^{n_1 \times \cdots \times n_m} = \F_q^{n_1} \otimes \cdots \otimes \F_q^{n_m}.
\]
If $m = 2$, we will identify the space of tensors $\F_{q}^{n_1 \times n_2}$ with the space of matrices of size $n_1 \times n_2$ over $\F_q$ with the relation $u \otimes v = u^\top v$, for each $u \in \F_q^{n_1}$ and $v \in \F_q^{n_2}$. We denote by $\S_{\mathbf n}$ the set of simple tensors in $\F_q^{\mathbf n}$, which is the set
\[
    \S_{\mathbf n} = \S_{n_1 \times \cdots \times n_m} = \left\{u_1 \otimes \cdots \otimes u_m \ : \ (u_1,\dots,u_m) \in (\F_q^{n_1} \backslash\{0\}) \times \cdots \times (\F_q^{n_m}\backslash\{0\})\right\} \subseteq \F_{q}^{\mathbf n}.
\]
The set of perfect tensors of $\F_{q}^{\mathbf n}$ has cardinality
\[
   \displaystyle \left|\S_{\mathbf n}\right| = \frac{\prod_{j=1}^m (q^{n_j}-1)}{(q-1)^{m-1}}.
\]
Let us denote by $\GL_{\mathbf{n}}(\F_q)$ the group
$
    \GL_{\mathbf{n}}(\F_q) = \GL_{n_1 \times \cdots \times n_m }(\F_q)= \GL_{n_1}(\F_q) \times \cdots \times \GL_{n_m}(\F_q).
$
For each $x \in \F_q^{\mathbf{n}}$, we denote by $(x_i)_{ i \in \prod_{j=1}^m \llbracket 1,n_j\rrbracket}$ the unique family of scalars such that
\[
    x = \sum_{i \in \prod_{j=1}^m \llbracket 1,n_j\rrbracket} x_{ i} e_{i_1}^{(n_1)} \otimes \cdots \otimes e_{i_m}^{(n_m)}.
\]

\begin{Def}
    For each tuple of positive integers $\mathbf n = (n_1,\dots,n_m)$, we define the dot-product on the space $\F_q^{\mathbf n}$ to be the bilinear map
    \[
        \langle \cdot | \cdot \rangle : \F_q^{\mathbf n} \times \F_q^{\mathbf n} \to \F_q, (x,y) \mapsto \langle x | y \rangle = \sum_{i \in \prod_{j=1}^m\llbracket 1,n_j\rrbracket} x_{i}y_i.
    \]
    Then, if $\mathcal C$ is a subspace of $\F_q^{\mathbf n}$, we denote by $\mathcal C^\perp$ the dual-space of $\mathcal C$ with respect to the bilinear map $\langle \cdot | \cdot \rangle$, which is the space $\CCC^\perp = \{ y \in \F_q^{\mathbf n} \ : \ \forall x \in \CCC, \langle x | y \rangle = 0\}$.
\end{Def}

\begin{Def}
    Let $x \in \F_q^{\mathbf n}$ be an $m$-order tensor. The least integer $\tau \in \N$ such that for each $j \in \llbracket 1,m\rrbracket$ there exists a sequence $(v_{i,j})_{i \in \llbracket 1,\tau\rrbracket} \in (\F_q^{n_j})^\tau$ such that
    \begin{equation}
        \label{eq:minimaltensorrankform}
        x = \sum_{i = 1}^\tau v_{i,1}\otimes \cdots \otimes v_{i,m}
    \end{equation}
    is called the {tensor rank} (or rank) of $x$ and is denoted by $\rank(x)$. If $\tau = \rank(x)$, then the sum in equation (\ref{eq:minimaltensorrankform}) is called a {minimal rank form} of $v$.
\end{Def}

Similarly, the rank of a matrix $M$ is the least amount $t$ of matrices of rank one $A_1,\dots,A_t$ such that $M = A_1 + \cdots + A_t$. In that sense, the notion of tensor rank is a generalisation of the rank of a matrix. In the matrix case, as well as the tensor case, an element may have several minimal rank forms. Contrary to the rank of a matrix, the maximum possible value of the tensor rank of an element in $\F_q^{\mathbf n}$ is not necessarily known. We recall the following upper bounds presented in \cite{HOWELL19789} and generalised in \cite{RothMaximumArrayCodes}.

\begin{Prop}[{\cite{HOWELL19789}}] \label{prop:howellbound}
    Let $x \in \F_q^{\mathbf n}$ be a tensor. Then 
    \[
        \rank(x) \leq \frac{\prod_{j = 1}^m n_j}{\max_{j \in \llbracket 1,m\rrbracket} n_j}.
    \]
    Moreover, if there exist integers $\ell,n \in \N$ such that $\ell \leq n$ and $\mathbf n \in\{ (\ell,n,n),(n,\ell,n),(n,n,\ell)\}$, then we have
    \[
        \rank(x) \leq n\ell - \left\lfloor \frac{\ell^2}{4} \right\rfloor.
    \]
\end{Prop}

The Kruskal rank of a sequence of elements in a vector space is an integer measuring the degree of linear dependency between its elements in a finer way than its rank. In \cite{KRUSKAL197795}, Kruskal introduced this notion to describe a property sufficient to ensure that a tensor has a single minimal rank form up to permutation of the terms of the sum. As we will see in Section \ref{sec:multilinearfromsgraph}, the number of such decompositions of a tensor will be involved in the distance-regularity assessment of the graphs that will be considered, and thus the uniqueness criterion (Theorem~\ref{thm:kruskaluniquenesstheorem}) introduced by the author will be of use.

\begin{Def}[Kruskal's rank] 
    Let $V$ be a vector space and let $(v_1,\dots,v_s)$ be a sequence of $V$. We define the {Kruskal rank} or the k-rank of $(v_1,\dots,v_s)$, denoted by $\krank(v_1,\dots,v_s)$, as the integer
    \[
        \krank(v_1,\dots,v_s) = \max \{ r \in \llbracket 1,s \rrbracket \ : \ \forall J \subseteq \llbracket 1,s\rrbracket, |J| = r, (v_j)_{j \in J} \text{ is linearly independent}\}.
    \]
\end{Def}

We can now state Kruskal's uniqueness theorem for tensors over a finite field, which is also valid for any field.
\begin{Thm}[{Kruskal's uniqueness theorem}, \cite{KRUSKAL197795,RHODES20101818ConciseProofKruskalThmDecomp}] \label{thm:kruskaluniquenesstheorem}
    Assume $m\geq 3$. Let $t \geq 2$ be an integer. For each $j \in \llbracket 1,m\rrbracket$, let $(v_{i,j})_{i \in \llbracket 1,t\rrbracket}$ be a sequence of $\F_{q}^{n_j}$. Assume that
        \[
        2t \leq 1+\sum_{j= 1}^m (\krank(v_{1,j},\dots,v_{t,j}) -1).
    \]
    Then the tensor
    $
        v = \sum_{i = 1}^t v_{i,1} \otimes \cdots \otimes v_{i,m}
   $ 
    has a unique minimal rank form, up to reordering the terms of the sum.
\end{Thm}

We now consider the action of $\GL_{\mathbf{n}}(\F_q)$ on the additive group $\F_q^{\mathbf{n}}$. We will use in the rest of this paper the generalisation of this action, that is the monoid action of $\prod_{j = 1}^m \F_q^{n_j\times n_j}$ on $\F_q^{\mathbf n}$. Such actions on a given tensor correspond to changes of basis of its slice spaces; see Definition \ref{def:modeactionslicespace}. If we identify an $m$-order tensor with an $m$-linear form, then such actions correspond also to a change of basis of the domain. For further reading; see \cite{Burgisser1997ch14}.

\begin{Rq}
    In the following, for convenience, we will identify the tensor space 
    $\F_q^{\mathbf n}$ with the corresponding tensor-space of lower order when one or more of the coefficients of $\mathbf n$ is one, namely we will use the identification given by the isomorphism
    \[
        \phi : \F_q^{n_1\times \dots \times n_{j-1}\times 1\times n_{j+1}\times \dots \times n_m} \to \F_q^{n_1\times \dots \times n_{j-1}\times n_{j+1}\times \dots \times n_m}
    \]
    defined for each $(u_1,\dots,u_{j-1},u_{j+1},\dots,u_m) \in \prod_{i \in \llbracket 1,m\rrbracket, i \neq j} \F_q^{n_i}$ and for each $u \in \F_q$ by
    \[
        \phi\left( \left(\bigotimes_{i = 1}^{j-1} u_i\right) \otimes u \otimes \left( \bigotimes_{i = j+1}^m u_i\right) \right) = u \bigotimes_{i \in \llbracket 1,m\rrbracket \backslash \{j\}} u_j.
    \]
    Likewise, we identify $\F_q^{\mathbf n}$ with $\F_q$ if $\mathbf{n} = (1,\dots,1)$.
\end{Rq}
\begin{Def}
    Let $\mathbf n = (n_1,\dots,n_m)$ and $\mathbf n' = (n_1',\dots,n_m')$ be tuples of positive integers of same length. The map
    $
      :  \F_q^{n_1'\times n_1} \times \cdots \times \F_q^{n_m'\times n_m} \times \F_q^{\mathbf n} \to \F_q^{\mathbf n'}
    $
   is defined by
    \[
        (M_1,\dots,M_m,u_1\otimes \cdots \otimes u_m ) \mapsto (M_1,\dots,M_m) \cdot (u_1\otimes \cdots \otimes u_m) := (u_1M_1^\top) \otimes \cdots \otimes (u_mM_m^\top).
    \]
    for each $(M_1,\dots,M_m) \in \prod_{i = 1}^m \F_q^{n_i'\times n_i}$ and each $ u_1\otimes \cdots \otimes u_m \in \S_{\mathbf n}$.   
\end{Def}
\begin{Rq}\label{rk:LinearityOfActionOnTensors}
    The map above is an $(m+1)$-linear map. Moreover, for each $\mathbf n$, $\mathbf n'$, and $\mathbf n''$, tuples of positive integers of length $m$, for each $(M_1,\dots,M_m) \in \prod_{i = 1}^m \F_q^{n_i'\times n_i}$, for each $(M_1',\dots,M_m') \in \prod_{i = 1}^m \F_q^{n_i''\times n_i'}$ and for each $x \in \F_q^{\mathbf n}$, we have
    \[
        (M_1',\dots,M_m') \cdot ((M_1,\dots,M_m) \cdot x) = (M_1'M_1,\dots,M_m'M_m) \cdot x.
    \]
\end{Rq}

\begin{Def}\label{def:modeactionslicespace}
    Let $\mathbf{n} = (n_1,\dots,n_m)$ be a sequence of positive integers. Let $j \in \llbracket 1,m \rrbracket$. The linear map $m_j : \F_q^{n_j} \times \F_q^{\mathbf{n}} \to \F_q^{n_1\times \dots \times n_{j-1}\times n_{j+1}\times \dots \times n_m}$ defined for each $u \in\F_q^{n_j}$ and each $x\in \F_q^{\mathbf{n}}$ by
    \[
        \modemult_j(u,x) = (I_{n_1},\dots,I_{n_{j-1}},u,I_{n_{j+1}},\dots,I_{n_m})\cdot x
    \]
    is called the $j$-th contraction map and we define the $j$-th slice space of $x$ by $\slicesp_j(x) = \im \modemult_j(\cdot,x)$. We then say that $x$ is a generator tensor of $\slicesp_j(x)$.
\end{Def}

We state the following results, which are straightforward to verify.

\begin{Prop}\label{prop:propertiestensorproddotprod}
    Let $\mathbf{n} = (n_1,\dots,n_m)$ and $\mathbf n' = (n_1',\dots,n_m')$ be two sequences of positive integers.
    \begin{enumerate}
        \item For each $u_1 \otimes \cdots \otimes u_m \in \S_{\mathbf n}$ and each $x \in \F_q^{\mathbf n}$ we have 
        $
            \langle u_1 \otimes \cdots \otimes u_m | x \rangle = (u_1,\dots,u_m) \cdot x.
        $
        \item For each $x \in \F_q^{\mathbf n}$, $y \in \F_q^{\mathbf n'}$, and $M\in \prod_{i=1}^m \F_q^{n_i'\times n_i}$, we have
        $
            \langle (M_1,\dots,M_m) \cdot x | y \rangle = \langle x | (M_1^\top,\dots,M_m^\top) \cdot y \rangle.
        $
        \item The group $\GL_{\mathbf{n}}(\F_q)$ acts on $\F_q^{\mathbf{n}}$ with the map $\GL_{\mathbf{n}}(\F_q) \to \Endo_{\F_q}(\F_q^{\mathbf{n}}), P \mapsto (x \in \F_q^{\mathbf n} \mapsto P \cdot x)$.
        \item The map
        $
            \F_q^{\mathbf{n}} \to \Hom_{\F_p}(\F_q^{\mathbf{n}},\F_p), x \mapsto (y \in \F_q^{\mathbf{n}} \mapsto \Tr(\langle x | y \rangle))
        $ for all $x \in \F_q^{\mathbf{n}}$
        is an isomorphism of $\F_p$-vector spaces.

    \end{enumerate}
\end{Prop}

As stated before, the induced action of $\GL_{n_j}(\F_q)$ on the space of tensors $\F_q^{\mathbf n}$ corresponds to a change of basis of the $j$-th slice space of a tensor. That is, if $x \in \F_q^{\mathbf n}$ and $P_j \in \GL_{n_j}(\F_q)$, then the sequence of tensors $(\modemult_j(e_i,(I_{n_1},\dots,I_{n_{j-1}},P_j,I_{n_{j+1}},\dots,I_{n_m})\cdot x))_{i \in \llbracket 1,n_j \rrbracket }$ is a linear combination of $(\modemult_j(e_i,x))_{i \in [n_j]}$ whose coefficients are determined by $P_j$. We also obtain that for each $x,y \in\F_q^{\mathbf n}$, we have $\slicesp_j(x) = \slicesp_j(y)$ if and only if there exists $P_j \in \GL_{n_j}(\F_q)$ such that $y = (I_{n_1},\dots,I_{n_{j-1}},P_j,I_{n_{j+1}},\dots,I_{n_m})\cdot x$.

It has been noted in \cite{Burgisser1997ch14} that the slice space of a tensor characterises its tensor rank. More precisely, the rank of a tensor $x$ is the least number of elements of rank one whose linear span contains a slice space of $x$. A useful consequence of this fact can be stated as follows.

\begin{Prop}[{\cite[14.45]{Burgisser1997ch14}}]\label{prop:slicespdimlowerboundstrank}
        Let $x \in \F_q^{\mathbf{n}}$ be a tensor and let $j \in \llbracket 1,m\rrbracket$. Then $\dim_{\F_q}\slicesp_j(x) \leq \rank(x)$.
\end{Prop}

\subsection{The degree one complex characters of the group of tensors}

In this paper, we will use complex (additive) characters of degree one of vector spaces over finite fields as groups, that is, group morphisms $(V,+) \to (\C^\times ,\times)$ with $V$ such a vector space. Let us recall the known results on irreducible characters of a vector-space over a finite field. We refer the reader to \cite[Chapter 5]{LidlNiederreiter1996FiniteFields} for more details. The additive group $\F_q$ has a cyclic character group generated by 
\[
    \F_p \to \C^*, i \mapsto \exp\left(\frac{2i\pi}{p}\right).
\]
The irreducible characters of the group $\F_q$, as a direct sum of groups isomorphic to $\F_p$, can be expressed with this fundamental character, that is any irreducible character of $\F_q$ can be expressed as $x\mapsto \chi(xy)$ for a certain $y\in \F_q$, where
\begin{equation}
    \chi : \F_q \to \C^*, x \mapsto \exp\left(\frac{2\Tr(x)\pi}{p}\right),
    \label{eq:fundamentalcharacter}
\end{equation}
see \cite[Theorem 5.7]{LidlNiederreiter1996FiniteFields}. Likewise, for each positive integer $n$, the irreducible characters of the group of vectors $\F_q^n$ are given for each $x \in \F_q^n$ by $\F_q \to \C^\times, y \mapsto \chi(x\cdot y)$, where $x \cdot y$ is the canonical dot-product. In particular, the characters of the group of $m$-order tensors $\F_q^{\mathbf n}$, isomorphic to $\F_q^{n_1\cdots n_m}$, can then be stated using the isomorphism described in Proposition~\ref{prop:propertiestensorproddotprod}. Namely, the irreducible characters of $\F_q^{\mathbf{n}}$ are 
\begin{equation}
    \chi(\langle x | \cdot \rangle ) : \F_q^{\mathbf{n}} \to \C^\times, y \mapsto \chi(\langle x | y \rangle) = \exp\left(\frac{2\Tr(\langle x | y \rangle)\pi}{p} \right),
    \label{eq:irreduciblecharactersofFqtensors}
\end{equation}
where $x$ ranges over $\F_q^{\mathbf{n}}$.

As seen in \cite{Lovasz75SpectraOfGraphsWithTransitiveGroups}, the eigenvalues of a Cayley graph are determined by the irreducible characters of the group and the generating set of the graph. We will hence use the results above to compute the eigenvalues of the multilinear forms graph.

\subsection{Graph theory fundamentals}
\label{subsec:graphtheoryfundamentals}
Let us introduce the graph theory notions relevant to our study. We will confine ourselves to non-oriented finite graphs in the entirety of this paper. We refer the reader to \cite{Biggs1974AlgebraicGraphTheory,Godsil1993AlgebraicCombinatorics,GodsilRoyle2001Ch3TransitiveGraphs} for more details.
A {graph} is a pair $\Gamma = (V,E)$ with $E\subseteq \{\{v_1,v_2\}\ : \ v_1,v_2 \in V\}$, where $V$ is a set called the set of vertices of $\Gamma$ and $E$ is called its set of edges. We assume that $V$ is finite. The adjacency matrix of the graph $\Gamma$ is the symmetric matrix $A \in \{0,1\}^{V\times V}$ defined for each $u,v \in V$ by $A_{u,v} = 1$ if ${\{u,v\} \in E}$ and $A_{u,v} = 0$ else. The eigenvalues of the graph are the eigenvalues of its adjacency matrix.
For each positive integer $i$, a {walk} of length $i$ on the graph $\Gamma$ is a sequence $(v_0,\dots,v_i) \in V^{i+1}$ such that $\{v_{k-1},v_{k}\} \in E$ for each $k \in \llbracket 1,i\rrbracket$. If $v_i = v_0$, we say that $(v_0,\dots,v_i)$ is a {closed walk}. If there exists a walk between each pair of vertices in $V$, we say that $\Gamma$ is {connected}. 
The minimum number of vertices in any walk between a pair of vertices defines a notion of distance between them. 
\begin{Def}
    Let $\Gamma$ be a connected finite graph. The function $d_\Gamma : V(\Gamma)^2 \to \N_0$ defined by 
    \begin{equation*}
        d_\Gamma(x,y) = \min\{ i \in \N_0 \ : \ \exists (v_0,\dots,v_i) \in V(\Gamma)^{i+1}, (v_0,v_i) = (x,y), \forall k \in \llbracket 1,i\rrbracket, \{v_{k-1},v_k\} \in E(\Gamma)\}
    \end{equation*}
    for each $(x,y) \in V(\Gamma)^2$ 
    is a distance function called the {geodesic distance} of $\Gamma$. 
\end{Def}

\begin{Def} 
    Let $\Gamma$ be a graph and let $k$ be a positive integer. We say that $\Gamma$ is {$k$-partially walk-regular} if, for each $\ell \leq k$, the number of closed walks of length $\ell$ starting from a certain point depends uniquely on $\ell$ (and not on the choice of the starting point). We say that $\Gamma$ is {walk-regular} if the graph is $k$-partially regular for any $k \in \N$.
\end{Def}

\begin{Def}
    Let $\Gamma = (V,E)$ be a connected finite graph. For each $x,y \in V$ and each integer $i,j,h \in \N$, we denote by $p^i_{j,h}(x,y)$ the integer
    \begin{equation*}
        p^i_{j,h}(x,y) = |\{z \in V \mid d_\Gamma(x,z) = j, d_\Gamma(y,z) = h\}|.
    \end{equation*}
    We say that $\Gamma$ is {$k$-partially distance-regular} if $p^i_{j,h}(x,y)$ depends only on $j,h$ below a certain threshold, more precisely if for each integer $i \in\N$ and each $x,x',y,y' \in V$ such that  $d_\Gamma(x,y) = d_\Gamma(x',y') = i$ we have
    \begin{equation*}
        \forall j,h\in \N_0, i \leq j+h \leq k : p^i_{j,h}(x,y) = p^i_{j,h}(x',y').
    \end{equation*} 
    We say that $\Gamma$ is {distance-regular} if it is $k$-partially distance-regular for any $k \in \N$.
\end{Def}

If a graph is distance-regular, its eigenvalues can be computed using the intersection numbers; see \cite[Section 4.1]{Brouwer1989TheoryOfDistanceRegularGraphs}. Additionally, for a graph $\Gamma = (V,E)$ of geodesic distance $d_\Gamma$, the binary relations $(\RR_i)_{i \geq 0}$ given by 
\[
    \RR_i = \{ (x,y) \in V^2 \ : \ d_\Gamma(x,y) = i \}
\]
form a symmetric association scheme if and only if the graph $\Gamma$ is distance-regular. For applications to coding theory (to obtain upper bounds on the size of a code), assessment of such properties is essential \cite{AbiadPetersRavagnani2025EigenvalueMethod,Delsarte1973AlgebraicAproachAssociationSchemes}. Here, we will study the following parameters of the graphs.

\begin{Def}
    Let $\Gamma = (V,E)$ be a graph and let $k$ be a positive integer. The $k$-th independence number is $\alpha_k(\Gamma) = \max\left\{|\VVV| \ : \ \VVV \subseteq V , \forall v,v' \in \VVV, v\neq v' \implies d_{\Gamma}(v,v') > k\right\}$.
\end{Def}

This notion can be used in the context of coding theory as follows. Regarding the study of codes $C$ subsets of a space $G$ with a metric $\delta : G \times G \to \N_0$, assume that there is a graph $\Gamma = (V,E)$ such that $V = G$ and such that the geodesic distance of the graph $G$ coincides with the metric $\delta$. Then the maximum cardinality of a code $C \subseteq G$ of minimum distance $d$ is precisely $\alpha_{d-1}(\Gamma)$, the $(d-1)$-th independence number of the graph. This perspective for the study of bounds on the cardinality of graphs has been used in \cite{AbiadAlfranoRavagnani2025EVboundsAlternatingRMC,AbiadKhramovaAntoniaRavagnani2024EigenvalueBoundSumRankMetric,AbiadNeriReijnders2025EVboundsLeeCodes,AbiadPetersRavagnani2025EigenvalueMethod}. To obtain bounds, we will consider the following bounds on the independence number of graphs.

\begin{Thm}[{\cite{AbiadCoutinhoFiol2019kindependancenumber}}]\label{thm:inertiaratio}
    Let $\Gamma = (V,E)$ be a graph with $N$ vertices and eigenvalues $\lambda_1,\dots,\lambda_{N}$ counted with multiplicities in decreasing order. Let $A$ be the adjacency matrix of $\Gamma$, let $k$ be an integer and let $p(X) \in \R_k[X]$ be a polynomial. Denote by $W(p) = \max \{ (p(A))_{u,u} : u \in V\}$, by $w(p) = \min \{ (p(A))_{u,u} : u \in V\}$, and by  $\lambda(p) = \min \{ (p(\lambda_i)) : i \in \llbracket 2,N\rrbracket\}$.
    \begin{enumerate}[label = (\roman*)]
        \item \emph{Inertia-type bound:} We have 
        \begin{equation}
            \label{eq:intertiatype}
            \alpha_k \leq \min\{ |\{ i \in \llbracket 1,N\rrbracket \ : \ p(\lambda_i) \geq w(p)\} | , |\{ i \in \llbracket 1,N\rrbracket \ : \ p(\lambda_i) \leq W(p)\} |.
        \end{equation}
        \item \emph{Ratio-type bound:} If $\Gamma$ is a regular graph and if $p(\lambda_1) > \lambda(p)$, then we have 
        \begin{equation}
            \label{eq:ratiotype}
            \alpha_k \leq N\frac{W(p) - \lambda(p)}{p(\lambda_1) - \lambda(p)}.
        \end{equation}
    \end{enumerate}
\end{Thm}

Without any assumptions on the graph, there exists a mixed integer linear-programming program that computes the best possible inertia-type bound on the $k$-independence number of a graph; see \cite{AbiadCoutinhoFiol2019kindependancenumber}. If the graph is $k$-partially walk-regular, then there exists an improvement of the former algorithm as well as a linear-programming algorithm that computes the best possible ratio-type bound on the $k$-independence number of a graph; see \cite{AbiadCoutinhoFiol2019kindependancenumber,FIOL20201}. We will apply these last two algorithms to the multilinear forms graph in Section~\ref{sec:bounds}. For the $2$- and the $3$-independence number of a graph, the best possible ratio-type bound can also be stated as follows. 

\begin{Thm}[{\cite[Corollary 3.3]{AbiadCoutinhoFiol2019kindependancenumber},\cite[Theorem 11]{KaviNewman2023OptimalBound3IndepNumber}}]\label{thm:ratiobound23}
    Let $\Gamma$ be a regular graph with $N$ vertices and eigenvalues $\theta_0,\dots,\theta_r$ sorted in strictly decreasing order with respective multiplicities $m_0,\dots,m_r$. 
    \begin{enumerate}[label = (\roman*)]
        \item Assume $r \geq 2$. With $i = \min\{j \in \llbracket 1,r\rrbracket \ : \ \theta_j \leq -1\}$, we have 
        \begin{equation}
            \label{eq:ratiotype2}
            \alpha_2 \leq N\frac{\theta_0 + \theta_i\theta_{i-1}}{(\theta_0 - \theta_i)(\theta_0 - \theta_{i-1})},
        \end{equation}
        and this is the best possible bound that can be obtained by a choice of polynomial in (\ref{eq:ratiotype}).
        \item Assume $r \geq 3$. Denote by $A$ the adjacency matrix of $\Gamma$ and let $\Delta = \max\left\{ (A^3)_{u,u} : u \in V\right\}$. With $s = \max\left\{j \in \llbracket 0,r-1\rrbracket : \theta_j \geq -\frac{\theta_0^2 + \theta_0\theta_r - \Delta}{\theta_0(\theta_r+1)}\right\}$, we have 
        \begin{equation}
            \label{eq:ratiotype3}
            \alpha_3 \leq N\frac{\Delta - \theta_0(\theta_s + \theta_{s+1}+\theta_r) - \theta_s\theta_{s+1}\theta_r}{(\theta_0 - \theta_s)(\theta_0 - \theta_{s+1})(\theta_0 - \theta_r)},
        \end{equation}
        and this is the best possible bound that can be obtained by a choice of polynomial in (\ref{eq:ratiotype}).
    \end{enumerate}
\end{Thm}

In Section \ref{sec:eigenvalues233}, we will compute upper bounds on the cardinality of a tensor code of size $3\times 2 \times 2$ over any field $\F_q$ of minimum distance $3$ or $4$ using these self-contained expressions of the ratio-type bounds for the $2$- and $3$-independence number.

We now focus on a particular type of graph constructed from the structure of a given group. The multilinear forms graph introduced and studied in Section \ref{sec:multilinearfromsgraph} will be a certain Cayley graph.

\begin{Def}
    Let $G$ be a finite abelian group and let $S$ be a subset of $G$. The {Cayley graph} of $G$ with generating set $S$ is the graph denoted by $\Cay(G,S) = (V,E)$ such that $V= G$ and $E = \{\{g,g+s\} \ : \ g \in G, s\in S \}$.
\end{Def}

We briefly recall some properties of Cayley graphs; the reader is referred to \cite{Cayley1878GraphicalRepresentationGroups,GodsilRoyle2001Ch3TransitiveGraphs,LiuZhou2022EigenvaluesOfCayleyGraphs} for more details. Let $G$ be a finite abelian group and $S$ a subset of $G$. The Cayley graph $\Cay(G,S)$ is a regular graph of degree $|S|$. More precisely, $\Cay(G,S)$ is vertex-transitive, i.e. for any given pair of vertices $u,v$, there exists a graph automorphism that maps $u$ to $v$. This in turn implies that $\Cay(G,S)$ is walk-regular. The graph is connected if and only if $S$ generates the group $G$. Since the 
$\Cay(G,S)$ is undirected, its adjacency matrix is a real symmetric matrix and hence is diagonalisable. Moreover, its trace is exactly 
$|G|$ if the identity element of $G$ is contained in $S$ and is $0$ otherwise.

\begin{Thm}[{\cite[Corollary 3.2]{Babai79SpectraofCayleyGraphs}, \cite{Lovasz75SpectraOfGraphsWithTransitiveGroups}}]\label{thm:eigenvaluesofaCayleyGraph}
    Let $G$ be a finite abelian group of order $N$ with irreducible characters $\chi_1,\dots,\chi_N$ over $\C$. Let $S$ be a subset of $G$. Then the eigenvalues of $\Cay(G,S)$ are exactly $\lambda_1,\dots,\lambda_N$ with
    \[
        \lambda_i = \sum_{s\in S} \chi_i(s)
    \]
    for each $i \in \llbracket 1,N\rrbracket$.
\end{Thm}
For each $i \in \llbracket 1,N\rrbracket$, the vector $(\chi_i(g))_{g \in G}$ is an eigenvector associated to the eigenvalue $\lambda_i$. Consequently, one can use the following well-known fact for any polynomial evaluation of the adjacency matrix of the graph.

\begin{Corol}\label{corol:polynomialexpressionofadjacencymatrixcayley}
    Let $G$ be a finite abelian group of order $N$, and let $S$ be a subset of $G$. Let $f(X) \in \C[X]$ be a polynomial and let $A \in \{0,1\}^{N \times N}$ be the adjacency matrix of $\Cay(G,S)$. Let $\theta_0,\dots,\theta_r$ be the distinct eigenvalues of $A$ with respective multiplicities $m_0,\dots,m_r$. Then, for each $u \in G$, we have  
    \[
        (f(A))_{u,u} = \frac{1}{N}\sum_{i=0}^r m_if(\theta_i).
    \]
\end{Corol}

\subsection{Rank-metric codes and tensor codes}
\label{sec:rankmetriccodesandtensorcodes}
In this section, we will recall the notions of a rank-metric code and tensor code. We will also recall the definitions and properties of the mathematical objects that will be relevant in the rest of the paper. We refer the reader to \cite{ConciseCodingTheoryRankMetricCodes} for more information on rank-metric codes, and \cite{ROTHTensorCodesForRankMetric,RothMaximumArrayCodes} for tensor codes. 

\begin{Def}
    A (linear) rank-metric code over $\F_q$ is a subspace of the matrix space $\F_q^{n_1\times n_2}$. We say that a $k$-dimensional subspace $\CCC$ of $\F_q^{n_1 \times n_2}$ is an $[n_1\times n_2,k]_q$ rank-metric code. If $k\geq 1$, we define its minimum rank distance to be $d = \min_{(C,C') \in \CCC^2, C \neq C'} \rank(C-C') = \min_{C \in \CCC\backslash\{0\}} \rank(C)$, and we say that $\CCC$ is an $[n_1 \times n_2,k,d]_q$ code. 
\end{Def}

Moreover, if $\CCC$ is an $[n_1\times n_2,k]_q$ rank-metric code, we have the Singleton bound 
$$k \leq \max(n_1,n_2)(\min(n_1,n_2) - d +1),$$ 
and we say that $\CCC$ is a maximum rank distance (MRD) code if this inequality is an equality. 

The rank induces a metric $d_{\rank} : \F_{q}^{n_1 \times n_2} \times \F_{q}^{n_1 \times n_2} \to \N_0$ defined by $d_{\rank}(C,C') = \rank(C - C')$ for all $C,C' \in \F_q^{n_1 \times n_2}$. More generally, for each tuple of positive integers $\mathbf{n} = (n_1,\dots,n_m)$ with $m \geq 1$, the tensor rank induces a metric on $\F_q^{\mathbf n}$ the space of $n_1 \times \cdots \times n_m$ tensors, which is given by 
\[
    d_{\rank} : \F_q^{\mathbf n} \times \F_q^{\mathbf n} \to \N_0, (x,y) \mapsto \rank(x-y).
\]
This allows one to introduce tensor codes as a generalisation of rank-metric codes.

\begin{Def}
    An $(\mathbf n,M)_q$ tensor code (or an $(n_1 \times \cdots \times n_m,M)_q$ tensor code), is a subset $\CCC$ of $\F_q^{\mathbf n}$ such that $|\CCC| = M$. If $M>1$, we define its minimum rank to be
    \[
        d = \min\left\{ d_{\rank}(x,y) \ : \ x,y \in \CCC, x\neq y\right\},
    \]
    and we say that $\CCC$ is an $(\mathbf n,M,d)_q$ tensor code (or $(n_1 \times \cdots \times n_m,M,d)_q$ tensor code).
    If $\CCC$ is $\F_q$-linear of dimension $k$, we say that it is an $[\mathbf n,k,d]_q$ tensor code.
\end{Def}

\begin{Def} 
    Let $\CCC$ be a subspace of $\F_q^{m\times n}$. For each $r \in \llbracket 0,\min(m,n)\rrbracket$, we denote by $A_r(\CCC)$ the number of codewords of $\CCC$ with rank $r$, i.e.,
    \[
        A_r(\CCC) = |\{M \in \CCC \ : \ \rank(M) = r\}|.
    \]
    In general, if $\CCC$ is a subset of $\F_q^{\mathbf n}$, for each $r \in \N_0$, we denote by $A_r(\CCC)$ the number of codewords of $\CCC$ with rank $r$, i.e.,
    \[
        A_r(\CCC) = |\{M \in \CCC \ : \ \rank(M) = r\}|,
    \]
    and we say that the sequence of integers $(A_0(\CCC),A_1(\CCC),\dots)$ is the weight-distribution of the code $\CCC$.
\end{Def}

\begin{Def}
    Let $\CCC$ be an $\F_q\text{-}[n_1\times n_2,k]$ rank-metric code. We then define the {rank-weight enumerator} of $\CCC$ to be the two variable polynomial 
    \[
        W_\CCC(X,Y) = \sum_{r = 0}^{\min(n_1,n_2)} A_r(\CCC) X^{\min(n_1,n_2)-r}Y^r.
    \]  
\end{Def}

In the rest of the paper, we will provide bounds on the cardinality of tensor codes from the study of the graphs that will be introduced in the following section. First we recall some known upper bounds. The two following statements are direct generalisations of bounds introduced by Roth in \cite{RothMaximumArrayCodes}. 
\begin{Thm}[Singleton-like bound, \cite{RothMaximumArrayCodes}]\label{thm:singletonlike}
    Let $\CCC$ be an $(\mathbf n,M,d)_q$ tensor code with $M>1$. Denote by $n_{\max} = \max\{n_i \ : \ i \in \llbracket 1,m\rrbracket\}$. Then we have 
    \begin{equation}
        \label{eq:singletonlike}
        \log_q(M) \leq n_{\max}\left(\frac{\prod_{i = 1}^m n_i}{n_{\max}} - d +1 \right).
    \end{equation}
\end{Thm}

The next statement is a small generalisation of \cite[Theorem 5]{RothMaximumArrayCodes}. We include a proof for the convenience of the reader.

\begin{Thm}[Improved Singleton-like bound, {\cite[Theorem 5]{RothMaximumArrayCodes}}]\label{thm:improvedsingletonlike}
    Let $n$ and $\ell$ be two positive integers such that $\ell \leq n$. Let $\CCC$ be an $(\ell \times n \times n,M,d)_q$ tensor code with $M > 1$. Let $\sigma = n - \left\lceil \sqrt{n^2 - d +1} \right\rceil$. Then we have 
    \begin{equation}
        \label{eq:improvedsingletonlike}
        \log_q(M) \leq n \left(n\ell - d + 1 - \sigma^2\right).
    \end{equation}
\end{Thm}
\begin{proof}
    Following \cite[Theorem 5]{RothMaximumArrayCodes}, we let $\tau = \left\lceil \sqrt{n^2 - d +1}\right\rceil^2 - (n^2 - d + 1)$ which is the integer satisfying $\tau \leq n$ and $2\sigma n + \tau - \sigma^2 = d - 1$. By Proposition~\ref{prop:howellbound}     we have 
    \[
        d \leq \ell n - \left\lfloor \frac{\ell^2}{4} \right\rfloor < \ell n + 1 - \frac{\ell^2}{4},
    \]
    from which we obtain $\sqrt{n^2 -d + 1} > n - \frac \ell 2 $. Since $\sigma < n - \sqrt{n^2 -d + 1}$, we have $\sigma^2 < n\ell - d +1$.
    Hence we have $\ell n - 2\sigma n - \tau = \ell n - d + 1 - \sigma^2 > 0$. Let $\III$ be the set of indexes of the last $\ell n - 2\sigma n - \tau$ 3-fibres of an $\ell \times n \times n$ tensor, i.e. $ \III = \llbracket 2\sigma +1, \ell\rrbracket \times \llbracket 1,n\rrbracket \backslash \{(2\sigma +1,1),\dots,(2\sigma + 1,\tau) \}$. Consider the epimorphism defined by
    \[ 
    \phi : \F_q^{\ell \times n \times n} \to  (\F_q^n)^{\III}, x \mapsto \left((x_{i,j,1},\dots,x_{i,j,n})\right)_{(i,j) \in \III} \;\text{ for all } x \in \F_q^{\ell \times n \times n}.     \]  
    Let $\psi : \F_{q}^{\ell \times n \times n} \to \F_{q}^{2\sigma \times n \times n} \times \F_q^{(\ell - 2\sigma)\times n \times n}$  be the isomorphism given for each $u \otimes v \otimes w \in \S_{\ell \times n \times n}$ by 
    \[
        \psi(u\otimes v \otimes w) = ((u_1,\dots,u_{2\sigma})\otimes v \otimes w\ ,\ (u_{2\sigma +1},\dots,u_\ell)\otimes v \otimes w),
    \]
    that is, the isomorphism that separates the first $2\sigma$ 1-slices from the last $(\ell -2\sigma)$ 1-slices of any tensor. Let $x \in \ker \phi$ and $(x^{(1)},x^{(2)}) =\psi(x)$. By Proposition~\ref{prop:howellbound}, we have 
    \begin{align*}
        \rank(x) \leq \rank(x^{(1)}) + \rank(x^{(2)}) \leq \left(2\sigma n - \left\lfloor \frac{2^2\sigma^2}{4}\right\rfloor\right) + \tau  < \mu.
    \end{align*}
    Therefore, the restriction map $\phi_{|\CCC} : \CCC \to (\F_q^{n})^{\III}$ is injective and we have $|\CCC| \leq |\F_q^{\III}|$ and in turn we have $\log_q(|\CCC|) \leq  {n(\ell n- 2\sigma n - \tau)} = n(\ell n -d +1-\sigma^2)$.
\end{proof}

    The improved Singleton-like bound (\ref{eq:improvedsingletonlike}) is a strict improvement on the Singleton-like bound (\ref{eq:singletonlike}) for $((\ell,n,n),M,d)_q$ tensor codes with $\ell \leq n$ if and only if $\sigma = n - \left\lceil \sqrt{n^2 - d +1} \right\rceil > 0$. This is equivalent to $d \geq 2n$. Similarly, by Proposition~\ref{prop:howellbound}, there exists no $((\ell,n,n),M,d)_{q}$ tensor code such that $d > n\ell - \left\lfloor \frac{\ell^2}{4} \right\rfloor$. Therefore, the bound (\ref{eq:improvedsingletonlike}) shows a strict improvement compared to (\ref{eq:singletonlike})     only for  $((\ell,n,n),M,d)_q$ tensor codes for which 
    \[
        2n \leq d \leq n\ell - \left\lfloor \frac{\ell^2}{4} \right\rfloor.
    \]
    This indicates that (\ref{eq:improvedsingletonlike}) sharpens the other bounds for tensors codes with $\ell$ large and large minimum distance. In the following parts of the paper, we will focus on the computation of graph eigenvalues for small parameters ($q$, $m$, and $\mathbf n$); see Table \ref{tab:RatioType2345}. For larger parameters, the computations are much less feasible. Hence, parameters for which (\ref{eq:improvedsingletonlike}) is a strict improvement of (\ref{eq:singletonlike}) do not often arise in our results.
    
    The sphere-packing bound is another type of generic bound on the size of a code of given minimum distance. The bound for tensor codes can be stated as follows.

\begin{Thm}[Sphere-packing bound  {\cite[Lemma 1]{ROTHTensorCodesForRankMetric}}] \label{thm:spherepackingbound}
    Let $\CCC$ be an $(\mathbf{n},M,d)_q$ tensor code. Then we have 
    \begin{equation}
        M \leq q^{n_1\cdots n_m} \left| \left\{x \in \F_q^{\mathbf{n}} \ : \ \rank(x) \leq \left\lfloor \frac{d-1}2 \right\rfloor\right\}\right|^{-1}.
        \label{eq:spherepacking}
    \end{equation}
\end{Thm}

Contrary to the Hamming-metric, the rank-metric or the sum-rank metric, the tensor rank distribution in the ambient space $\F_q^{\mathbf{n}}$ is not known in general. In particular, for $d$ large, one can only use lower-bounds for the ball size on the right-hand side of inequality (\ref{eq:spherepacking}). 

For minimum distance $3$ tensor codes, the value of this ball-size is known as the ball in question is $\S_{\mathbf n} \sqcup \{0\}$. Consequently, we have the following bound.

\begin{Corol}[Sphere-packing bound for $d = 3$ {\cite[Theorem 1]{ROTHTensorCodesForRankMetric}}] \label{corol:spherepackingboundwithd3}
    Let $\CCC$ be an $(\mathbf{n},M,3)_q$ tensor code. Then we have 
    \begin{equation}
    \label{eq:spherepacking3}
    M \leq \frac{q^{n_1\cdots n_m}}{ 1 + \frac{\prod_{j=1}^m (q^{n_j}-1)}{(q-1)^{m-1}}}.
    \end{equation}
\end{Corol}

For $3$-order tensors, the exact value of the ball-size of radius $2$ can be computed using the number of tensors of rank $2$ in $\F_q^{n_1\times n_2\times n_3}$ which is given in \cite[Lemma 5.1]{byrne2026decodingalgorithmstensorcodes}. Therefore, we have the following statement.

\begin{Corol}[Sphere-packing bound for $d = 5$ and $m = 3$] \label{corol:spherepackingboundwithd5m3}
    Let $\CCC$ be an $(n_1 \times n_2 \times n_3,M,5)_q$ tensor code. Then 
    \begin{equation}
    \label{eq:spherepacking5}
        M \leq  \frac{q^{n_1n_2n_3}}{ 1 + \frac{\prod_{j=1}^3 (q^{n_j}-1)}{(q-1)^{2}} + \Nnn_2},
    \end{equation}
    where $\Nnn_2$ the number of tensors of rank $2$ in $\F_q^{n_1 \times n_2 \times n_3}$ is given by 
    \[
        \Nnn_2 := \frac{qABC}{(q-1)^3(q^2-1)}\left(\frac{q^2(q+1)}{2}A'B'C' + (q-1)\left(A'B' + A'C' + B'C'\right)\right),
    \]
    and $A = q^{n_1}-1$, $B = q^{n_2}-1$, $C = q^{n_3}-1$, $A' = q^{n_1-1}-1$, $B' = q^{n_2-1}-1$, and $C' = q^{n_3-1}-1$.
\end{Corol}

In Table \ref{tab:spherepacking}, for different values of $q$ and $\mathbf{n}$, we compute the number of tensors in $\F_q^{\mathbf n}$ of rank $1$ and $2$, as well as the sphere-packing bound for $(\mathbf n,M,d)_q$ tensor codes of distance $d=3$ and $d=5$. For $3$-order tensors, the number of elements of tensor rank $2$ and the value of the sphere-packing bound are computed using Corollary~\ref{corol:spherepackingboundwithd5m3}. For the $4$-order tensors cases in the table, the number of elements of rank $2$ is computed by brute-force using MAGMA, and we apply Theorem~\ref{thm:spherepackingbound} with this value to compute the sphere-packing bound. We write `n/a' in the table to indicate that there exists no code with the corresponding parameters, and so the bounds are not applicable.

\begin{table}[ht!]   \centering
\renewcommand\arraystretch{1.06}
    \begin{tabular}{|c|c||c|c||c|c|}\hline
$q$&
$\mathbf{n}$&
\!\!\!\!\!\begin{minipage}{3cm} \centering\phantom{x}\\[-0.2cm] Number of\\rank-$1$ tensors\\[-0.2cm]\phantom{x}\end{minipage}\!\!\!\!\!&
\!\!\!\!\!\begin{minipage}{3cm} \centering\phantom{x}\\[-0.2cm] Number of\\rank-$2$ tensors\\[-0.2cm]\phantom{x}\end{minipage}\!\!\!\!\!&
\!\!\!\!\!\begin{minipage}{3cm} \centering\phantom{x}\\[-0.2cm] Sphere-packing \\bound for $d=3$\\[-0.2cm]\phantom{x}\end{minipage}\!\!\!\!\!&
\!\!\!\!\!\begin{minipage}{3cm}\centering\phantom{x}\\[-0.2cm]Sphere-packing\\bound for $d=5$\\[-0.2cm]\phantom{x}
\end{minipage}\!\!\!\!\!\\\hline\hline
2&(2,2,2) & 27 & 162 & 3.17 & {\footnotesize n/a}\\\hline
2&(3,2,2) & 63 & 1,050 & 6.00 & {\footnotesize n/a} \\\hline
2&(4,2,2) & 135 & 5,130 & 8.91 & {\footnotesize n/a}\\\hline
2&(3,3,2) & 147 & 6,762 & 10.79 & 5.21\\\hline
2&(4,3,2) & 315 & 32,970 & 15.70 & 8.98\\\hline
2&(5,3,2) & 651 & 144,522 & 20.66 & 12.86\\\hline
2&(6,3,2) & 1,323 & 604,170 & 25.63 & 16.80\\\hline
2&(4,4,2) & 675 & 160,649 & 22.60 & 14.70\\\hline
2&(5,4,2) & 1,395 & 704,009 & 29.56 & 20.58\\\hline
2&(6,4,2) & 2,835 & 2,942,729 & 36.53 & 26.51\\\hline
2&(7,4,2) & 5,715 & 12,028,169 & 43.52 & 32.48\\\hline
2&(3,3,3) & 343 & 43,218 & 18.58 & 11.59\\\hline
2&(4,3,3) & 735 & 210,210 & 26.48 & 18.32\\\hline
2&(5,3,3) & 1,519 & 920,514 & 34.43 & 25.19\\\hline
2&(6,3,3) & 3,087 & 3,846,402 & 42.41 & 32.13\\\hline
2&(7,3,3) & 6,223 & 15,719,298 & 50.40 & 39.10\\\hline
2&(2,2,2,2) & 81 & 2,268 & 9.65 & 4.76 \\\hline
2&(3,2,2,2) & 189 & 13,608 & 16.43 & 10.25\\\hline
2&(4,2,2,2) & 405 & 64,800 & 23.34 & 16.01\\\hline
2&(5,2,2,2) & 837 & 281,232 & 30.29 & 21.90\\\hline
3&(2,2,2) & 64 & 4,032 & 4.20 & {\footnotesize n/a}\\\hline
3&(3,2,2) & 208 & 50,544 & 7.14 & {\footnotesize n/a}\\\hline
3&(4,2,2) & 640 & 501,120 & 10.12 & {\footnotesize n/a}\\\hline
3&(3,3,2) & 676 & 632,736 & 12.07 & 5.84\\\hline
3&(4,3,2) & 2,080 & 6,271,200 & 17.05 & 9.76\\\hline
3&(5,3,2) & 6,292 & 58,213,584 & 22.04 & 13.73\\\hline
3&(6,3,2) & 18,928 & 529,283,664 & 27.04 & 17.72\\\hline
4&(2,2,2) & 125 & 38,699 & 4.52 & {\footnotesize n/a}\\\hline
4&(3,2,2) & 525 & 797,580 & 7.49 & {\footnotesize n/a}\\\hline
4&(4,2,2) & 2,125 & 13,509,900 & 10.48 & {\footnotesize n/a}\\\hline
4&(3,3,2) & 2,205 & 16,431,660 & 12.45 & 6.02\\\hline
4&(4,3,2) & 8,925 & 278,310,060 & 17.44 & 9.98\\\hline
4&(5,3,2) & 35,805 & 4,515,296,940 & 22.44 & 13.97\\\hline
4&(6,3,2) & 143,325 & 72,494,816,940 & 27.44 & 17.97\\\hline
5&(2,2,2) & 216 & 224,640 & 4.66 & {\footnotesize n/a}\\\hline
5&(3,2,2) & 1,116 & 6,889,440 & 7.64 & {\footnotesize n/a}\\\hline
5&(4,2,2) & 5,616 & 178,813,440 & 10.64 & {\footnotesize n/a}\\\hline
5&(3,3,2) & 5,766 & 211,266,240 & 12.62 & 6.09\\\hline
5&(4,3,2) & 29,016 & 5,483,250,240 & 17.62 & 10.07\\\hline
5&(5,3,2) & 145,266 & 138,095,670,240 & 22.62 & 14.07\\\hline
\end{tabular}
\renewcommand\arraystretch{1}
    \caption{Number of tensors of rank $1$, number of tensor of rank $2$, and upper bounds on $\log_q(M)$ for $(\mathbf n,M,d)_{q}$ tensor codes with $d \in \{3,5\}$, rounded to 2 decimal places.}
    \label{tab:spherepacking}
\end{table}

In Table \ref{tab:bestbetween}, for different values of $q$ and $\mathbf n$, we compare the previously mentioned upper bounds on $M$ for $(\mathbf n,k,d)_q$ tensor codes, that is, the Singleton-like bound (\ref{eq:singletonlike}), its improvement (\ref{eq:improvedsingletonlike}) as well as the sphere-packing bounds for $d=3$ or $d=5$ in Table \ref{tab:spherepacking}. We use the colour \legendbox{lightgray} to indicate that (\ref{eq:improvedsingletonlike}) is the best-bound out the three, colour \legendbox{gray} to indicate that the sphere-packing bound is, and by the colour \legendbox{white} to indicate that (\ref{eq:singletonlike}) is. We use again write `n/a' in the table to indicate that there exists no code with the corresponding parameters. 

\begin{table}[ht!]   \centering
\renewcommand\arraystretch{1.06}
    \begin{tabular}{|c|c||c|c|c|c|c|c|c|}\hline
        $q$&$\mathbf{n}$&$d=3$&$d=4$&$d=5$&$d=6$&$d=7$&$d=8$&$d=9$\\\hline\hline
2 & (2,2,2) & \cellcolor{gray}3.17 & {\footnotesize n/a} & {\footnotesize n/a} & {\footnotesize n/a} & {\footnotesize n/a} & {\footnotesize n/a} & {\footnotesize n/a}\\\hline
2 & (3,2,2) & 6.00 & 3.00 & {\footnotesize n/a} & {\footnotesize n/a} & {\footnotesize n/a} & {\footnotesize n/a} & {\footnotesize n/a}\\\hline
2 & (4,2,2) & 8.00 & 4.00 & {\footnotesize n/a} & {\footnotesize n/a} & {\footnotesize n/a} & {\footnotesize n/a} & {\footnotesize n/a}\\\hline
2 & (3,3,2) & \cellcolor{gray}10.79 & 9.00 & \cellcolor{gray}5.21 & {\footnotesize n/a} & {\footnotesize n/a} & {\footnotesize n/a} & {\footnotesize n/a}\\\hline
2 & (4,3,2) & \cellcolor{gray}15.70 & 12.00 & 8.00 & 4.00 & {\footnotesize n/a} & {\footnotesize n/a} & {\footnotesize n/a}\\\hline
2 & (5,3,2) & 20.00 & 15.00 & 10.00 & 5.00 & {\footnotesize n/a} & {\footnotesize n/a} & {\footnotesize n/a}\\\hline
2 & (6,3,2) & 24.00 & 18.00 & 12.00 & 6.00 & {\footnotesize n/a} & {\footnotesize n/a} & {\footnotesize n/a}\\\hline
2 & (4,4,2) & \cellcolor{gray}22.60 & 20.00 & \cellcolor{gray}14.70 & 12.00 & 8.00 & {\footnotesize n/a} & {\footnotesize n/a}\\\hline
2 & (5,4,2) & \cellcolor{gray}29.56 & 25.00 & 20.00 & 15.00 & 10.00 & 5.00 & {\footnotesize n/a}\\\hline
2 & (6,4,2) & 36.00 & 30.00 & 24.00 & 18.00 & 12.00 & 6.00 & {\footnotesize n/a}\\\hline
2 & (7,4,2) & 42.00 & 35.00 & 28.00 & 21.00 & 14.00 & 7.00 & {\footnotesize n/a}\\\hline
2 & (3,3,3) & \cellcolor{gray}18.58 & 18.00 & \cellcolor{gray}11.59 & \cellcolor{lightgray}9.00 & \cellcolor{lightgray}6.00 & {\footnotesize n/a} & {\footnotesize n/a}\\\hline
2 & (4,3,3) & \cellcolor{gray}26.48 & 24.00 & \cellcolor{gray}18.32 & 16.00 & 12.00 & 8.00 & 4.00\\\hline
2 & (5,3,3) & \cellcolor{gray}34.43 & 30.00 & 25.00 & 20.00 & 15.00 & 10.00 & 5.00\\\hline
2 & (6,3,3) & 42.00 & 36.00 & 30.00 & 24.00 & 18.00 & 12.00 & 6.00\\\hline
2 & (7,3,3) & 49.00 & 42.00 & 35.00 & 28.00 & 21.00 & 14.00 & 7.00\\\hline
2 & (2,2,2,2) & \cellcolor{gray}9.65 & \cellcolor{gray}9.65 & \cellcolor{gray}4.76 & \cellcolor{gray}4.76 & 4.00 & 2.00 & {\footnotesize n/a}\\\hline
2 & (3,2,2,2) & \cellcolor{gray}16.43 & 15.00 & \cellcolor{gray}10.26 & 9.00 & 6.00 & 3.00 & {\footnotesize n/a}\\\hline
2 & (4,2,2,2) & \cellcolor{gray}23.34 & 20.00 & 16.00 & 12.00 & 8.00 & 4.00 & {\footnotesize n/a}\\\hline
2 & (5,2,2,2) & 30.00 & 25.00 & 20.00 & 15.00 & 10.00 & 5.00 & {\footnotesize n/a}\\\hline
3 & (2,2,2) & 4.00 & {\footnotesize n/a} & {\footnotesize n/a} & {\footnotesize n/a} & {\footnotesize n/a} & {\footnotesize n/a} & {\footnotesize n/a}\\\hline
3 & (3,2,2) & 6.00 & 3.00 & {\footnotesize n/a} & {\footnotesize n/a} & {\footnotesize n/a} & {\footnotesize n/a} & {\footnotesize n/a}\\\hline
3 & (4,2,2) & 8.00 & 4.00 & {\footnotesize n/a} & {\footnotesize n/a} & {\footnotesize n/a} & {\footnotesize n/a} & {\footnotesize n/a}\\\hline
3 & (3,3,2) & 12.00 & 9.00 & \cellcolor{gray}5.84 & {\footnotesize n/a} & {\footnotesize n/a} & {\footnotesize n/a} & {\footnotesize n/a}\\\hline
3 & (4,3,2) & 16.00 & 12.00 & 8.00 & 4.00 & {\footnotesize n/a} & {\footnotesize n/a} & {\footnotesize n/a}\\\hline
3 & (5,3,2) & 20.00 & 15.00 & 10.00 & 5.00 & {\footnotesize n/a} & {\footnotesize n/a} & {\footnotesize n/a}\\\hline
3 & (6,3,2) & 24.00 & 18.00 & 12.00 & 6.00 & {\footnotesize n/a} & {\footnotesize n/a} & {\footnotesize n/a}\\\hline
4 & (2,2,2) & 4.00 & {\footnotesize n/a} & {\footnotesize n/a} & {\footnotesize n/a} & {\footnotesize n/a} & {\footnotesize n/a} & {\footnotesize n/a}\\\hline
4 & (3,2,2) & 6.00 & 3.00 & {\footnotesize n/a} & {\footnotesize n/a} & {\footnotesize n/a} & {\footnotesize n/a} & {\footnotesize n/a}\\\hline
4 & (4,2,2) & 8.00 & 4.00 & {\footnotesize n/a} & {\footnotesize n/a} & {\footnotesize n/a} & {\footnotesize n/a} & {\footnotesize n/a}\\\hline
4 & (3,3,2) & 12.00 & 9.00 & 6.00 & {\footnotesize n/a} & {\footnotesize n/a} & {\footnotesize n/a} & {\footnotesize n/a}\\\hline
4 & (4,3,2) & 16.00 & 12.00 & 8.00 & 4.00 & {\footnotesize n/a} & {\footnotesize n/a} & {\footnotesize n/a}\\\hline
4 & (5,3,2) & 20.00 & 15.00 & 10.00 & 5.00 & {\footnotesize n/a} & {\footnotesize n/a} & {\footnotesize n/a}\\\hline
4 & (6,3,2) & 24.00 & 18.00 & 12.00 & 6.00 & {\footnotesize n/a} & {\footnotesize n/a} & {\footnotesize n/a}\\\hline
5 & (2,2,2) & 4.00 & {\footnotesize n/a} & {\footnotesize n/a} & {\footnotesize n/a} & {\footnotesize n/a} & {\footnotesize n/a} & {\footnotesize n/a}\\\hline
5 & (3,2,2) & 6.00 & 3.00 & {\footnotesize n/a} & {\footnotesize n/a} & {\footnotesize n/a} & {\footnotesize n/a} & {\footnotesize n/a}\\\hline
5 & (4,2,2) & 8.00 & 4.00 & {\footnotesize n/a} & {\footnotesize n/a} & {\footnotesize n/a} & {\footnotesize n/a} & {\footnotesize n/a}\\\hline
5 & (3,3,2) & 12.00 & 9.00 & 6.00 & {\footnotesize n/a} & {\footnotesize n/a} & {\footnotesize n/a} & {\footnotesize n/a}\\\hline
5 & (4,3,2) & 16.00 & 12.00 & 8.00 & 4.00 & {\footnotesize n/a} & {\footnotesize n/a} & {\footnotesize n/a}\\\hline
5 & (5,3,2) & 20.00 & 15.00 & 10.00 & 5.00 & {\footnotesize n/a} & {\footnotesize n/a} & {\footnotesize n/a}\\\hline
    \end{tabular} 
    \renewcommand\arraystretch{1}

    \caption{Best upper bound on $\log_q(M)$ for $(\mathbf n,M,d)_q$ tensor codes between the Singleton-like bound \legendbox{white} (\ref{eq:singletonlike}), the improved Singleton-like bound \legendbox{lightgray} (\ref{eq:improvedsingletonlike}), and the sphere-packing bound \legendbox{gray} for $d=3$ (\ref{eq:spherepacking3}) and $d=5$ (\ref{eq:spherepacking5}).}
    \label{tab:bestbetween}
\end{table}

\section{Multilinear forms Cayley graph}
\label{sec:multilinearfromsgraph}
We define the \emph{multilinear forms graph} of size $\mathbf n$ over the finite field $\F_q$ to be the graph 
$\Cay(\F_q^{\mathbf{n}},\S_{\mathbf n})$. Since the set $\S_{\mathbf n}$ is stable under multiplication by non-zero scalars, this is a particular case of the graph associated to the $\FFF$-projective metric discussed in our context in \cite[Section 4.2]{AbiadPetersRavagnani2025EigenvalueMethod}. Let us focus on the study of the properties of this graph as well as the computation of its eigenvalues.

By equation (\ref{eq:irreduciblecharactersofFqtensors}) and Theorem~\ref{thm:eigenvaluesofaCayleyGraph}, the eigenvalues (counted with multiplicity) of the multilinear forms graph $\Cay(\F_q^{\mathbf{n}},\S_{\mathbf n})$ are $(\lambda_x)_{x \in \F_q^{\mathbf n}}$, where for each tensor $x \in \F_q^{\mathbf n}$ we have
\[
    \lambda_x  = \sum_{s \in \S_{\mathbf n}} \chi(\langle x|s\rangle).
\]
In the rest of this paper, we will say that $\lambda_x$ is the eigenvalue of $\Cay(\F_q^{\mathbf{n}},\S_{\mathbf n})$ associated to the tensor $x\in\F_q^{\mathbf n}$. Two tensors may have an identical associated eigenvalue in the multilinear forms graph, and the number of tensors with an identical associated eigenvalue is exactly the multiplicity of the eigenvalue in the graph. 

Since the set of simple tensors $\S_{\mathbf n}$ is invariant under the action of $\GL_{\mathbf n}(\F_q)$, and by Proposition~\ref{prop:propertiestensorproddotprod}, for each $x \in \F_{q}^{\mathbf n}$, the eigenvalue associated to a $x$ depends only on its orbit $\GL_{\mathbf n}(\F_q)\cdot x$ under the action of the group $\GL_{\mathbf n}(\F_q)$. Indeed, for each $P = (P_1,\dots,P_m) \in \GL_{\mathbf n}(\F_q)$, since $P^\top = (P_1^\top,\dots,P_m^\top) \in \GL_{\mathbf n}(\F_q)$, we have 
\begin{equation}
    \label{eq:invarianceeigenvalueactionGL}
    \lambda_{M \cdot x} 
    = \sum_{s\in \S_{\mathbf n}} \chi(\langle P \cdot x | s \rangle)
    = \sum_{s\in \S_{\mathbf n}} \chi(\langle x | P^\top \cdot s \rangle)
    = \lambda_x.
\end{equation}
Since the action of $\GL_{\mathbf n}(\F_q)$ corresponds to changes of bases for the slices of the tensor, this immediately implies that the eigenvalue associated to a tensor only depends on its different slice spaces and not on the particular choice of coordinate tensor.

Although not directly relevant to our work, we remark that, similarly to the above, any group action for which the adjoint action under 
$\langle \cdot | \cdot \rangle$ preserves the set of rank-one tensors presents a similar invariance. In particular, if some entries of $\mathbf n$ are equal, the mode permutation on the tensors is an action that will preserve the associated eigenvalues. More precisely, if we have $n_{j_1} = n_{j_2}$ for two indices $j_1,j_2 \in \llbracket 1,m\rrbracket$ with $j_{1} < j_2$, then there is an action of the permutation group $S_2$ on $\F_q^{\mathbf n}$ given by $\text{id} \cdot u = u$ and 
\[
    (1,2) \cdot u =  u_1 \otimes \cdots \otimes u_{j_1-1} \otimes u_{j_2} \otimes u_{j_1+1} \otimes \cdots \otimes   u_{j_2-1} \otimes u_{j_1} \otimes u_{j_2+1} \otimes \cdots u_{m}
\]
for each $u = u_1 \otimes \cdots \otimes u_{m} \in \S_{\mathbf n}$, which can be extended by linearity to any element of the group $\F_q^{\mathbf n}$. Then one can easily note that $\langle (1,2)\cdot x | y\rangle = \langle x | (1,2) \cdot y\rangle$ for each $x,y \in \F_q^{\mathbf n}$, and thus obtain the invariance of the eigenvalues $\lambda_{(1,2)\cdot x} = \lambda_{x}$ for each $x \in \F_q^{\mathbf n}$, since $\S_{\mathbf n}$ is also invariant under mode permutation. 

For $m \leq 2$, the multilinear forms graph eigenvalues are already known as explained below. The following sections will provide formulas for the computation of the eigenvalues for the other cases. For $m = 1$, since $\S_{n_1} = \F_q^{n_1} \backslash \{0\}$, the graph $\Cay(\F_q^{n_1},\S_{n_1})$ is the complete graph $K_{q^{n_1}}$ and thus its eigenvalues are known. Namely, its eigenvalues are $-1$ and $(q^{n_1}-1)$ with respective multiplicities $(q^{n_1}-1)$ and $1$. For $m = 2$, the graph $\Cay(\F_q^{n_1\times n_2},\S_{n_1 \times n_2})$ corresponds to the underlying graph of the bilinear forms association scheme $\Bil_q(n_1,n_2)$ introduced in \cite{DELSARTE1978226}. Its eigenvalues have been computed and can be stated as follows.

\begin{Prop}[{\cite[Section 9.5]{Brouwer1989TheoryOfDistanceRegularGraphs},\cite{DELSARTE1978226}}]\label{prop:eigenvaluesCayleyGraphMatrices}
    The eigenvalues of $\Cay(\F_q^{n_1\times n_2},\S_{n_1 \times n_2})$ are $\theta_0,\dots,\theta_{\min(n_1,n_2)}$ where
    \[
        \theta_r = |\S_{n_1 \times n_2}| - \frac{q^{n_1+n_2}(1-q^{-r})}{q-1},
    \]
    and each $\theta_r$ has multiplicity 
    $ m_r = A_r(\F_q^{n_1\times n_2}) = \prod_{s= 0}^{r-1} \frac{(q^{n_1-s}-1)(q^{n_2}-q^{s})}{q^{s+1}-1}$ for each $r \in \llbracket 0,\min(n_1,n_2)\rrbracket$.
\end{Prop}

More precisely, the eigenvalue associated to a matrix $x \in \F_{q}^{n_1 \times n_2}$ is $\lambda_x = \theta_{\rank(x)}$. In particular, there is a one-to-one correspondence between the orbits of $\F_q^{n_1 \times n_2}$ under the action of $\GL_{n_1\times n_2}(\F_q)$ and the different eigenvalues of $\Cay(\F_{q}^{n_1 \times n_2},\S_{n_1 \times n_2})$. We will observe in Section \ref{sec:eigenvalues233} that this is not the case in general for higher order tensors; see Table \ref{tab:spectrumFq233}. It has also been established that the graph $\Cay(\F_q^{n_1\times n_2},\S_{n_1 \times n_2})$ is distance-regular; see \cite[Section 9.5]{Brouwer1989TheoryOfDistanceRegularGraphs}. In particular, its eigenvalues can be computed using the intersection array of the graph; see \cite[4.1.B]{Brouwer1989TheoryOfDistanceRegularGraphs}. Alternatively, it is also possible to compute, or each $r \in \llbracket 0,\min(n_1,n_2) \rrbracket$ the eigenvalue $\theta_r$ associated to canonical representative of the orbit of rank-$r$ matrices in $\F_{q}^{n_1 \times n_2}$  with Theorem~\ref{thm:eigenvaluesofaCayleyGraph}, and extend the result to all matrices using the  aforementioned invariance of the eigenvalues up to equivalence of the associated matrix.

For the rest of this section, we will give properties of the multilinear forms graph and then give an iterative formula that characterises the eigenvalue of any multilinear forms graph (Theorem~\ref{thm:eigenvaluesmultilinearformsgraphrecursive}), as well as a dual formula linking the eigenvalues to the types of intersection that a vector subspace of high enough dimension has with the Segre variety (see Theorem~\ref{thm:eigenvaluesmultilineargraphdual}). 

We remark that in \cite{DhankharKadyan2026EigengaluesIntegralCayley}, the authors provide formulas to compute the spectrum of a Cayley graph defined over an abelian group whose eigenvalues are integral, and deduce bounds on the least eigenvalue of the graph. In our work, we additionally exploit the vector space structure of $\F_q^{\mathbf n}$ and use the properties of the different tensor contraction spaces to characterise the eigenvalues of the multilinear forms graph. In particular, we obtain formulas for the exact values of the second-largest and least eigenvalues in the spectrum of the graph.

\subsection{Geodesic distance, walk-regularity and distance-regularity of the graph}

As a Cayley graph on a finite abelian group, we can first state the following properties of the multilinear forms graph deduced from the ones recalled in Section \ref{subsec:graphtheoryfundamentals}. 

\begin{Prop}\label{prop:tensorrankdistanceandgeodesiccoincides}
    The (tensor) rank distance $d_{\rank}$ on $\F_{q}^{\mathbf{n}}$ and the Geodesic distance $d_{\Cay(\F_q^{\mathbf{n}},\S_{\mathbf n})}$ coincide. Moreover, the graph $\Cay(\F_q^{\mathbf{n}},\S_{\mathbf n})$ is vertex-transitive, and in particular regular and walk-regular.
\end{Prop}

The fact that the geodesic distance is identical to the rank distance is a direct consequence of the definition of the rank. The distance-regularity of the graph however depends on the parameter $m$. To observe this, let us start by characterising a particular number $p^i_{j,k}(x,y)$ of the graph using the different minimal rank forms of the tensor $x-y$.

\begin{Lemma}\label{lemma:piandnumberofsecompositionsstartingat}
    Let $i$ be an integer with $i\geq 2$ and let $x,y \in \F_{q}^{\mathbf{n}}$ be such that $\rank(x-y) = i$. Then the number $p^i_{1,i-1}(x,y)$ of the graph $\Cay(\F_q^{\mathbf n},\S_{\mathbf n})$ is exactly the number of elements in $\S_{\mathbf n}$ that appear in at least one minimal rank form of $y-x$, i.e. we have 
    \[
        p^i_{1,i-1}(x,y) = |\{s_1 \in \S_{\mathbf n} \mid \exists s_2,\dots,s_i \in \S_{\mathbf n} : y - x = s_1 + s_2 + \cdots + s_i\}|.
    \] 
\end{Lemma}
\begin{proof}
    By definition of the rank we have
    \begin{align*}
        p^i_{1,i-1}(x,y) 
        &= |\{z \in \F_{q}^{\mathbf n} \ : \  \rank(z - x) = 1, \rank(y - z) = i-1\}|\\
        &= |\{s_1 \in \S_{\mathbf n} \ : \  \rank(y - x - s_1) = i-1\}|.
    \end{align*}
    If $i = 2$, then $p^{2}_{1,1}(x,y) = |\{s_1 \in \S_{\mathbf n} \ : \  \exists s_2 \in \S_{\mathbf n}, y - x - s_1 = s_2\}|$, which concludes the proof. Assume now that $i \geq 3$. Let $s_1 \in \S_{\mathbf n}$ be a simple tensor. Since $i-1 \geq 2$ we have 
    \begin{equation}
        \label{eq:equivalencerank}         
        \rank(y-x-s_1) \leq i-1 \iff \exists s_2,\dots,s_i \in \S_{\mathbf n} : y - x -s_1 =  s_2 + \cdots + s_i.
    \end{equation}
    Moreover, since $\rank(y-x) = i$ we have $\rank(y-x-s_1) \geq i-1$. Therefore, $\rank(y-x-s_1) = i-1$ is equivalent to (\ref{eq:equivalencerank}) which concludes the proof.
\end{proof}

\begin{Prop}\label{prop:distanceregularityofsegregraph}
    Let $\mathbf{n} = (n_1,\dots,n_m)$ be a sequence of integers with $n_j \geq 2$ for each $j \in \llbracket 1,m\rrbracket$. Then the graph $\Cay(\F_{q}^{\mathbf{n}},\S_{\mathbf n})$ is distance-regular if and only if $m \leq 2$.
\end{Prop}
\begin{proof}
    The statement is void for $m=1$ since the graph is complete. The statement is true for $m=2$ by \cite[Section 9.5]{Brouwer1989TheoryOfDistanceRegularGraphs}. Now assume that $m \geq 3$, and let us prove that for each $t \in \llbracket 2, \min_{j \in \llbracket 1,m\rrbracket}(n_j)\rrbracket$, there exist $x_1,x_2 \in \F_q^{\mathbf{n}}$ two tensors of such that $\rank(x_1) = \rank(x_2) = t$ and such that $p^t_{1,t}(0,x_1) \neq p^t_{1,t}(0,x_2)$. By Lemma~\ref{lemma:piandnumberofsecompositionsstartingat}, one can conclude the proof by finding two such tensors that do not have the same number of elements of $\S_{\mathbf n}$ appearing as terms of their respective list of minimal rank forms. Fix such an integer $t$ and define the tensors
    \[
        x_1 = \sum_{i=1}^t e_i^{(n_1)} \otimes \cdots \otimes e_i^{(n_m)} \text{ and }  x_2 = \sum_{i=1}^t e_i^{(n_1)} \otimes e_i^{(n_2)} \otimes  e_1^{(n_3)} \otimes \cdots \otimes e_1^{(n_m)}.
    \]
    Note that Proposition~\ref{prop:slicespdimlowerboundstrank} implies that $\rank(x_1) = \rank(x_2) = t$. On the one hand, Kruskal uniqueness theorem (Theorem~\ref{thm:kruskaluniquenesstheorem}) implies that $x_1$ has a unique minimal rank form up to permutation of the terms of the sum above. Therefore, there are precisely $t$ elements in $\S_{\mathbf n}$ in the list of its minimal rank forms. On the other hand, note that 
    \begin{align*}
         x_2 
         &= (e_1^{(n_1)} + e_2^{(n_1)}) \otimes e_1^{(n_2)} \otimes \cdots \otimes e_1^{(n_m)}\\ 
         &\ \ \ \ + e_2^{(n_1)} \otimes (e_2^{(n_2)}-e_1^{(n_2)}) \otimes  e_1^{(n_3)} \otimes \cdots \otimes e_1^{(n_m)}\\
         &\ \ \ \ + \sum_{i=3}^t e_i^{(n_1)} \otimes e_i^{(n_2)} \otimes  e_1^{(n_3)} \otimes \cdots \otimes e_1^{(n_m)}.
    \end{align*}
    Therefore, there are at least $t+1$ elements in $\S_{\mathbf n}$ that appear in the minimal rank forms of $x_2$, which concludes the proof.
\end{proof}

\subsection{Iterative computation of the eigenvalues}

Since the multilinear forms graph is not distance-regular for $m \geq 3$, we will use Theorem~\ref{thm:eigenvaluesofaCayleyGraph} and the expression of the complex characters of $\F_{q}^{\mathbf n}$ given in (\ref{eq:irreduciblecharactersofFqtensors}) to obtain the eigenvalues of the multilinear forms graph. Let us first observe that the eigenvalues of the graph can be theoretically computed recursively using the properties of the slice spaces as follows.

\begin{Thm}\label{thm:eigenvaluesmultilinearformsgraphrecursive}
    Assume $m\geq 2$. Let $j \in \llbracket 1,m\rrbracket$ be an integer. For each $y \in \F_q^{\mathbf n(j)}$, denote by $\lambda_{y}$ the eigenvalue of $\Cay(\F_q^{\mathbf n(j)},\S_{\mathbf n(j)})$ associated to $y$. Then, for each $x \in \F_q^{\mathbf n}$, $\lambda_x$ the eigenvalue of $\Cay(\F_q^{\mathbf{n}},\S_{\mathbf n})$ associated to $x$ is given by
    \[
        (q-1)\lambda_x = -|\S_{\mathbf n(j)}| + q^{(n_{j} - \dim \slicesp_{j}(x))}\sum_{y \in \slicesp_{j}(x)} \lambda_{y}.
    \]  
\end{Thm}
\begin{proof}

    Let $x_1,\dots,x_{n_{j}} \in \F_{q}^{\mathbf n(j)}$ be the $j$-th slices of $x$, that is the elements such that $x = \sum_{i =1}^{n_j} x_i \otimes e_i^{(n_j)}$, or equivalently $\modemult_{j}(e_i^{(n_j)},x) = x_i$ for each $i \in \llbracket 1,n_{j}\rrbracket$. Note that for each $s' \in  \F_{q}^{\mathbf n(j)}$ and each $w \in \F_{q}^{n_{j}}$ we have
    \[
        \langle x \mid s'\otimes w\rangle = \sum_{i = 1}^{n_{j}} w_i\langle x_i \mid s'\rangle = \left.\left\langle\sum_{i = 1}^{n_{j}} w_i x_i \right| s' \right\rangle. 
    \]
    Therefore, by (\ref{eq:irreduciblecharactersofFqtensors}) and Theorem~\ref{thm:eigenvaluesofaCayleyGraph}, we have 
    \begin{align*}
        \lambda_x 
        &= \sum_{s \in \S_{\mathbf n}} \chi(\langle x \mid s\rangle) \\
        &= \frac{1}{q-1}\sum_{s' \in \S_{\mathbf n(j)}}\sum_{w \in \F_{q}^{n_{j}} \backslash \{0\}} \chi(\langle x \mid s'\otimes w\rangle) \\
        &= \frac{1}{q-1}\sum_{s' \in \S_{\mathbf n(j)}}\sum_{w \in \F_{q}^{n_{j}} \backslash \{0\}} \chi\left(\left.\left\langle\sum_{i = 1}^{n_{j}} w_i x_i \right| s' \right\rangle\right).
    \end{align*}
    Moreover, the map $\modemult_{j}(\cdot,x) : \F_{q}^{n_{j}} \to \slicesp_{j}(x), w \mapsto \sum_{i=1}^{n_{j}} w_ix_i$, surjective by construction, has a kernel of dimension $n_j-\dim \slicesp_{j}(x)$, therefore we have
    \begin{align*}
        (q-1)\lambda_x 
            &= \sum_{s' \in\S_{\mathbf n(j)}}\left( -1 + \sum_{w \in \F_{q}^{n_{j}} } \chi\left(\left.\left\langle\sum_{i = 1}^{n_{j}} w_i x_i \right| s' \right\rangle\right)\right)\\
        &= -|\S_{\mathbf n(j)}| + q^{(n_{j}-\dim\slicesp_{j}(x))}\sum_{s' \in\S_{\mathbf n(j)}}\sum_{y \in \slicesp_{j}(x) } \chi\left(\left.\left\langle y \right| s' \right\rangle\right)\\
        &= -|\S_{\mathbf n(j)}| + q^{(n_{j}-\dim\slicesp_{j}(x))}\sum_{y \in \slicesp_{j}(x) }  \lambda_{y}.
    \end{align*}
\end{proof}

    Since $\F_q^{\mathbf n}$ is an abelian group isomorphic to a product of cyclic groups of order $p$, and since $\S_{\mathbf n}$ is a union of cyclic subgroups, the results of \cite{BridgesMena1982RationalGmatrices} establish that the eigenvalues of $\Cay(\F_q^{\mathbf n},\S_{\mathbf n})$ are integers. 
    This can also be deduced as a corollary of Theorem~\ref{thm:eigenvaluesmultilinearformsgraphrecursive} as follows.

\begin{Corol}\label{corol:eigenvaluesareintegers}
    For every $\mathbf{n}$, the eigenvalues of $\Cay(\F_q^{\mathbf{n}},\S_{\mathbf n})$ are integers.
\end{Corol}
\begin{proof}
    Since the eigenvalues of the complete graph are integers, the statement is true for $m=1$. Moreover, the statement can be extended by induction using Theorem~\ref{thm:eigenvaluesmultilinearformsgraphrecursive} since $|\S_{\mathbf n}|$ is divisible by $(q-1)$ for each integer tuple $\mathbf n$ and since the eigenvalue associated to a tensor is the same as the one associated to any of its non-zero scalar multiples.
\end{proof}


Note that Theorem~\ref{thm:eigenvaluesmultilinearformsgraphrecursive} can be used to obtain the result of Proposition~\ref{prop:eigenvaluesCayleyGraphMatrices} from the spectrum of the complete graph $\Cay(\F_q^{n_1},\S_{n_1})$. Indeed, for each $x \in \F_q^{n_1\times n_2}$, we have 
\[
    q^{n_2 - \dim \slicesp_2(x)}\sum_{y \in \slicesp_2(x)} \lambda_{y} = q^{n_2 - \rank(x)}\left((q^{n_1}-1) - (q^{\rank(x)}-1) \right) = q^{n_1 + n_2 - \rank(x)} - q^{n_2},
\]
since in the graph $\Cay(\F_q^{n_1},\S_{n_1})$, the eigenvalue associated to a non-zero vector is $-1$ while the eigenvalue associated to the zero vector is $(q^{n_1}-1)$. Therefore, we have
\[
    (q-1)\lambda_x = -(q^{n_1}-1) + q^{n_1 + n_2 - \rank(x)} - q^{n_2} = (q^{n_1} -1)(q^{n_2}-1) - q^{n_1 + n_2} + q^{n_1 + n_2 - \rank(x)},
\]
which proves the statement. 

Moreover, we can use the fact that the eigenvalues of $\Cay(\F_q^{n_1\times n_2},\S_{n_1 \times n_2})$ only depend on the rank of the associated matrix to obtain an expression of the eigenvalues of $\Cay(\F_q^{n_1\times n_2\times n_3},\S_{n_1 \times n_2 \times n_3})$ in terms of the weight enumerator of the tensor slice space as a linear rank-metric code.
\begin{Corol}\label{corol:eigenvaluestrilinearformsgraph}
    Let $x \in \F_q^{n_1\times n_2\times n_3}$ be a tensor, and let $j \in \{1,2,3\}$. Then the eigenvalue of the trilinear forms graph $\Cay(\F_q^{n_1\times n_2\times n_3},\S_{n_1 \times n_2 \times n_3})$ associated to $x$ is given by
    \[
        \lambda_x 
        = |\S_{n_1 \times n_2 \times n_3}| + \frac{q^{n_1 +n_2+  n_3}}{(q-1)^2}\left( -1 +  \sum_{r \geq 0} \frac{A_r(\slicesp_j(x))}{|\slicesp_j(x)|}q^{-r} \right).
    \]
\end{Corol}
\begin{proof}
    For convenience, we will fix $j = 1$, the proof is similar for $j = 2$ or $j = 3$. By Theorem~\ref{thm:eigenvaluesmultilinearformsgraphrecursive} and Proposition~\ref{prop:eigenvaluesCayleyGraphMatrices}, if we denote by $\lambda_y$ the eigenvalue of $\Cay(\F_{q}^{n_2 \times n_3},\S_{n_2 \times n_3})$ associated to any tensor $y \in \F_q^{n_2 \times n_3}$ we have 
    \begin{align*}
        (q-1)\lambda_x 
        &= -|\S_{n_2 \times n_3}| + q^{(n_{1} - \dim \slicesp_{1}(x))}\sum_{y \in \slicesp_{1}(x)} \lambda_{y}\\
        &= -|\S_{n_2 \times n_3}| + \frac{q^{n_{1}}}{|\slicesp_1(x)|} \sum_{y \in \slicesp_{1}(x)}\left(|\S_{n_2 \times n_3}| - \frac{q^{n_2+n_3}(1-q^{-\rank(y)})}{q-1} \right) \\
        &= |\S_{n_2 \times n_3}|\left(-1+q^{n_1}\right) - \frac{q^{n_1+n_2+n_3}}{(q-1)|\slicesp_1(x)|}\sum_{y \in \slicesp_1(x)}(1-q^{-\rank(y)})\\
        &= {|\S_{n_1 \times n_2 \times n_3}|}(q-1) - \frac{q^{n_1+n_2+n_3}}{(q-1)} +  \frac{q^{n_1+n_2+n_3}}{(q-1)|\slicesp_1(x)|}\sum_{y \in \slicesp_1(x)}q^{-\rank(y)}.
    \end{align*}
    Therefore, we have 
    \begin{align*}
        \lambda_x 
        &= |\S_{n_1 \times n_2 \times n_3}| + \frac{q^{n_1+n_2+n_3}}{(q-1)^2}\left( -1  + \sum_{r\geq  0} \frac{|\{y \in \slicesp_1(x) : \rank(y) = r\}|}{|\slicesp_1(x)|}q^{-r}\right),
    \end{align*}
    which gives us the wanted result.
\end{proof}

\begin{Rq}
    The result of Corollary~\ref{corol:eigenvaluestrilinearformsgraph} can be rephrased using either the weight enumerator or the generating function of the discrete variables associated to the tensors in question. Namely, if $x \in \F_q^{n_1\times n_2 \times n_3}$ is a tensor, then for each $j \in \{1,2,3\}$ we have 
    \begin{equation}
        \label{eq:eigenvalue3tensorwithweightenum}
        \lambda_x = |\S_{n_1 \times n_2 \times n_3}| + \frac{q^{n_1+n_2+n_3}}{(q-1)^2}\left(q^{-\dim\slicesp_j(x)}W_{\slicesp_j(x)}(1,q^{-1})-1\right).
    \end{equation}
    This in turn indicates that the value of $\sum_{r \geq 0} \frac{A_r(\slicesp_j(x))}{|\slicesp(x)|}q^{-r} = q^{-\dim\slicesp_j(x)}W_{\slicesp_j(x)}(1,q^{-1}) $ is independent of the choice of $j \in \{1,2,3\}$. We remark that this independence property was also observed in \cite[Theorem 3.11]{alfaranoborelloneri2026geometryrankmetriccodes}. Indeed, the authors observed that $\sum_{M \in \slicesp_1(x) \backslash \{0\}}(q^{n_3 - \rank(M)} -1) = \sum_{X \in \slicesp_3(x) \backslash \{0\}}(q^{n_1 - \rank(X)} -1)$ for any tensor $x \in \F_q^{n_1 \times n_2 \times n_3}$ such that $\dim \slicesp_1(x) = n_1$ and $\dim \slicesp_3(x) = n_3$, and this equality may be rewritten as  $1-q^{n_1} - q^{n_3} + q^{n_1 + n_3}q^{-n_1}W_{\slicesp_1(x)}(1,q^{-1})= 1-q^{n_1} - q^{n_3} + q^{n_1 + n_3}q^{-n_3}W_{\slicesp_3(x)}(1,q^{-1})$.\\
    Additionally, equation (\ref{eq:eigenvalue3tensorwithweightenum}) can be also rephrased as follows. Let $X_j$ is a random variable chosen uniformly at random in $\slicesp_j(x)$. Then $\lambda_x$ is a function of the probability generating function $\E[t^{\rank(X_j)}]$ of $\rank(X_j)$ at the point $t= q^{-1}$, i.e. we have 
    \begin{equation}
        \lambda_x = |\S_{n_1 \times n_2 \times n_3})| + \frac{q^{n_1+n_2+n_3}}{(q-1)^2}\left(\E\left[q^{-\rank(X_j)}\right] -1\right).
    \end{equation}
\end{Rq}
   
\subsection{Dual expressions}
Using MacWilliams identities for rank-metric codes, we can obtain a dual expression of the eigenvalues of $\Cay(\F_q^{n_1 \times n_2 \times n_3},\S_{n_1 \times n_2 \times n_3})$. We can also provide a similar dual expression of all the eigenvalues of $\Cay(\F_q^{\mathbf{n}},\S_{\mathbf n})$ for any sequence of positive integers $\mathbf{n}$. To achieve that, we will start by studying the value of the rank-weight enumerator of a rank-metric code at the point $(1,q^{-t})$ for any integer $t \in \llbracket 1,n\rrbracket$.

For any $[m\times n, k]_q$ rank-metric code $\CCC$ and subspace $U \leq \F_q^m$ we denote by $\CCC(U)$ the subspace of $\CCC$ given by $\CCC(U)=\{C \in \CCC \ : \ \rowsp(C) \subseteq U\}$.
For any integer $i \in \llbracket 1,n\rrbracket$, 
we denote by $B_{i}(\CCC)$ the $i$-th binomial moment of $\CCC$, that is the integer given by
\[
    B_{i}(\CCC) = \sum_{U \subseteq \F_{q}^{n}, \dim U = i} |\CCC(U)|.
\]
We denote by $p^n_i(X,Y)$ the homogeneous polynomial given by
\[
    p^n_{i}(X,Y) = Y^i\prod_{\ell = 0}^{n-i-1}(X-q^iY).
\]
Then we can state the following relation between the weight enumerator and the rank-distribution of its dual code.

\begin{Prop}\label{prop:weightennumeratorandbernstein}
    Let $\CCC$ be an $[m\times n,k]_q$ rank-metric code. 
    Then 
    \begin{align*}
        W_{\CCC}(X,Y) = \sum_{ j=0}^n A_j(\CCC^\perp) \sum_{i=0}^{n-j} q^{k-m(n-i)} \gbc{i}{n-j}{q} p^n_i(x,y)
    \end{align*}
\end{Prop}
\begin{proof}
    Rewriting \cite[(4.7)]{BlancoChaconByrneDuursmaSheekey2018RankMetricZeta}, we obtain 
       $ \displaystyle W_{\CCC}(x,y) = \sum_{i = 0}^n p_{i}^n(X,Y)B_i(\CCC).$
    Moreover, by \cite[Lemmas 28 and 30]{RavagnaniRankMetricCodes} we have 
    \[
        B_i(\CCC) = q^{k - m(n-i)}B_{n-i}(\CCC^\perp) \text{ and } B_{n-i}(\CCC^\perp) = \sum_{j=0}^{n-i} \gbc{i}{n-j}{q} A_{j}(\CCC^\perp).
    \]
    Therefore we have 
    \begin{align*}
        W_{\CCC}(X,Y) &= \sum_{i = 0}^n\sum_{j = 0}^{n-i} p_{i}^n(X,Y)q^{k-m(n-i)}\gbc{i}{n-j}{q}A_j(\CCC^\perp)\\
        &= \frac{1}{|\CCC^\perp|}\sum_{j = 0}^n A_j(\CCC^\perp)  \sum_{i = 0}^{n-j} p_{i}^n(X,Y)q^{mi}\gbc{i}{n-j}{q}.
    \end{align*}
\end{proof}

The following formula has already been computed in \cite[Theorem 30]{delaCruzGorlaLopezRavagnani2018}. We obtain it here by evaluation 
of $W_{\CCC}(X,Y)$ at $(1,q^{-t})$.

\begin{Corol}\label{corol:weightennumeratorexpressiondual}
        Let $\CCC$ be an $[m\times n,k]_q$ rank-metric code. Let $t \in \llbracket 1,n\rrbracket$ be a positive integer. Then $W_{\CCC}(1,q^{-t})$ only depends on $q$, $m$, $n$, $k$ and $A_j(\CCC^\perp)$ for $j\in \llbracket 1,t\rrbracket$. That is,
    \[
        W_{\CCC}(1,q^{-t}) = \frac{1}{|\CCC^\perp|} \sum_{j=0}^{t} {A_j(\CCC^\perp)} \sum_{i=n-t}^{n-j} p^n_i(1,q^{-t})q^{mi}\gbc{i}{n-j}{q}.
    \]
    In particular, we have 
    \[
        W_{\CCC}(1,q^{-1}) = q^{k-m-n}\left(q^{n} + q^{m} - 1 + (q-1)A_1(\mathcal C^\perp)\right).
    \]
\end{Corol}
\begin{proof}
    By definition, for each $i \in \llbracket 0,n\rrbracket$ we have $p_{i}^n(1,q^{-t}) = q^{-ti}\prod_{\ell = 0}^{n-i-1}(1-q^{\ell -t})$, which is non-zero if and only if $n-i-1 <t$, i.e. $i\geq n-t$. Therefore, by Proposition~\ref{prop:weightennumeratorandbernstein}, we have 
    \begin{align*}
        W_{\CCC}(1,q^{-t}) &= \frac{1}{|\CCC^\perp|}\sum_{j = 0}^n A_j(\CCC^\perp)  
        \sum_{i = n-t}^{n-j} p_{i}^n(1,q^{-t})q^{mi}\gbc{i}{n-j}{q}.
    \end{align*}
    Since the sum over the variable $i$ is empty when $j >t$, one can assume that the sum over the variable $j$ stops at $j = t$ and we have the desired formula for $W_{\CCC}(1,q^{-t})$. In particular, since $p_{n-1}^n(1,q^{-1}) = q^{-n}(q-1)$ and $p_{n}^n(1,q^{-1}) = q^{-n}$ we have 
    \begin{align*}
        W_{\CCC}(1,q^{-1})
        &= \frac{1}{|\CCC^\perp|}\sum_{j = 0}^1 A_j(\CCC^\perp)  \sum_{i = n-1}^{n-j} p_{i}^n(1,q^{-1})q^{mi}\gbc{i}{n-j}{q}\\
        &=  q^{k - mn} \left(q^{-n}(q-1)q^{m(n-1)}\gbc{n-1}{n}{q} + q^{-n}q^{mn} + A_1(\CCC^\perp)q^{-n}(q-1)q^{m(n-1)} \right)\\
        &= q^{k-n-m} \left(   q^m +  q^n -1+ A_1(\CCC^\perp)(q-1)\right).
    \end{align*}
\end{proof}

\begin{Corol}\label{corol:eigenvaluestrilinearformsgraphdual}
    Let $\mathbf n = (n_1,n_2,n_3)$ be a tuple of positive integers and $x \in \F_q^{n_1 \times n_2 \times n_3}$ be a tensor, and let $j \in \{1,2,3\}$. Then the eigenvalue of the trilinear forms graph $\Cay(\F_q^{n_1 \times n_2 \times n_3},\S_{n_1 \times n_2 \times n_3})$ associated to $x$ is given by
    \[
        \lambda_x 
        = \frac{q^{n_j}A_1(\slicesp_j(x)^\perp) - |\S_{\mathbf n(j)}|}{q-1}.
    \]
\end{Corol}
\begin{proof}
    For convenience, let us fix $j = 1$. The proofs for $j = 2$ and $j=3$ are similar. By Corollary~\ref{corol:eigenvaluestrilinearformsgraph} and Corollary~\ref{corol:weightennumeratorexpressiondual}, we have
    \begin{align*}
        \lambda_x 
        &= |\S_{n_1 \times n_2 \times n_3}| + \frac{q^{n_1+n_2+n_3}}{(q-1)^2}\left(-1+q^{-n_2-n_3} \left(   q^{n_2 }+  q^{n_3} -1+ A_1(\slicesp_1(x)^\perp)(q-1)\right)\right)\\
        &= |\S_{n_1 \times n_2 \times n_3}| -\frac{q^{n_1+n_2+n_3} + q^{n_1 + n_2}+q^{n_1 + n_3}-q^{n_1}}{(q-1)^2} + \frac{q^{n_1}A_1(\slicesp_1(x)^\perp)}{q-1}\\
        &=  -\frac{|\S_{n_2 \times n_3}|}{q-1} + \frac{q^{n_1}A_1(\slicesp_1(x)^\perp)}{q-1}.
    \end{align*}
\end{proof}

In the rest of this section, we will show that the eigenvalue formula of the trilinear forms graph given in Corollary~\ref{corol:eigenvaluestrilinearformsgraphdual} may be extended to the multilinear forms graphs with an independent proof (see Theorem~\ref{thm:eigenvaluesmultilineargraphdual}). We will proceed by induction starting with the base case $m=2$. We first use a counting lemma.

In the following lemma and in the proof of the theorem below, we will use the following abuse of notation: for $\mathbf n = (n_1,\dots,n_m)$, $x \in \F_q^{\mathbf n}$, and $u \in \F_q^{n_j}$, we will enumerate the slice spaces and the contraction maps of $y = \modemult_j(u,x)$ by the numbering of the original tensor $x$, i.e. $\llbracket 1,m\rrbracket\backslash\{j\}$ and not by $\llbracket 1,m-1\rrbracket$. More precisely, for each $\ell \in \llbracket 1,m\rrbracket\backslash\{j\}$, we denote by $\slicesp_{\ell}(y)$ the vector-space generated by $M\cdot x$ where $M$ a tuple of matrices such that $M_{j} = u$, $M_\ell \in \F_q^{n_\ell}$ and $M_{s} = I_{n_s}$ for $s \neq j,\ell$, as well as $\modemult_{\ell}(\cdot,y) : \F_q^{n_{\ell}} \to \F_q^{\mathbf n(j,\ell)}$ the corresponding contraction map (without the abuse of notation, the mentioned slice space is $\slicesp_\ell(y)$ if $\ell < j$, and is $\slicesp_{\ell-1}(y)$-th if $j<\ell$).

\begin{Lemma}\label{lemma:sumofA1overslicespace}
    Let $\mathbf{n} = (n_1,\dots,n_m)$ be a sequence of positive integers with $m\geq 3$. Let $j,\ell \in \llbracket 1,m\rrbracket$ be distinct integers. Let $x \in \F_{q}^{\mathbf{n}}$. Then we have 
    \[
        \sum_{y \in \slicesp_{j}(x)}A_1(\slicesp_{\ell}(y)^\perp) = q^{\dim {\slicesp_{j}(x)}-n_{j}} \left(\left|\S_{\mathbf n(j,\ell)}\right| + (q-1) A_1\left(\left(\slicesp_{\ell}\left(x\right)\right)^\perp\right)\right).
    \]
\end{Lemma}

\begin{proof}
    Without loss of generality, we assume that $j = 1$ and $\ell = 2$. By Proposition~\ref{prop:propertiestensorproddotprod}, for each element $u_{1} \in \F_{q}^{n_{1}}$ and for each simple tensor $\sigma = u_{3}\otimes \cdots \otimes u_m \in \S_{\mathbf n(1,2)}$ we have
    \begin{align*}
        \sigma \in \left(\slicesp_{2}\left(\modemult_{1}(u_{1},x)\right)\right)^\perp
        &\iff  \langle \sigma \ \mid \ \modemult_{2}(u_{2},\modemult_{1}(u_{1},x)) \rangle = 0 \text{ for all }u_{2} \in \F_{q}^{n_{2}} ,\\ 
        &\iff \langle \sigma \ \mid \ (u_1,u_2,I_{n_3},\dots, I_{n_m})\cdot x \rangle = 0 \text{ for all }u_{2} \in \F_{q}^{n_{2}} ,\\ 
        &\iff  (u_{1},u_2,u_3,\dots, u_{m})\cdot x  = 0 \text{ for all }u_{2} \in \F_{q}^{n_{2}} ,\\ 
        &\iff  \langle u_1 \otimes \sigma \ \mid \ (u_2,I_{n_3},\dots, I_{n_m})\cdot x \rangle = 0 \text{ for all }u_{2} \in \F_{q}^{n_{2}} ,\\ 
        &\iff \langle u_1 \otimes \sigma \ \mid \ \modemult_{2}(u_{2},x) \rangle = 0 \text{ for all }u_{2} \in \F_{q}^{n_{2}} ,\\
        &\iff\sigma \otimes u_{1}  \in \slicesp_{2}\left(x\right)^\perp.
    \end{align*}
    Consequently, we have 
    \begin{align*}
        &\sum_{y \in \slicesp_{1}(x)}A_r(\slicesp_{2}(y)^\perp) \\
        &= q^{\dim {\slicesp_{1}(x)}-n_{1}} \sum_{u_{1} \in \F_{q}^{n_{1}}} A_1\left(\left(\slicesp_{2}(\modemult_{1}(u_{1},x))\right)^\perp\right)\\
        &= q^{\dim {\slicesp_{1}(x)}-n_{1}} \sum_{u_{1} \in \F_{q}^{n_{1}}} \left| \left\{ \sigma\in \S_{\mathbf n(1,2)} \ : \ \sigma \in \left(\slicesp_{2}(\modemult_{1}(u_{1},x))\right)^\perp\right\}\right| \\
        &= q^{\dim {\slicesp_{1}(x)}-n_{1}} \sum_{u_{1} \in \F_{q}^{n_{1}}} \left| \left\{ \sigma\in \S_{\mathbf n(1,2)} \ : \  \sigma \otimes u_{1}  \in \left(\slicesp_{2}\left(x\right)\right)^\perp \right\}\right| \\
        &= q^{\dim {\slicesp_{1}(x)}-n_{1}} \left(\left|\S_{\mathbf n(1,2)}\right| + \sum_{u_{1} \in \F_{q}^{n_{1}}\backslash\{0\}} \left| \left\{ \sigma\in \S_{\mathbf n(1,2)} \ : \  \sigma \otimes u_{1}  \in \left(\slicesp_{2}\left(x\right)\right)^\perp  \right\}\right|  \right)\\
        &= q^{\dim {\slicesp_{1}(x)}-n_{1}} \left(\left|\S_{\mathbf n(1,2)}\right| + (q-1) \left| \left\{ s' \in \S_{\mathbf n(2)} \ : \  s'  \in \left(\slicesp_{2}\left(x\right)\right)^\perp  \right\}\right|\right) \\
        &= q^{\dim {\slicesp_{1}(x)}-n_{1}} \left(\left|\S_{\mathbf n(1,2)}\right| + (q-1) A_1\left(\left(\slicesp_{2}\left(x\right)\right)^\perp\right)\right).
    \end{align*}
\end{proof}

We now prove the main result of this section, which gives an expression of the eigenvalue $\lambda_x$ of a tensor $x \in \F_q^{\mathbf{n}}$ in terms of the number of simple tensors of any one of its slice spaces. 

\begin{Thm}\label{thm:eigenvaluesmultilineargraphdual}
    Assume $m \geq 2$ and let $j \in \llbracket 1,m\rrbracket$ be an integer. Then, for each $x \in \F_q^{\mathbf n}$, the eigenvalue $\lambda_x$ of the multilinear forms graph $\Cay(\F_q^{\mathbf{n}},\S_{\mathbf n})$ associated to $x$ is given by
    \[
        \lambda_x = \frac{q^{n_j}A_1(\slicesp_j(x)^\perp) -  |\S_{\mathbf n(j)}|}{q-1}.
    \]  
\end{Thm}
\begin{proof}
    We will proceed by induction on $m$. Assume first that $m = 2$ and let $x \in \F_q^{n_1 \times n_2}$ be a matrix. For $j =1$, we have $\slicesp_1(x) = \rowsp(x)$ and since the rank on $\F_q^{n_2}$ is either 0 or 1, we have $A_1(\slicesp_1(x)^\perp) = q^{n_2 -\dim \rowsp(x)} -1 = q^{n_2-\rank(x)} -1$, and we obtain the eigenvalues given in Proposition~\ref{prop:eigenvaluesCayleyGraphMatrices}. The same holds for $j = 2$ since $\slicesp_2(x) = \colsp(x)$. Now assume that the statement of the theorem is true for a given $m\geq 2$. Let $\mathbf n = (n_1,\dots,n_m,n_{m+1})$ be a sequence of positive integers. Let $j,\ell \in \llbracket 1,m+1\rrbracket$ be distinct integers. For each tensor $y \in \F_{q}^{\mathbf n(\ell)}$, denote by $\lambda_y$ the eigenvalue of $\Cay(\F_{q}^{\mathbf n(\ell)},\S_{\mathbf n(\ell)})$ associated to $y$. Then by Theorem~\ref{thm:eigenvaluesmultilinearformsgraphrecursive} and Lemma~\ref{lemma:sumofA1overslicespace} we have 
    \begin{align*}
        (q-1)\lambda_x 
        &= -|\S_{\mathbf n(\ell)}| + q^{(n_{\ell} - \dim \slicesp_{\ell}(x))}\sum_{y \in \slicesp_{\ell}(x)} \lambda_{y}\\
        &= -|\S_{\mathbf n(\ell)}| + \frac{q^{(n_{\ell} - \dim \slicesp_{\ell}(x))}}{q-1}\sum_{y \in \slicesp_{\ell}(x)}\left(q^{n_j}A_1(\slicesp_j(y)^\perp) - |\S_{\mathbf n(j,\ell)}|\right)\\
        &= -|\S_{\mathbf n(\ell)}| - \frac{q^{n_{\ell}}}{q-1}|\S_{\mathbf n(j,\ell)}| +  \frac{q^{(n_j + n_{\ell} - \dim \slicesp_{\ell}(x))}}{q-1}\sum_{y \in \slicesp_{\ell}(x)}A_1(\slicesp_j(y)^\perp) \\
        &= \frac{1-q^{n_j}-q^{n_{\ell}}}{q-1} |\S_{\mathbf n(j,\ell)}| +  \frac{q^{n_j}}{q-1} \left(\left|\S_{\mathbf n(j,\ell)}\right| + (q-1) A_1\left(\left(\slicesp_{j}\left(x\right)\right)^\perp\right)\right) \\
        &= \frac{1-q^{n_{\ell}}}{q-1} |\S_{\mathbf n(j,\ell)}| +    q^{n_j}A_1\left(\left(\slicesp_{j}\left(x\right)\right)^\perp\right) \\
        &= - |\S_{\mathbf n(j)}| +    q^{n_j}A_1\left(\left(\slicesp_{j}\left(x\right)\right)^\perp\right),
    \end{align*}
    which proves the statement.
\end{proof}

\begin{Rq}
    Using the group isomorphism $\F_q^{\mathbf n} \simeq \F_p^{n_1\times\cdots \times n_m \times\delta}$, \cite[Corollary 3.4.1]{DhankharKadyan2026EigengaluesIntegralCayley} can be applied to yield a formula for $\lambda_x$, for any tensor $x \in \F_q^{\mathbf n}$, as a function of the number of elements of rank one in an $\F_p$-vector subspace of $\F_q^{\mathbf n}$, while Theorem~\ref{thm:eigenvaluesmultilineargraphdual} provides a formula for $\lambda_x$ as a function of the number of elements of rank one in an $\F_q$-vector subspace (of lower dimension). 
\end{Rq}

\begin{Rq}
    Since the adjacency matrix of the multilinear forms graph has null trace, we have $\sum_{x \in \F_{q}^{\mathbf n}} \lambda_x = 0$ and thus Theorem~\ref{thm:eigenvaluesmultilineargraphdual} implies that for each $j \in \llbracket 1,m\rrbracket$ we have
    \[
        \sum_{x \in \F_{q}^{\mathbf n}} A_1(\slicesp_j(x)^\perp) = q^{(n_1\cdots n_{j-1}n_{j+1}\cdots n_m -1)n_j}|\S_{\mathbf n(j)}|.
    \]
\end{Rq}

Using Theorem~\ref{thm:eigenvaluesmultilineargraphdual}, we can state the spectrum of the graph $\Cay(\F_{q}^{\mathbf n},\S_{\mathbf n})$ in terms of the possible cardinalities of the intersection between $\S_{\mathbf n(j)}$ and the dual of tensors $j$-th slice space.  

\begin{Corol}\label{corol:eigenvaluesmultilinearformscountingspaces}
    Let $\mathbf n = (n_1,\dots,n_m)$ be a sequence of positive integers with $m \geq 2$ and let $j \in \llbracket 1,m\rrbracket$ be an integer. Then for each $a \in \N_0$, the integer
    \[
        \lambda_a = \frac{q^{n_j}a - |\S_{\mathbf n(j)}|}{q-1}
    \]
    is an eigenvalue of $\Cay(\F_q^{{\mathbf n}},\S_{\mathbf n})$ if and only if there exists a subspace $\CCC$ of $\F_q^{{\mathbf n}(j)}$ such that $\codim \CCC \leq n_j$ and $A_1(\CCC) = a$. In this case, $\lambda_a$ has multiplicity $1$ if $a = |\S_{\mathbf n(j)}|$ and else
    \[
        m_a = \sum_{k = 1}^{n_j} \gbc{k}{n_j}{q} \prod_{t=0}^{k-1}(q^{n_j}-q^t) \left| \left\{ \CCC \leq \F_q^{{\mathbf n}(j)} \ : \ \codim \CCC = k , A_1(\CCC) = a  \right\} \right|.
    \]
\end{Corol}
\begin{proof}
    Firstly, note that since the expression of $\lambda_a$ is affine in $a$, the eigenvalues for distinct $a$ are pairwise different. Additionally, for each subspace $\DDD$ of $\F_q^{{\mathbf n}(j)} $, there exists an $x \in \F_q^{{\mathbf n}}$ such that $\DDD = \slicesp_j(x)$ if and only if $\dim \DDD \leq n_j$, in which case we have
    \[
        \left|\left\{x \in \F_q^{{\mathbf n}} \ : \ \DDD = \slicesp_j(x) \right\}\right| = \left|\left\{M \in \F_q^{n_j \times \dim(\DDD)} \ : \ \rank M = \dim \DDD\right\} \right|.
    \]
    Indeed, choosing an $x \in \F_q^{{\mathbf n}} $ such that $\DDD = \slicesp_j(x) $ corresponds to choosing a generating sequence of $\DDD$ of length $n_j$, which is exactly the number of rank $(\dim \DDD)$ matrices in $\F_q^{n_j \times \dim \DDD}$. By Theorem~\ref{thm:eigenvaluesmultilineargraphdual}, the multiplicity of $\lambda_a$ for any $a$ is $1$ for $a = |\S_{\mathbf n(j)}|$, as this is the eigenvalue associated to the zero-tensor, and else
    \begin{align*}
        &\left|\left\{x \in \F_q^{{\mathbf n}} \ : \ A_1(\slicesp_j(x)^\perp) = a \right\}\right|\\
        &=\sum_{k = 1}^{n_j}\sum_{\substack{\DDD \leq \F_q^{{\mathbf{n}(j)}}\\ \dim \DDD =k \\A_1(\DDD^\perp) = a}}\left|\left\{x \in \F_q^{{\mathbf n}} \ : \ \slicesp_j(x) = \DDD \right\}\right|\\
         &= \sum_{k = 1}^{n_j} \left|\left\{\DDD \leq \F_q^{{\mathbf{n}(j)}} \ : \ \dim \DDD = k, A_1(\DDD^\perp) = a\right\} \right|  \left|\left\{M \in \F_q^{n_j \times k} \ : \ \rank M = k\right\} \right|.\\
         & =  \sum_{k = 1}^{n_j} \left|\left\{\DDD \leq \F_q^{{\mathbf{n}(j)}} \ : \ \dim \DDD = k, A_1(\DDD^\perp) = a\right\} \right| \gbc{k}{n_j}{q} \prod_{t=0}^{k-1}(q^{n_j}-q^t) 
    \end{align*}
    Since $ \DDD\mapsto \DDD^\perp$ is a bijection of the set of subspaces of $\F_q^{{\mathbf n}(j)}$ such that $\dim \DDD = \codim (\DDD^\perp)$, we have the required statement.
\end{proof}

\subsection{Eigenvalues range and order}

In this section, we will study the spectrum of a multilinear forms graph. It is known that the largest eigenvalue of a Cayley graph generated by a subset $S$ of a finite abelian group $G$ is $|S|$. The following proposition characterises the second largest eigenvalue of the multilinear forms graph.

\begin{Prop}\label{prop:secondbigesteigenvalue}
     Assume $m\geq 2$ and let $x \in\S_{\mathbf n}$. For each $y \in \F_q^{\mathbf n}$, denote by $\lambda_y$ the eigenvalue of $\Cay(\F_{q}^{\mathbf n},\S_{\mathbf n})$ associated to $y$. Then we have
    \[
        \lambda_0 > \lambda_{x} > \max\{ \lambda_y  :  y \in \F_q^{\mathbf n} , \rank(y) \geq 2\}.    \]
\end{Prop}
\begin{proof}
    We prove the statement by induction on $m$. By Proposition~\ref{prop:eigenvaluesCayleyGraphMatrices}, the statement is true for $m= 2$. 
    Let $m\geq 2$ and assume the statement holds for any sequence of positive integers 
    $(n_1,\dots,n_m)$.  Let $\mathbf n = (n_1,\dots,n_{m+1})$ be a sequence of positive integers, and let $x_1 \in \S_{\mathbf n}$ and $x_2 \in \F_q^{\mathbf n}$ have $\rank(x_2) \geq 2$. Then there exists $j \in \llbracket 1,m\rrbracket$ such that $\dim \slicesp_j(x_2) \geq 2$. Denote by ${\lambda}_{y}$ the eigenvalue of $\Cay(\F_q^{{\mathbf{n}(j)}},\S_{\mathbf n(j)})$ associated to $y \in \F_q^{{\mathbf n}(j)}$. Finally, let $s \in \S_{\mathbf n(j)}$ be a simple tensor. By Theorem~\ref{thm:eigenvaluesmultilinearformsgraphrecursive}, we have 
    \[
        \begin{array}{ll}
             &q^{-n_j}\left((q-1)\lambda_0 +  |\S_{\mathbf n(j)}| \right) = {\lambda}_0,\\[0.3cm]
             &q^{-n_j}\left((q-1)\lambda_{x_1} +  |\S_{\mathbf n(j)}| \right) = q^{-1}{\lambda}_{0} + (1-q^{-1}){\lambda}_{ s},\\
             \text{and}& \displaystyle q^{-n_j}\left((q-1)\lambda_{x_2} +  |\S_{\mathbf n(j)}| \right) = q^{-k}\left({\lambda}_0 + \sum_{y \in \slicesp_j(x_2)\backslash \{0\}}{\lambda}_{y}\right) \leq q^{-k}{\lambda}_0 + (1-q^{-k}){\lambda}_{ s},
        \end{array}
    \]
    by the induction hypothesis, where $k = \dim \slicesp_j(x_2) \geq 2$. And since ${\lambda}_0 > \lambda_{s}$, again by hypothesis, we have 
    ${\lambda}_0 > q^{-1}{\lambda}_{0} + (1-q^{-1}){\lambda}_{ s} >q^{-k}{\lambda}_0 + (1-q^{-k}){\lambda}_{ s}$, which concludes the proof.
\end{proof}

Since the eigenvalue associated to a tensor depends only on its class under the action of 
$\GL_{\mathbf{n}}(\F_q)$, we can give an explicit formula for the eigenvalue associated to any element of rank one (the second-largest eigenvalue) of any multilinear forms graph as follows.

\begin{Corol}\label{corol:valuesecondlargestev}
    Let $\mathbf n= (n_1,\dots,n_m)$ be a sequence of positive integers with $m\geq 2$. Let $x \in \S_{\mathbf n}$ be a simple tensor of $\F_{q}^{\mathbf n}$. Then the eigenvalue of $\Cay(\F_{q}^{\mathbf n},\S_{\mathbf n})$ associated to $x$ is 
    \[
        \lambda_x = |\S_{\mathbf n}|- q^{n_1 + \cdots + n_{m-1} + n_m -m+1}.
    \]
\end{Corol}
\begin{proof}
    We may assume without loss of generality that $x = e_1^{(n_1)} \otimes \cdots \otimes e_1^{(n_m)}$. Note that 
    \begin{align*}
        \S_{\mathbf n(m)} \cap \slicesp_m(x)^\perp
        &= \S_{\mathbf n(m)} \cap \left\langle e_1^{(n_1)} \otimes \cdots \otimes e_1^{(n_{m-1})}\right\rangle^\perp\\
        &= \S_{\mathbf n(m)} \cap  \left\{y \in \F_{q}^{\mathbf n(m)} \ : \ y_{(1,\dots,1)} = 0\right\}\\
        &= \left\{u_1 \otimes \cdots \otimes u_{m-1} \in \S_{\mathbf n(m)} \ : \ u_{1,1}\cdots u_{m-1,1} = 0\right\}.
    \end{align*}
    Therefore we have 
    \begin{align*}
        &A_1(\slicesp_m(x)^\perp)\\
        &= |\S_{\mathbf n(m)}| - \left|\left\{u_1 \otimes \cdots \otimes u_{m-1} \in \S_{\mathbf n(m)} \ : \ u_{1,1}\cdots u_{m-1,1} \neq 0\right\}\right|\\
        &= |\S_{\mathbf n(m)}| - {(q-1)^{-(m-2)}} \left|\left\{(u_1,\dots, u_{m-1}) \in \prod_{j=1}^{m-1} \F_q^{n_j} \backslash \{0\}, \forall j \in \llbracket 1,m-1\rrbracket , u_{j,1} \neq 0\right\}\right|\\
        &= |\S_{\mathbf n(m)}| - {(q-1)^{-(m-2)}}  \prod_{j=1}^{m-1} ((q-1)q^{n_j-1})\\
        &= |\S_{\mathbf n(m)}| - (q-1)q^{n_1 + \cdots + n_{m-1} -m+1}.
    \end{align*} 
    Consequently, by Theorem~\ref{thm:eigenvaluesmultilineargraphdual}, we have
    \begin{align*}
        \lambda_x 
        = \frac{(q^{n_m}-1)|\S_{\mathbf n(m)}| - (q-1)q^{n_1 + \cdots + n_{m-1} + n_m -m+1}}{q-1}
        = |\S_{\mathbf n}|- q^{n_1 + \cdots + n_{m-1} + n_m -m+1}.
    \end{align*}
\end{proof}

It is known that one can bound the diameter $D(\Gamma)$ of a regular graph $\Gamma$ with eigenvalues $\theta_0 > \theta_1 > \cdots > \theta_r$ using the ratio between the two largest eigenvalues, that is 
\begin{equation}
        D(\Gamma) \leq \frac{\log(|V(G)|-1)}{\log(\theta_0/\theta_1)},
        \label{eq:diameterboundeigenvalue}
\end{equation}
see \cite{Chung1989DiameterEigenvalue} for more details. In the case of the multilinear forms graph $\Gamma = \Cay(\F_{q}^{\mathbf n},\S_{\mathbf n}))$, since the geodesic distance of the graph is exactly the rank distance, the diameter of the graph is exactly the maximum rank in the space of tensors, that is 
\[
    D(\Cay(\F_{q}^{\mathbf n},\S_{\mathbf n})) = \max\{ \rank(x) \ : \ x \in \F_{q}^{\mathbf n}\}.
\]
We can observe here that the bound on the diameter that we can obtain using equation (\ref{eq:diameterboundeigenvalue}) does not improve the bounds given in \cite{HOWELL19789} on the maximum rank; see Proposition~\ref{prop:howellbound}. Indeed, let $m \geq 3$ and $\mathbf n = (n_1,\dots,n_m)$ be such that $n_j \geq 2$ for each $j \in \llbracket 1,m\rrbracket$. Since $(1-q^{-1})^{m-1}(1+q^{-1})\leq (1-q^{-2})^m$, we have
\[
    q^{n_1 + \cdots + n_m - m} (q-1)^{m-1}(q+1) \leq \prod_{j=1}^m (q^{n_j}-1).
\]
In turn, we obtain 
\[
    \frac{q^{n_1 + \cdots + n_m - m}}{|\S_{\mathbf n}| - q^{n_1 + \cdots + n_m - m+1}} \leq 1,
\]
and hence, by Proposition~\ref{prop:secondbigesteigenvalue} and Corollary~\ref{corol:valuesecondlargestev}, we get that
\[
    \log_q\left(\frac{\lambda_0}{\lambda_x}\right)\leq 1,
\]
for any $x \in \S_{\mathbf n}$, thus proving that the upper bound (\ref{eq:diameterboundeigenvalue}) is at least $\log_q(q^{n_1\cdots n_m}-1)$.

Using Theorem~\ref{thm:eigenvaluesmultilineargraphdual}, we can also provide more information on the possible integers that can be found in the spectrum of the multilinear forms graph. Let us start with a coding theory lemma on the existence of minimum distance $2$ tensor codes that uses the optimal codes found by \cite{ROTHTensorCodesForRankMetric}, and deduce the value of the least eigenvalue of the spectrum. 
\begin{Lemma}\label{lemma:existancemindistancetwocodes}
    Let $\mathbf n = (n_1,\dots,n_m)$ be a sequence of positive integers with $m \geq 2$. Denote the greatest integer in $\mathbf n$ by $\eta = \max (n_1,\dots,n_m)$. Then there exists an (optimal) $[{\mathbf n},n_1\cdots n_m- \eta,2]_q$ tensor code.
\end{Lemma}
\begin{proof}
    By \cite[Theorem 5]{ROTHTensorCodesForRankMetric}, there exists $\CCC$ an $[\eta \times \cdots \times \eta, \eta^m - \eta,2]_q$ linear tensor code. Moreover, the shortened subcode
    \[
        \CCC' = \left\{ x \in \CCC \ : \ \forall i \in \llbracket 1,\eta\rrbracket^m \backslash \left(\prod_{j = 1}^m \llbracket 1,n_j\rrbracket\right), x_i = 0 \right\}
    \]
    has dimension $\dim \CCC' \geq \dim \CCC - (\eta^m - n_1\cdots n_m) = n_1 \cdots n_m - \eta$. By the Singleton-like bound (\ref{eq:singletonlike}), we have $\dim \CCC' \leq n_1\cdots n_m - \eta$, and hence $\dim \CCC' = n_1\cdots n_m - \eta$. Finally, since $\CCC' \subseteq \CCC$ and since $\CCC'$ can be identified with a subspace of $\F_q^{\mathbf n}$, the code $\CCC'$ can be identified with an $[{\mathbf n},n_1\cdots n_m- \eta,2]_q$ linear tensor code.
\end{proof}

\begin{Corol}
    Let $\mathbf n = (n_1,\dots,n_m)$ be a sequence of positive integers with $m\geq 2$. For each $x \in \F_q^{\mathbf n}$, denote by $\lambda_x$ the eigenvalue of $\Cay(\F_{q}^{\mathbf n},\S_{\mathbf n})$ associated to $y$. Let $\ell \in \llbracket 1,m\rrbracket$ be an index such that we have $n_{\ell} = \max(n_1,\dots,n_m)$. Then we have
    \begin{equation}
       \lambda_x \in \left\{ |\S_{\mathbf n}| -q^{n_{\ell}}t \ : \ t \in \N_0 \right\} \cap \left[ \frac{-|\S_{\mathbf n(\ell)}|}{q-1} \ , \  |\S_{\mathbf n}|\right].
       \label{eq:eigenvaluelocation}
    \end{equation}
    The least eigenvalue of the graph $\Cay(\F_{q}^{\mathbf{n}},\S_{\mathbf n})$ is 
    \begin{equation}
        \frac{-|\S_{\mathbf n(\ell)}|}{q-1},
        \label{eq:eigenvaluemin}
    \end{equation}
    and is the eigenvalue associated to any $x \in \F_{q}^{\mathbf n}$ such that $A_1(\slicesp_j(x)^\perp) = 0$. In particular for $\mathbf n = (n_1,n_2,n_3)$ with $n_1 \geq n_2 \geq n_3$, this is the eigenvalue of  $\Cay(\F_q^{n_1\times n_2 \times n_3},\S_{n_1 \times n_2 \times n_3})$ associated to any $3$-order tensor $x \in \F_{q}^{n_1\times n_2 \times n_3}$ such that $\slicesp_1(x)$ is an $[n_2\times n_3, n_2(n_3-d+1),d]_q$ MRD code, namely 
    \[
        \lambda_x = -\frac{(q^{n_2}-1)(q^{n_3}-1)}{(q-1)^2}.
    \]
\end{Corol}
\begin{proof}
   Since for each $x \in \F_q^{\mathbf n}$ and each $j \in \llbracket 1,m\rrbracket$, the value $A_1(\slicesp_j(x)^\perp)$ is a multiple of $(q-1)$, hence Theorem~\ref{thm:eigenvaluesmultilineargraphdual} gives
    \begin{align*}
        \lambda_x  \in \left\{ |\S_{\mathbf n}| -q^{n_j}t \ : \ t \in \N_0 \right\} \cap \left[ \frac{-|\S_{\mathbf n(j)}|}{q-1} \ , \  |\S_{\mathbf n}|\right].
    \end{align*}
    Since this is true for each $j \in \llbracket 1,m \rrbracket$, and since $\min_{j \in \llbracket 1,m\rrbracket} |\S_{\mathbf n(j)}| = |\S_{\mathbf n(\ell)}|$ we have 
    \[
        \lambda_x \in \bigcap_{j = 1}^m \left( \left\{ |\S_{\mathbf n}| -q^{n_j}t \ : \ t \in \N_0 \right\} \cap \left[ \frac{-|\S_{\mathbf n(j)}|}{q-1} \ , \  |\S_{\mathbf n}|\right]\right), 
    \]
    and this intersection is exaclty the set given in (\ref{eq:eigenvaluelocation}).
    Assume now that $n_1 \geq \cdots \geq n_m$, i.e. $\ell = 1$. The other cases are similar. By Lemma~\ref{lemma:existancemindistancetwocodes}, there exists an $[n_2 \times \cdots \times n_{m},n_2\cdots n_{m}-n_{2},2]_q$ tensor code $\CCC$. Since $n_{1} \geq n_{2}$ there exists an $[n_2 \times \cdots \times n_{m},n_2\cdots n_{m}-n_{1},d]_q$ linear tensor code $\tilde \CCC$ with $d \geq 2$. Therefore, there exists $x \in \F_{q}^{\mathbf n}$ such that $\slicesp_1(x)^{\perp} = \tilde \CCC$. In particular $A_1(\slicesp_1(x)^{\perp}) = 0$ and by Theorem~\ref{thm:eigenvaluesmultilineargraphdual} we have $(q-1)\lambda_x = -|\S_{\mathbf n(1)}|$. Finally, assume $\mathbf n = (n_1,n_2,n_3)$ with $n_1 \geq n_2 \geq n_3$, and $x \in \F_{q}^{n_1 \times n_2 \times n_3}$ is such that $\slicesp_1(x)$ is an $[n_2\times n_3, n_2(n_3-d+1),d]_q$ MRD codes. If $d = 1$, then $\slicesp_1(x) = \F_{q}^{n_2 \times n_3}$ and $\slicesp_1(x)^\perp = \{0\}$. Otherwise, by \cite[Theorem 5.5]{DELSARTE1978226}, the dual code of an MRD code of minimum distance at least two is also an MRD code of minimum distance at least two, thus we have $A_1(\slicesp_1(x)^\perp) = 0$. Therefore, $x$ is associated to the least eigenvalue of the trilinear forms graph, which concludes the proof.
\end{proof}

\section{Eigenvalues of the trilinear forms graph of 2x3x3 tensors.}
\label{sec:eigenvalues233}

In this section, we will compute the spectrum of the graph $\Cay(\F_q^{2 \times 3 \times 3},\S_{2\times 3 \times 3})$ and derive bounds on the possible cardinality of a tensor code in $\F_q^{2 \times 3 \times 3}$ of given minimum distance. In general, the computation of each eigenvalue multiplicity of the graph turns out to be a difficult task as it can be reduced to computing the number of matrix subspaces with sufficiently large dimension that have a given number of rank one elements, as seen in Corollary~\ref{corol:eigenvaluestrilinearformsgraphdual}, or the number of matrix subspaces with sufficiently low dimension that have a given rank distribution as seen in Corollary~\ref{corol:eigenvaluestrilinearformsgraph}. This task becomes much simpler with the classification of matrix subspaces up to equivalence.

In \cite{LavrauwSheekey2015CanonicalFormsTensors} and \cite{LavrauwSheekey2017ClassificationSubspacesFqr}, Lavrauw and Sheekey classified the subspaces of $\F^{2 \times 3}$ under the action of $\GL_2(\F)\times \GL_3(\F)$ for an arbitrary field $\F$ and the classification of orbits of the tensors in $\F^{2 \times 3 \times 3}$ under the action of $\GL_{2 \times 3 \times 3}(\F)$. In the case that $\F$ is the finite field $\F_q$, they showed that there are twenty non-trivial orbits of subspaces $\F_q^{2 \times 3}$, which they denoted by $o_1$, $o_2$, $o_3$, $o_4$, $o_4^T$, $o_5$, $o_6$, $o_7$, $o_7^T$, $o_8$, $o_9$, $o_{10}$, $o_{11}$, $o_{11}^T$, $o_{12}$, $o_{13}$, $o_{14}$, $o_{15}$, $o_{16}$, $o_{17}$. Two orbits $o_{10}, o_{15}$ have representatives corresponding to pairs $(u,v) \in \F_q^2$ such that $v \neq 0$ and $v\lambda^2 - uv\lambda  - 1 \neq 0$ for each $\lambda \in \F_q$. In the case of $o_{17}$ there is also a parameter in the form of a triple $(\alpha,\beta,\gamma) \in \F_q^3$ such that $\lambda^3 + \gamma \lambda^2 - \beta \lambda + \alpha \neq 0$ for each $\lambda \in \F_q$. The representative of each orbit is given in Table \ref{tab:SizeAndRankDistribSpaces}.

As seen above, two tensors with equivalent slice spaces will have the same eigenvalues. Thanks to the computation of the number of rank one matrices in each subspace of each orbit that was given in \cite{LavrauwSheekey2017ClassificationSubspacesFqr}, we can compute the rank distribution and thus the eigenvalue associated to each orbit. With the orbits sizes computations below and summed up in Table \ref{tab:SizeAndRankDistribSpaces}, one can obtain the spectrum of the graph in its entirety as stated in Tables \ref{tab:spectrumFq233} and \ref{tab:spectrumFq233q2}. 
 
As we show in this section, using the results of \cite{AbiadCoutinhoFiol2019kindependancenumber} and in particular Theorem~\ref{thm:ratiobound23}, we can derive bounds on the independence number of the graph $\Cay(\F_q^{2 \times 3 \times 3},\S_{2\times 3 \times 3})$. In particular, we obtain the following statement.

\begin{Thm}\label{thm:sumup233}
    Let $\CCC$ be subset of $\F_q^{2 \times 3 \times 3}$. Assume that $|\CCC| \geq 2$ and denote its minimum rank distance by $d := \min\{\rank(C-C') \ : \ C,C'\in \CCC,C\neq C'\}$. 
    \begin{itemize}
        \item Assume that $d =3$. Then
        \begin{itemize}
            \item if $q =2$, we have $|\CCC| < 2^{12}\left(\frac{22}{57}\right)$,
            \item and if $q \geq 3$, we have $|\CCC| < q^{12}(1-4q^{-1} + 12q^{-2})$.
        \end{itemize}
        \item Assume that $d=4$. Then 
        \begin{itemize}
            \item if $q = 2$, we have $|\CCC|\leq 2^{9}\left(\frac{3780}{6426}\right)$,
            \item if $q = 3$, we have $|\CCC|\leq 3^{9}\left(\frac{1017}{1225}\right)$,
            \item and if $q = 4$, we have $|\CCC| \leq 4^9\left(\frac{1688}{1725}\right)$.
        \end{itemize}
    \end{itemize}
\end{Thm}

The bounds listed above are obtained by the eigenvalue method and are strict improvements on the bounds mentioned in Section \ref{sec:rankmetriccodesandtensorcodes}. First, we provide details on the link between the orbit sizes of the classification of subspaces and the orbit size of the tensors under the action of the general linear group. 

\begin{Lemma}\label{lemma:orbitsofcodesvstensors}
    Let $j \in \llbracket 1,m\rrbracket$ be an integer. Let $x \in \F_q^{\mathbf{n}}$ be a tensor. Then the orbit of $x$ under the action of the group $\GL_{\mathbf{n}}(\F_q)$ is 
    \[
        \GL_{\mathbf n}(\F_q)\cdot x = \left\{ y \in \F_q^{\mathbf{n}} , \slicesp_j(y) \in \GL_{{\mathbf{n}(j)}}(\F_q)\cdot \slicesp_j(x) \right\}.
    \]
    In particular, we have
    \[
         \left|\GL_{\mathbf n}(\F_q)\cdot x\right| =  \gbc{\dim \slicesp_j(x)}{n_j}{q} \prod_{t=0}^{\dim \slicesp_j(x)-1}(q^{n_j}-q^t)\left|\GL_{{\mathbf n}(j)}(\F_q)\cdot \slicesp_j(x)\right|. 
    \]
\end{Lemma}
\begin{proof}
    Let $x,y \in \F_q^{\mathbf{n}}$. Suppose first that $y \in \GL_{\mathbf n}(\F_q)\cdot x$. By definition, there exists $P \in \GL_{\mathbf n}(\F_q)$ such that $P\cdot x = y$. In particular, we have 
    \begin{align*}
        \slicesp_j(y) &= \{(I_{n_1},\dots,I_{n_{j-1}},u_j,I_{n_{j+1}},\dots,I_{n_m}) \cdot y  \ : \ u_j \in \F_q^{n_j} \}\\
        &= \{(P_{1},\dots,P_{{j-1}},P_ju_j,P_{{j+1}},\dots,P_{m}) \cdot x  \ : \ u_j \in \F_q^{n_j} \}\\
        &= \{(P_{1},\dots,P_{{j-1}},v_j,P_{{j+1}},\dots,P_{m}) \cdot x  \ : \ v_j \in \F_q^{n_j} \}\\
        &= \{(P_{1},\dots,P_{{j-1}},1,P_{{j+1}},\dots,P_{m}) \cdot (I_{n_1},\dots,I_{n_{j-1}},v_j,I_{n_{j+1}},\dots,I_{n_m}) \cdot x \ : \ v_j \in \F_q^{n_j} \}\\
        &= \{(P_{1},\dots,P_{{j-1}},1,P_{{j+1}},\dots,P_{m}) \cdot z \ : \ z \in \slicesp_j(x) \}.
    \end{align*}
    Thus we have $\slicesp_j(y) \in \GL_{{\mathbf n}(j)}(\F_q)\cdot \slicesp_j(x)$. 

    Conversely, if there exists $(P_{1},\dots,P_{{j-1}},P_{{j+1}},\dots,P_{m}) \in \GL_{{\mathbf n}(j)}(\F_q) $ such that 
    $$\slicesp_j(y) = \{(P_{1},\dots,P_{{j-1}},P_{{j+1}},\dots,P_{m}) \cdot z \ : \ z \in \slicesp_j(x) \},$$
    then $\slicesp_j(y) =\slicesp_j((P_{1},\dots,P_{{j-1}},I_{n_j},P_{{j+1}},\dots,P_{m}) \cdot x)$
    and thus there exists an invertible matrix $P_j \in \GL_{n_j}(\F_q)$ such that
    $y =(P_{1},\dots,P_{{j-1}},P_j,P_{{j+1}},\dots,P_{m}) \cdot x)$. Therefore, we obtain 
    $$y = (I_{n_1},\dots,I_{n_{j-1}},P_j,I_{n_{j+1}},\dots,I_{n_m})\cdot (P_{1},\dots,P_{{j-1}},I_{n_j},P_{{j+1}},\dots,P_{m}) \cdot x,$$ hence $y = P\cdot y$ which concludes the first part of the proof. Regarding the cardinality of the orbit of $x$, the first part establishes that, to count every element one and only one time in the orbit, it suffices to count the possible choices for a $j$-th slice space, and to fix a basis of that space, hence 
    \[
        \left|\GL_{\mathbf n}(\F_q)\cdot x\right| = \left|\left\{ M \in \F_q^{{(\dim \slicesp_j(x)) \times  n_j}} \ : \ \rank M = \dim \slicesp_j(x)\right\} \right| \left|\GL_{{\mathbf n}(j)}(\F_q)\cdot \slicesp_j(x)\right|.
    \]
    The results follows from the value of the number of matrices of given rank. 
\end{proof}

\subsection{Stabiliser sizes}

In this section, we will compute the stabiliser sizes of the orbits of subspaces of $\F_{q}^{2 \times 3}$. They will later be used with the orbit-stabiliser theorem to compute each orbit size. Let us start by computing the stabiliser sizes of fifteen non-trivial orbits of subspaces of $2\times 3$ up to equivalence. 

\begin{Lemma} \label{lemma:sizeofstabilisers}
    Regarding the action of $\GL_{2\times 3} (\F_q)$ on the subspaces of $\F_q^{2 \times 3}$, we have 
    \begin{itemize}
        \item \underline{$o_2$:} $\left|\Stab\left(\left\langle   
        \left[\begin{smallmatrix}
        1&0&0\\0&0&0
        \end{smallmatrix}\right],
        \left[\begin{smallmatrix}
        0&1&0\\0&0&0
        \end{smallmatrix}\right]
        \right\rangle\right)\right| = q^4(q-1)^5(q+1)$,

        \item \underline{$o_4^T$:} $\left|\Stab\left(\left\langle   
        \left[\begin{smallmatrix}
        1&0&0\\0&0&0
        \end{smallmatrix}\right],
        \left[\begin{smallmatrix}
        0&0&0\\1&0&0
        \end{smallmatrix}\right]
        \right\rangle\right)\right| =  q^4(q-1)^5(q+1)^2$,

        \item \underline{$o_5$:} $\left|\Stab\left(\left\langle   
        \left[\begin{smallmatrix}
        1&0&0\\0&0&0
        \end{smallmatrix}\right],
        \left[\begin{smallmatrix}
        0&0&0\\0&1&0
        \end{smallmatrix}\right]
        \right\rangle\right)\right| =  2q^2(q-1)^5$,
        
        \item \underline{$o_6$:} $\left|\Stab\left(\left\langle   
        \left[\begin{smallmatrix}
        1&0&0\\0&1&0
        \end{smallmatrix}\right],
        \left[\begin{smallmatrix}
        0&0&0\\1&0&0
        \end{smallmatrix}\right]
        \right\rangle\right)\right| = q^4(q-1)^4$,

        \item \underline{$o_7$:} $\left|\Stab\left(\left\langle   
        \left[\begin{smallmatrix}
        1&0&0\\0&0&1
        \end{smallmatrix}\right],
        \left[\begin{smallmatrix}
        0&1&0\\0&0&0
        \end{smallmatrix}\right]
        \right\rangle\right)\right| =  q^2(q-1)^4$,
    
        \item \underline{$o_{11}$:} $\left|\Stab\left(\left\langle   
        \left[\begin{smallmatrix}
        1&0&0\\0&1&0
        \end{smallmatrix}\right],
        \left[\begin{smallmatrix}
        0&1&0\\0&0&1
        \end{smallmatrix}\right]
        \right\rangle\right)\right|  = q(q-1)^3(q+1) $,

        \item \underline{$o_{3}$:} $\left|\Stab\left(\left\langle   
        \left[\begin{smallmatrix}
        1&0&0\\0&0&0
        \end{smallmatrix}\right],
        \left[\begin{smallmatrix}
        0&1&0\\0&0&0
        \end{smallmatrix}\right],
        \left[\begin{smallmatrix}
        0&0&1\\0&0&0
        \end{smallmatrix}\right]
        \right\rangle\right)\right| =  q^4(q-1)^5(q+1)(q^2+q+1)$,

        \item \underline{$o_{7}^T$:} $\left|\Stab\left(\left\langle   
        \left[\begin{smallmatrix}
        1&0&0\\0&0&0
        \end{smallmatrix}\right],
        \left[\begin{smallmatrix}
        0&0&0\\1&0&0
        \end{smallmatrix}\right],
        \left[\begin{smallmatrix}
        0&0&0\\0&1&0
        \end{smallmatrix}\right]
        \right\rangle\right)\right| =  q^4(q-1)^5$,

        \item \underline{$o_{8}$:} $\left|\Stab\left(\left\langle   
        \left[\begin{smallmatrix}
        1&0&0\\0&0&0
        \end{smallmatrix}\right],
        \left[\begin{smallmatrix}
        0&0&0\\0&1&0
        \end{smallmatrix}\right],
        \left[\begin{smallmatrix}
        0&0&0\\0&0&1
        \end{smallmatrix}\right]
        \right\rangle\right)\right| =  q(q-1)^5(q+1)$,

        \item \underline{$o_{9}$:} $\left|\Stab\left(\left\langle   
        \left[\begin{smallmatrix}
        1&0&0\\0&0&1
        \end{smallmatrix}\right],
        \left[\begin{smallmatrix}
        0&0&0\\1&0&0
        \end{smallmatrix}\right],
        \left[\begin{smallmatrix}
        0&0&0\\0&1&0
        \end{smallmatrix}\right]
        \right\rangle\right)\right| =  q^4(q-1)^4$,

        \item \underline{$o_{11}^T$:} $\left|\Stab\left(\left\langle   
        \left[\begin{smallmatrix}
        1&0&0\\0&0&0
        \end{smallmatrix}\right],
        \left[\begin{smallmatrix}
        0&1&0\\1&0&0
        \end{smallmatrix}\right],
        \left[\begin{smallmatrix}
        0&0&0\\0&1&0
        \end{smallmatrix}\right]
        \right\rangle\right)\right| = q^3(q-1)^5(q+1)$,

        \item \underline{$o_{12}$:} $\left|\Stab\left(\left\langle   
        \left[\begin{smallmatrix}
        1&0&0\\0&0&1
        \end{smallmatrix}\right],
        \left[\begin{smallmatrix}
        0&1&0\\0&0&0
        \end{smallmatrix}\right],
        \left[\begin{smallmatrix}
        0&0&0\\0&1&0
        \end{smallmatrix}\right]
        \right\rangle\right)\right| =  q^3(q-1)^4(q+1)$,
        
        \item \underline{$o_{13}$:} $\left|\Stab\left(\left\langle   
        \left[\begin{smallmatrix}
        1&0&0\\0&1&0
        \end{smallmatrix}\right],
        \left[\begin{smallmatrix}
        0&1&0\\0&0&0
        \end{smallmatrix}\right],
        \left[\begin{smallmatrix}
        0&0&0\\0&0&1
        \end{smallmatrix}\right]
        \right\rangle\right)\right| =  q(q-1)^4$,

        \item \underline{$o_{14}$:} $\left|\Stab\left(\left\langle   
        \left[\begin{smallmatrix}
        1&0&0\\0&0&0
        \end{smallmatrix}\right],
        \left[\begin{smallmatrix}
        0&1&0\\0&1&0
        \end{smallmatrix}\right],
        \left[\begin{smallmatrix}
        0&0&0\\0&0&1
        \end{smallmatrix}\right]
        \right\rangle\right)\right| =  6(q-1)^4$,

        \item \underline{$o_{16}$:} $\left|\Stab\left(\left\langle   
        \left[\begin{smallmatrix}
        1&0&0\\0&1&0
        \end{smallmatrix}\right],
        \left[\begin{smallmatrix}
        0&1&0\\0&0&1
        \end{smallmatrix}\right],
        \left[\begin{smallmatrix}
        0&0&1\\0&0&0
        \end{smallmatrix}\right]
        \right\rangle\right)\right| =  q^3(q-1)^3$.
    \end{itemize}
\end{Lemma}

We give the proof of the statement in Appendix \ref{appendix:stabilisersizes}. In this proof, we describe the structure of each orbit representative stabiliser, and deduce the number of pairs $(P,Q) \in \GL_{2 \times 3}(\F_q)$ in the stabiliser from its structure. The proof only uses direct linear algebra and combinatorial techniques.

This statement alone is sufficient to compute the orbit sizes of all orbits listed above in the action of $\GL_{2 \times 3}(\F_q)$ on the subspaces of $\F_q^{2 \times 3}$ using the orbit-stabiliser formula. In the following, we will compute the orbit sizes of $o_{10}$, $o_{13}$ and $o_{15}$ in a different way. Namely, we can recover the orbit size of $o_{10}$ from the orbit sizes of all other $2$-dimensional subspaces orbits, and in Theorem~\ref{thm:Spectrum233} and Remark \ref{rk:retroactivecomputationoforbitsize}, we will explain how it is possible to compute the spectrum without the knowledge of the orbit size of $o_{15}$ and $o_{17}$, and how to compute those orbit sizes with the knowledge of the spectrum in its entirety.

\subsection{Eigenvalues computation}

Using the results of the stabiliser sizes obtained above, we will now compute the spectrum of the graph $\Cay(\F_{q}^{2\times 3 \times 3},\S_{2 \times 3 \times 3})$ for any prime power $q$. We will later use the spectrum in the computation of the ratio-type bounds on $\alpha_2$ and $\alpha_3$, as can be seen in the next part.

\begin{Prop}\label{prop:orbitsizes23subspaaces}
    Except for the orbit sizes of $o_{15}$ and $o_{17}$, for each orbit of the action of $\GL_2(\F_q) \times \GL_3(\F_q)$ on the subspaces of $\F_q^{2 \times 3}$ of dimension at most 3, the orbit size, the number $A_1(\mathcal C)$ of rank one matrices in $\CCC$, and the value of $q^{5-\dim\CCC}W_\CCC(1,q^{-1})$ of any representative $\CCC \leq \F_q^{2 \times 3}$ are given in Table \ref{tab:SizeAndRankDistribSpaces}.
\end{Prop}
\begin{proof}
    The result is trivial for the orbit of $0$. The results for the orbit $o_1$ and $o_4$ are known since these orbits are the orbits of one-dimensional subspaces generated respectively by a matrix of rank one or of rank two. For all the other orbits except $o_{10}$, $o_{15}$ and $o_{17}$, the size of the stabiliser under the action of $\GL_2(\F_q)\times \GL_3(\F_q)$ has been computed in Lemma~\ref{lemma:sizeofstabilisers}. Therefore, if an orbit $o_i$ has stabiliser of size $N$, then the size of the orbit is 
    \begin{equation}
        |o_i| = \frac{|\GL_2(\F_q)|\cdot |\GL_3(\F_q)|}{N} = \frac{q^4(q-1)^5(q+1)^2(q^2+q+1)}{N}.
        \label{eq:orbitstabequation}
    \end{equation}
    Regarding orbit $o_{10}$, we have 
    \[
        |o_{10}| = \gbc{2}{6}{q} - |o_2| - |o_4^T| - |o_5| - |o_6| - |o_7| - |o_{11}|.
    \]
    We obtain $|o_{10}| = \frac{1}{2}q^2(q-1)^2(q^2+q+1)$.  The number of rank one matrices in the subspaces has been computed in \cite{LavrauwSheekey2015CanonicalFormsTensors} and \cite{LavrauwSheekey2017ClassificationSubspacesFqr}.   
\end{proof}

\begin{table}[ht!]
    \centering
    \renewcommand{\arraystretch}{1.2}
    
    {
    \begin{tabular}{|c|c|c|c|c|} \hline
        \textbf{Orb}&\textbf{Representative $\CCC$}&\textbf{Orbit size}&$A_1(\CCC)$&$q^{5-\dim\CCC}W_\CCC(1,q^{-1})$\\\hline 
         
        $0$ & $\{0\}$ & $1$ & $0$ & $q^5$\\

        {$o_1$} &$\left\langle   
        \left[\begin{smallmatrix}
        1&0&0\\0&0&0
        \end{smallmatrix}\right]
        \right\rangle$ & $(q+1)(q^2+q+1)$ & $q-1$ & $2q^4 -q^3 $\\

        {$o_4$} &$\left\langle   
        \left[\begin{smallmatrix}
        1&0&0\\0&1&0
        \end{smallmatrix}\right]
        \right\rangle$ & $q(q-1)(q+1)(q^2+q+1)$& $0$ & $q^4 + q^3 - q^2$\\
        
        {$o_2$} &$\left\langle   
        \left[\begin{smallmatrix}
        1&0&0\\0&0&0
        \end{smallmatrix}\right],
        \left[\begin{smallmatrix}
        0&1&0\\0&0&0
        \end{smallmatrix}\right]
        \right\rangle$ & $(q+1)(q^2+q+1)$  & $q^2-1$ & $q^4 + q^3 - q^2$\\

        {$o_4^T$} &$\left\langle   
        \left[\begin{smallmatrix}
        1&0&0\\0&0&0
        \end{smallmatrix}\right],
        \left[\begin{smallmatrix}
        0&0&0\\1&0&0
        \end{smallmatrix}\right]
        \right\rangle$ & $q^2+q+1$  & $q^2-1$ &  $q^4 + q^3 - q^2$\\

        {$o_5$} &$\left\langle   
        \left[\begin{smallmatrix}
        1&0&0\\0&0&0
        \end{smallmatrix}\right],
        \left[\begin{smallmatrix}
        0&0&0\\0&1&0
        \end{smallmatrix}\right]
        \right\rangle$ & $\frac 12q^2(q+1)^2(q^2+q+1)$ & $2q-2$ & $4q^3 - 4q^2 + q$\\
        
        {$o_6$} &$\left\langle   
        \left[\begin{smallmatrix}
        1&0&0\\0&1&0
        \end{smallmatrix}\right],
        \left[\begin{smallmatrix}
        0&0&0\\1&0&0
        \end{smallmatrix}\right]
        \right\rangle$ & $(q-1)(q+1)^2(q^2 + q + 1)$& $q-1$ & $3q^3 - 2q^2$\\

        {$o_7$} &$\left\langle   
        \left[\begin{smallmatrix}
        1&0&0\\0&0&1
        \end{smallmatrix}\right],
        \left[\begin{smallmatrix}
        0&1&0\\0&0&0
        \end{smallmatrix}\right]
        \right\rangle$ & $ q^2(q-1)(q+1)^2(q^2+q+1)$ & $q-1$ & $3q^3 - 2q^2$\\

        {$o_{10}$} &$\left\langle   
        \left[\begin{smallmatrix}
        1&u&0\\0&1&0
        \end{smallmatrix}\right],
        \left[\begin{smallmatrix}
        0&1&0\\v&0&0
        \end{smallmatrix}\right]
        \right\rangle$ & $\frac{1}{2}q^2(q-1)^2(q^2+q+1)$ & $0$ & $2q^3 -q$\\

        {$o_{11}$} &$\left\langle   
        \left[\begin{smallmatrix}
        1&0&0\\0&1&0
        \end{smallmatrix}\right],
        \left[\begin{smallmatrix}
        0&1&0\\0&0&1
        \end{smallmatrix}\right]
        \right\rangle$ & $q^3(q-1)^2(q+1)(q^2+q+1)$ & $0$ & $2q^3 -q$ \\

        {$o_{3}$} &$\left\langle   
        \left[\begin{smallmatrix}
        1&0&0\\0&0&0
        \end{smallmatrix}\right],
        \left[\begin{smallmatrix}
        0&1&0\\0&0&0
        \end{smallmatrix}\right],
        \left[\begin{smallmatrix}
        0&0&1\\0&0&0
        \end{smallmatrix}\right]
        \right\rangle$ & $q+1$  & $q^3-1$ & $q^4 + q^2 -q $ \\

        {$o_{7}^T$} &$\left\langle   
        \left[\begin{smallmatrix}
        1&0&0\\0&0&0
        \end{smallmatrix}\right],
        \left[\begin{smallmatrix}
        0&0&0\\1&0&0
        \end{smallmatrix}\right],
        \left[\begin{smallmatrix}
        0&0&0\\0&1&0
        \end{smallmatrix}\right]
        \right\rangle$ & $(q+1)^2(q^2+q+1)$ & $2q^2-q-1$ & $3q^3 -2q^2 $\\

        {$o_{8}$} &$\left\langle   
        \left[\begin{smallmatrix}
        1&0&0\\0&0&0
        \end{smallmatrix}\right],
        \left[\begin{smallmatrix}
        0&0&0\\0&1&0
        \end{smallmatrix}\right],
        \left[\begin{smallmatrix}
        0&0&0\\0&0&1
        \end{smallmatrix}\right]
        \right\rangle$ & $q^3(q+1)(q^2+q+1)$ & $q^2+q-2$ & $ 2q^3 + q^2 -3q +1$\\

        {$o_{9}$} &$\left\langle   
        \left[\begin{smallmatrix}
        1&0&0\\0&0&1
        \end{smallmatrix}\right],
        \left[\begin{smallmatrix}
        0&0&0\\1&0&0
        \end{smallmatrix}\right],
        \left[\begin{smallmatrix}
        0&0&0\\0&1&0
        \end{smallmatrix}\right]
        \right\rangle$ & $(q-1)(q+1)^2(q^2+q+1)$& $q^2-1$ & $ 2q^3 -q$\\

        {$o_{11}^T$} &$\left\langle   
        \left[\begin{smallmatrix}
        1&0&0\\0&0&0
        \end{smallmatrix}\right],
        \left[\begin{smallmatrix}
        0&1&0\\1&0&0
        \end{smallmatrix}\right],
        \left[\begin{smallmatrix}
        0&0&0\\0&1&0
        \end{smallmatrix}\right]
        \right\rangle$ & $ q(q-1)(q+1)(q^2+q+1) $ & 
        $q^2-1$ & $ 2q^3 -q$ \\
        
        {$o_{12}$} &$\left\langle   
        \left[\begin{smallmatrix}
        1&0&0\\0&0&1
        \end{smallmatrix}\right],
        \left[\begin{smallmatrix}
        0&1&0\\0&0&0
        \end{smallmatrix}\right],
        \left[\begin{smallmatrix}
        0&0&0\\0&1&0
        \end{smallmatrix}\right]
        \right\rangle$ & $q(q-1)(q+1)(q^2+q+1)$ & $q^2-1$ & $2q^3 -q$\\
        
        {$o_{13}$} &$\left\langle   
        \left[\begin{smallmatrix}
        1&0&0\\0&1&0
        \end{smallmatrix}\right],
        \left[\begin{smallmatrix}
        0&1&0\\0&0&0
        \end{smallmatrix}\right],
        \left[\begin{smallmatrix}
        0&0&0\\0&0&1
        \end{smallmatrix}\right]
        \right\rangle$ & $q^3(q-1)(q+1)^2(q^2+q+1)$ & $2q-2$ & $ q^3 + 3q^2 - 4q +1$\\

        {$o_{14}$} &$\left\langle   
        \left[\begin{smallmatrix}
        1&0&0\\0&0&0
        \end{smallmatrix}\right],
        \left[\begin{smallmatrix}
        0&1&0\\0&1&0
        \end{smallmatrix}\right],
        \left[\begin{smallmatrix}
        0&0&0\\0&0&1
        \end{smallmatrix}\right]
        \right\rangle$ & $\frac 16q^4(q-1)(q+1)^2(q^2+q+1)$& $3q-3$& $q^3 +4q^2 -6q +2$ \\

        {$o_{15}$} &$\left\langle   
        \left[\begin{smallmatrix}
        1&u&0\\0&1&0
        \end{smallmatrix}\right],
        \left[\begin{smallmatrix}
        0&1&0\\v&0&0
        \end{smallmatrix}\right],
        \left[\begin{smallmatrix}
        0&0&1\\0&0&0
        \end{smallmatrix}\right]
        \right\rangle$ & $\frac{1}{2}q^4(q-1)^2(q+1)(q^2+q+1)$ & $ q-1$ & $q^3 + 2q^2 -2q $ \\ 

        {$o_{16}$} &$\left\langle   
        \left[\begin{smallmatrix}
        1&0&0\\0&1&0
        \end{smallmatrix}\right],
        \left[\begin{smallmatrix}
        0&1&0\\0&0&1
        \end{smallmatrix}\right],
        \left[\begin{smallmatrix}
        0&0&1\\0&0&0
        \end{smallmatrix}\right]
        \right\rangle$ & $ q(q-1)^2(q+1)^2(q^2+q+1)$ & $q-1$ & $q^3 + 2q^2 -2q$\\

        {$o_{17}$} &$\left\langle   
        \left[\begin{smallmatrix}
        1&0&0\\0&1&0
        \end{smallmatrix}\right],
        \left[\begin{smallmatrix}
        0&1&0\\0&0&1
        \end{smallmatrix}\right],
        \left[\begin{smallmatrix}
        0&0&1\\\alpha&\beta&\gamma
        \end{smallmatrix}\right]
        \right\rangle$ & $\frac{1}{3}q^4(q-1)^3(q+1)^2$ &  $0$ & $q^3 + q^2 -1$ \\\hline
    \end{tabular}
   }
    \renewcommand{\arraystretch}{1}
    \caption{Orbit size and rank-distribution of subspaces of $\F_q^{2 \times 3}$ of dimension at most $3$}
    \label{tab:SizeAndRankDistribSpaces}
\end{table}

Denote by $\theta_i$ the $i$-th eigenvalue of the graph $\Cay(\F_q^{2 \times 3 \times 3},\S_{2 \times 3 \times 3})$ and by $m_i$ its multiplicity. From the twelve different values of the normalised weight enumerator given in Table \ref{tab:SizeAndRankDistribSpaces}, we can compute the eigenvalues as follows.

\begin{Thm}\label{thm:Spectrum233}
    The spectrum of $\Cay(\F_q^{2 \times 3 \times 3},\S_{2\times 3 \times 3})$ is given by Table \ref{tab:spectrumFq233}.
\end{Thm}
\begin{proof}
    By Corollary~\ref{corol:eigenvaluestrilinearformsgraph}, the eigenvalue associated to an element $x \in \F_q^{2 \times 3 \times 3}$ is 
    \begin{align*}
        (q-1)^2\lambda_x &= (q^2-1)(q^3-1)^2 - q^{2+3+3} + q^{2+3+3-\dim \slicesp_3(x)} \sum_{r = 0,1,2} A_{r}(\slicesp_3(x))q^{-r}\\
        &= -q^6 -2q^5 + 2q^3 + q^2 -1 + q^3q^{5-\dim\slicesp_3(x)}W_{\slicesp_3(x)}(1,q^{-1}).
    \end{align*}
    Therefore, we have the values $\theta_0,\dots, \theta_{11}$ given in Table \ref{tab:spectrumFq233} from the values of $q^{5-\dim\slicesp_3(x)}W_{\slicesp_3(x)}(1,q^{-1})$ for any $x \in \F_q^{\mathbf n}$ given in Table \ref{tab:SizeAndRankDistribSpaces}. Moreover, by Proposition~\ref{prop:orbitsizes23subspaaces}, the respective multiplicities $m_{0},\cdots, m_{9}$ of the eigenvalues $\theta_0,\dots,\theta_{9}$ can be obtained using the orbit sizes (without using $o_{15}$ and $o_{17}$) given by Table \ref{tab:SizeAndRankDistribSpaces} with Lemma~\ref{lemma:orbitsofcodesvstensors}. Hence, one can check that the multiplicities of the eigenvalues $\theta_0,\dots,\theta_{9}$ are given in Table \ref{tab:spectrumFq233}. 
    The remaining two multiplicities $m_{10}$ and $m_{11}$ are solutions of the linear system given by $\sum_{i = 0}^{11} m_i = q^{18}$ and $\sum_{i = 0}^{11} \theta_im_i = 0$, since the adjacency matrix of the Cayley graph has a null trace. 
    Moreover, if we denote by $\eta_q = (q+1)(q-1)(q^2+q+1)$, we have 
    \renewcommand{\arraystretch}{1.2}
    \[
        \begin{array}{rcl}
            m_1&=&\eta_q\left( 1 +q +q^{2}\right) ,\\
            m_2&=&\eta_q\left( -3q -q^{2} +3q^{4} +q^{5}\right),\\
            m_3&=&\eta_q \left(  q^{3} -q^{4} -q^{5} +q^{6}\right),\\
            m_4&=&\eta_q \left( - \frac{1 }{2}q^{3} -q^{4} -\frac{1 }{2}q^{5} +\frac{1 }{2}q^{6} +q^{7} +\frac{1 }{2}q^{8}\right),\\
            m_5&=&\eta_q \left( q +q^{2} +q^{3} -3q^{5} -3q^{6} -q^{7} +2q^{8} +2q^{9}\right),\\
            m_6&=&\eta_q \left(q^{6} -q^{8} -q^{9} +q^{11} \right),\\
            m_7&=&\eta_q \left( -\frac{3 }{2}q^{3} -2q^{4} +\frac{9 }{2}q^{5} +\frac{9 }{2}q^{6} -2q^{7} -\frac{9 }{2}q^{8} -3q^{9} +4q^{10}\right),\\
            m_8&=&\eta_q \left( -\frac{1 }{6}q^{7} +\frac{1 }{3}q^{9} +\frac{1 }{6}q^{10} -\frac{1 }{6}q^{11} -\frac{1 }{3}q^{12} +\frac{1 }{6}q^{14} \right),\\
            m_9&=&\eta_q \left( -q^{6} +2q^{8} +q^{9} -q^{10} -2q^{11} +q^{13}  \right).\\
        \end{array}
    \]
    \renewcommand{\arraystretch}{1}
    Thus we obtain $ \sum_{i = 0}^9 m_i = 1 + \eta_qP(q)$ with $P \in \Z[q]$ is defined by
    \[  
        P(q) = 
       \left.1 -q +q^{2} -q^{4} +q^{5} +3q^{6} -\frac{13 }{6}q^{7} -q^{8} -\frac{2 }{3}q^{9} +\frac{19 }{6}q^{10} -\frac{7 }{6}q^{11} -\frac{1 }{3}q^{12} +q^{13} +\frac{1 }{6}q^{14}  \right.,
    \]
    and one can observe that observe that 
    \[
        q^{18} - \sum_{i=0}^9 m_i = \frac{5}{6}q^4(q-1)^5(q+1)^2(q^2+q+1)\left(q^5 + \frac 35q^4 + \frac 75q^3 + \frac{12}{5}q^2 + \frac{12}{5}q + \frac{6}{5}\right).
    \]
    Additionally, we have 
    \renewcommand{\arraystretch}{1.2}
    \[
        \begin{array}{l} 
             \theta_0m_0 = \eta_q\!\left(1 + q + q^2\right),\\
             \theta_1m_1 = \eta_q\!\left(-1 -3q - 5q^2 -4q^3 + 4q^5 + 4q^6 + 2q^7 \right), \\
             \theta_2m_2 =\eta_q\!\left( 3q + 7q^2 +8q^3 - q^4 - 13q^5 - 13q^6 - 3q^7 + 6q^8 + 5 q^9 + q^{10} \right), \\
             \theta_3m_3 =\eta_q\!\left( -q^3 - q^4 + q^5 + 3q^6 + q^7 - 2q^8 - 2q^9 + q^{11} \right), \\
             \theta_4m_4 = \eta_q\!\left(\frac 12q^3 + 2q^4 + \frac 72q^5 + \frac 52q^6 - \frac 52q^7 - \frac{13}2q^8 - \frac 92q^9 + \frac 12q^{10} + 3q^{11} + \frac32q^{12} \right),\\
             \theta_5m_5 = \eta_q\!\left( -q - 3q^2\! - 5q^3 \!- 4q^4 \!+ 3q^5 \!+ 11q^6\! + 15q^7\! + 6q^8 \!- 10q^9 \!- 14q^{10} \!- 6q^{11}\! + 4q^{12}\! + 4q^{13}\right),\\
             \theta_6m_6 = \eta_q\!\left( -q^6 - 2q^7 - q^8 + 4q^9 + 5q^{10} - 4q^{12}-3q^{13} + q^{14} + q^{15} \right),\\
             \theta_7m_7 = \eta_q\!\left( \frac 32q^{3}\! + 5q^4 \!+ \frac 52 q^5\!- \frac{19}2 q^6 \!- \frac{35}2q^{7}\! - \frac 52q^8\!+ \frac{41}2 q^{9}\! + \frac{31}{2}q^{10}\! - 4q^{11}\! - \frac{25}{2}q^{12}\! - 3q^{13}\! + 4q^{14}\right),\\
             \theta_8m_8 = \eta_q\!\left(\frac 16 q^7 + \frac 13q^8 - \frac 76q^{10} - \frac 56q^{11} + q^{12} + \frac 43 q^{13} + \frac 16 q^{14} - q^{15} - \frac 13 q^{16} + \frac 13q^{17} \right),\\
             \theta_9m_9 = \eta_q\!\left( q^6 + 2q^7 - 6q^9 - 5q^{10} + 4q^{11} + 7q^{12} + 2q^{13} - 4 q^{14} - 2 q^{15} + q^{16}\right).\\ 
        \end{array}
    \]
    \renewcommand{\arraystretch}{1}
    Thus we obtain $\sum_{i = 0}^9 \theta_im_i = (q-1)(q+1)(q^2 + q +1)Q(q)$ where 
    \begin{align*}
        Q(q) &= q^4 + q^5 - 2q^6 - \frac{29}{6}q^7 + \frac 13 q^8 + 7q^9 + \frac{11}{6}q^{10} - \frac{17}{6}q^{11} - 3q^{12}  + \frac 43q^{13}\\
        & \ \ \ \ \ \ \ \ \ \ \  + \frac 76q^{14} - 2q^{15} + \frac{2}{3}q^{16} + \frac{1}{3}q^{17}\\
        &= \frac 13q^4(q-1)^4(q+1)(q^2+q+1)\left(q^6 + 4q^5 + 2q^4 + \frac{13}{2} q^3  + 12q^2 + 9q + 3\right).
    \end{align*}
    We can then note that 
    \[
        m_{10} = \frac 12 q^4(q-1)^5(q+1)^2(q^2+ q+1)^2(q^3 +2q +2)
        \text{ and } 
        m_{11} = \frac{1}{3}q^7(q-1)^6(q+1)^3(q^2+q+1)
    \]
    is the unique solution of the linear system of equations
    \begin{equation*}
        \left\{ 
        \begin{array}{l}
             m_{10}  +  m_{11} = q^{18} - \sum_{i = 0}^9 m_i\\
             \varepsilon_q m_{10}+ (q^3 + \varepsilon_q)m_{11} = \sum_{i = 0}^9 \theta_im_i.
        \end{array}
        \right.
    \end{equation*}
    or in other words, with $\omega_q = q^4(q-1)^5(q+1)^2(q^2+q+1)$, we have 
    \begin{equation}
        \label{eq:systemM10M11}
        \left\{ 
        \begin{array}{l}
             m_{10}  +  m_{11} = \frac{5}{6}\omega_q\left(q^5 + \frac 35q^4 + \frac 75q^3 + \frac{12}{5}q^2 + \frac{12}{5}q + \frac{6}{5}\right)\\
             \varepsilon_q m_{10}+ (q^3 + \varepsilon_q)m_{11} = \frac 13\omega_q(q^2+q+1)\left(q^6 + 4q^5 + 2q^4 + \frac{13}{2} q^3  + 12q^2 + 9q + 3\right).
        \end{array}
        \right.
    \end{equation}
    with $\varepsilon_q = 2q^2 + 2q + 1$. Indeed, since $\left[\begin{smallmatrix}
        1&1 \\ \varepsilon_q & q^3 + \varepsilon_q
    \end{smallmatrix} \right] \left[\begin{smallmatrix}
        q^3 + \varepsilon_q & -1 \\ -\varepsilon_q & 1
    \end{smallmatrix}  \right] = q^3I_2$, with $\nu_q =q^4(q-1)^5(q+1)^2(q^2 + q +1)$, we have
    \begin{align*}
        q^3m_{10} 
        &=\nu_q\left(\frac 56(q^3+2q^2+2q+1)\left(q^5 + \frac 35q^4 + \frac 75q^3 + \frac{12}{5}q^2 + \frac{12}{5}q + \frac{6}{5}\right)\right)\\
        & \ \ \ \ \ - \nu_q\left(\frac 13(q^2+q+1)\left(q^6 + 4q^5 + 2q^4 + \frac{13}{2} q^3  + 12q^2 + 9q + 3\right)\right)\\
        &= \frac 12 q^7(q-1)^5(q+1)^2(q^2 + q + 1)^2 (q^3 + 2 q + 2),
    \end{align*}
    as well as
    \begin{align*}
        q^3m_{11} 
        &= - \nu_q\left(\frac 56(2q^2+2q+1)\left(q^5 + \frac 35q^4 + \frac 75q^3 + \frac{12}{5}q^2 + \frac{12}{5}q + \frac{6}{5}\right)\right)\\
        & \ \ \ \ \ + \nu_q\left(\frac 13(q^2+q+1)\left(q^6 + 4q^5 + 2q^4 + \frac{13}{2} q^3  + 12q^2 + 9q + 3\right)\right)\\
        &= \frac 13 q^{10}(q-1)^6(q+1)^3(q^2 + q +1).
    \end{align*}
    Therefore, this proves that Table \ref{tab:spectrumFq233} is correct.
\end{proof}

\begin{Lemma}\label{lemma:orderingeigenvaluesq233}
    With the notations of Table \ref{tab:spectrumFq233}, for $q = 2$ we have
    \[
        \theta_0 > \theta_1 > \theta_2 > \theta_3 = \theta_4>\theta_5> \theta_6> \theta_7= \theta_8 > 0 > -1 > \theta_9> \theta_{10}> \theta_{11},
    \]
    and for $q\geq 3$, we have 
    \[
        \theta_0 > \theta_1 > \theta_2 > \theta_3> \theta_4> \theta_5> \theta_6> \theta_7> \theta_8> \theta_9>0>-1> \theta_{10}> \theta_{11}.
    \]
\end{Lemma}
\begin{proof}
    This is immediate with the result of Table \ref{tab:spectrumFq233}.
\end{proof}

\begin{table}[ht!]
    \centering
    \renewcommand{\arraystretch}{1.2}
    \begin{tabular}{|c|c|c|c|} \hline
         $i$&\textbf{Orbits}&\textbf{Eigenvalue}&\textbf{Multiplicity} $m_i$  \\\hline
          $\theta_0$ &
          $0$ &
          $q^6 + 2 q^5 + 2 q^4 - \varepsilon_q $ 
          & 1\\
          $\theta_1$ &
          $o_1$ & 
          $2q^5 + 2q^4 - \varepsilon_q$& 
          $(q-1)(q+1)(q^2+q+1)^2 $ \\
          $\theta_2$ &
          $o_2 , o_4 , o_4^T$ &
          $q^5+2q^4-\varepsilon_q$&
          $q(q-1)^2(q+1)(q^2+q+1)^2(q+3)$\\
          $\theta_3$ &
          $o_3$ & 
          $ q^5 + q^4 - \varepsilon_q $&
          $q^3(q-1)^3(q+1)^2(q^2+q+1)$\\
          $\theta_4$ &
          $o_5$ & 
          $ 3q^4 - \varepsilon_q$&
          $\frac 12 q^3(q-1)^2(q+1)^3(q^2+q+1)^2$ \\
          $\theta_5$ &
          $o_6, o_7 , o_7^T$&
          $ 2q^4 - \varepsilon_q $&
          $q(q-1)^3(q+1)^3(q^2+q+1)^2(2q^2+1)$\\
          $\theta_6$ &
          $o_8$ & 
          $q^4 + q^3 - \varepsilon_q $&
          $q^6(q-1)^3(q+1)^2(q^2+q+1)^2$\\
          $\theta_7$ &
          $o_9, o_{10}, o_{11} ,o_{11}^T, o_{12} $ &
          $ q^4 - \varepsilon_q$&
          $\frac{1}{2}q^3(q-1)^4(q+1)(q^2+q+1)^2(4q+3)(2q+1) $\\
          $\theta_8$ &
          $o_{14}$ &
          $ 2q^3 - \varepsilon_q $&
          $ \frac 16 q^7(q-1)^4(q+1)^3(q^2+q+1)^2 $\\
          $\theta_9$ &
          $o_{13}$ &
          $ q^3 - \varepsilon_q $&
          $q^6(q-1)^4(q+1)^3(q^2+q+1)^2$\\
          $\theta_{10}$ &
          $o_{15},o_{16}$ &
          $ -\varepsilon_q$&
          $\frac 12 q^4(q-1)^5(q+1)^2(q^2+ q+1)^2(q^3 +2q +2)$
          \\
          $\theta_{11}$ &
          $o_{17}$ &
          $ -q^3 -\varepsilon_q$&
          $\frac{1}{3}q^7(q-1)^6(q+1)^3(q^2+q+1)$ 
          \\\hline
    \end{tabular}\\
    \phantom{x}\hfill where $\varepsilon_q = 2q^2 + 2q +1$.\hfill \phantom{x}
    \renewcommand{\arraystretch}{1}
    \caption{Spectrum of $\Cay(\F_q^{2 \times 3 \times 3},\S_{2\times 3 \times 3})$.}
    \label{tab:spectrumFq233}
\end{table}

\begin{Rq}\label{rk:retroactivecomputationoforbitsize}
    We can recover the orbit sizes of $o_{15}$ and $o_{17}$ using the values of $m_{10}$, $m_{11}$ and the orbit size of $o_{16}$, which gives the orbit sizes in Table \ref{tab:SizeAndRankDistribSpaces}. Indeed, let 
    $\CCC_{15} = \left\langle \left[\begin{smallmatrix}
        1&u&0\\0&1&0
        \end{smallmatrix}\right],
        \left[\begin{smallmatrix}
        0&1&0\\v&0&0
        \end{smallmatrix}\right],
        \left[\begin{smallmatrix}
        0&0&1\\0&0&0
        \end{smallmatrix}\right]
        \right\rangle
    $,
    $\CCC_{16} = \left\langle   
        \left[\begin{smallmatrix}
        1&0&0\\0&1&0
        \end{smallmatrix}\right],
        \left[\begin{smallmatrix}
        0&1&0\\0&0&1
        \end{smallmatrix}\right],
        \left[\begin{smallmatrix}
        0&0&1\\0&0&0
        \end{smallmatrix}\right]
        \right\rangle
    $, and 
    $ \CCC_{17} = \left\langle   
        \left[\begin{smallmatrix}
        1&0&0\\0&1&0
        \end{smallmatrix}\right],
        \left[\begin{smallmatrix}
        0&1&0\\0&0&1
        \end{smallmatrix}\right],
        \left[\begin{smallmatrix}
        0&0&1\\\alpha&\beta&\gamma
        \end{smallmatrix}\right]
        \right\rangle
    $ be representatives of the orbits $o_{15}$, $o_{16}$, $o_{17}$, respectively. The spaces $\CCC_{15}$, $\CCC_{16}$, and $\CCC_{17}$ are all $3$-dimensional subspaces of $\F_q^{2 \times 3}$, $o_{15}$ and $o_{16}$ correspond to the eigenvalue $\theta_{10}$ and $o_{17}$ to $\theta_{11}$. Hence, by Lemma~\ref{lemma:orbitsofcodesvstensors} we have 
    \[
        m_{10} = |\GL_3(\F_q)|(|o_{15}| + |o_{16}|) \text{ and } m_{11} = |\GL_3(\F_q)|\cdot|o_{17}|.
    \]
    Therefore, we obtain $|o_{15}| = |\GL_3(\F_q)|^{-1}m_{10} - |o_{16}|$ and $|o_{17}| = |\GL_3(\F_q)|^{-1}m_{11}$ which gives the result in Table \ref{tab:SizeAndRankDistribSpaces}.
\end{Rq}

\begin{Rq}\label{rk:eigenvaluescayleyqis2}
    To avoid conflict of notations with the equal eigenvalues in Table \ref{tab:spectrumFq233} when $q=2$, the eigenvalues of the graph $\Cay(\F_2^{2 \times 3 \times 3},\S_{2\times 3 \times 3})$ are given in strictly decreasing order in Table \ref{tab:spectrumFq233q2}. 
\end{Rq}

\begin{table}[ht!]
    \centering
    \renewcommand{\arraystretch}{1.2}
    \setlength{\tabcolsep}{4pt}
    \begin{tabular}{|c|c|c|c|c|c|c|c|c|c|c|} \hline
        $i$ & $0$ & $1$ & $2$ &  $3$ & $4$ & $5$ & $6$ & $7$ & $8$ & $9$ \\\hline
        \textbf{Eigenvalue} $\theta_i'$ & $147$ & $83$ & $51$ & $35$ & $19$ & $11$ & $3$ & $-5$ & $-13$ & $-21$ \\\hline
        \textbf{Multiplicity} $m_i'$ & $1$ & $147$ & $1470$ & $5,796$ & $23,814$ & $28,224$ & $60,564$ & $84,672$ & $49,392$ & $8,064$\\\hline
    \end{tabular}
    \renewcommand{\arraystretch}{1}
    \caption{Spectrum of $\Cay(\F_2^{2 \times 3 \times 3},\S_{2\times 3 \times 3})$.}
    \label{tab:spectrumFq233q2}
\end{table}

\subsection{Application to upper bounds on independence numbers }
With the full spectrum computed above, we can now apply Theorem~\ref{thm:ratiobound23} on $\Cay(\F_q^{2 \times 3 \times 3},\S_{2 \times 3 \times 3})$ and deduce upper bounds on the cardinality of $(2\times 3 \times 3,M,d)_q$ tensor codes for $d = 3$ and $d = 4$.

\begin{Prop}\label{prop:alpha2for233}
    Let $\alpha_2$ be the $2$-independence number of the graph $\Cay(\F_q^{2\times 3 \times  3},\S_{2 \times  3 \times 3})$. 
    \begin{itemize}
        \item If $q =2$, then $\alpha_2 < 2^{10.63}$, 
        \item and if $q \geq 3$, then $\alpha_2 \leq q^{12}\left(1 - 4q^{-1} + 12q^{-2}\right) $.
    \end{itemize}
\end{Prop}
\begin{proof}
    By \cite[Corollary 3.3]{AbiadCoutinhoFiol2019kindependancenumber}, with $\theta_0$ the largest eigenvalue of the graph, with $\theta_i$ the largest eigenvalue such that $\theta_i \leq -1$, and with $\theta_{i-1}$ the smallest eigenvalue such that $\theta_{i-1} > \theta_i$, we have
    \[
        \alpha_2 \leq q^{18} \frac{\theta_0 + \theta_{i}\theta_{i-1}}{(\theta_0 - \theta_i)(\theta_0-\theta_{i-1})}.
    \]
    For $q = 2$, by Remark \ref{rk:eigenvaluescayleyqis2} we observe that $\theta_i = \theta_9 = -5$ and $\theta_{i-1} = \theta_8 = 3$, therefore $\alpha_2 \leq 2^{13}\times \frac{11}{57} < 2^{10.63}$. For $q = 3$, we have 
    \begin{align*}
        \alpha^2q^{-12} &\leq q^{8}\frac{q^6 + 2q^5 + 2 q^4 - 2q^2 - 2q -1 + (-2q^2 -2q -1)(q^3 - 2q^2 -2q -1)}{(q^6 + 2q^5 + 2q^4)(q^6 + 2q^5 + 2q^4 - q^3)}\\
        &= 1 -4q^{-1} + 12q^{-2} +  \frac{-16q^9 - 64q^{8} - 72q^7 - 32q^6 + 24q^5}{q^{12} + 4q^{11} + 8q^{10} + 7q^{9} +  2q^8 - 2q^7}.
    \end{align*}
    Checking the sign of the last fraction suffices to conclude.
\end{proof}

\begin{Corol}
    No $(2\times 3 \times 3,M,3)_q$ tensor code meets the Singleton-like bound. More precisely, $M < 2^{10.63}$ if $q = 2$ and $M < q^{12}$ if $q \geq 3$.
\end{Corol}
\begin{proof}
    The Singleton-like bound for $(2\times 3 \times 3,M,3)_{q}$ tensor codes states that $M \leq q^{3(6-3+1)} = q^{12}$, and since $M \leq \alpha_2$, Proposition~\ref{prop:alpha2for233} yields the wanted result.
\end{proof}

\begin{Lemma}\label{lemma:computationdelta3indepnb}
    For each $i\in \llbracket 0,11\rrbracket$, denote by $m_i$ the multiplicity of the eigenvalue $\theta_i$ of the Cayley graph $\Cay(\F_{q}^{2\times3 \times 3},\S_{2\times 3 \times 3})$ as specified in Table \ref{tab:spectrumFq233}. Denote by  $A$ the adjacency matrix of the graph and by $\Delta = \max_{u \in \F_q^{2 \times 3 \times 3}} (A^3)_{(u,u)}$. Then we have 
    \begin{align*}
        \Delta = 
        q^{-18}\sum_{i=0}^{11} m_i \theta_i^3 &=(q-1)(q+1)(q^2+q+1)^2(2q^3 + q^2 - 2q -2),
    \end{align*}
    and if $q \geq 3$, we have 
    \[
        \theta_9 < -\frac{\theta_0^2 + \theta_0\theta_{11} - \Delta}{\theta_0( \theta_{11}+1)} <\theta_8.
    \]
\end{Lemma}
\begin{proof}
    By Corollary~\ref{corol:polynomialexpressionofadjacencymatrixcayley}, we have $\Delta =  q^{-18}\sum_{i=0}^{11} m_i \theta_i^3$. Since $A$ is a matrix with integer coefficients, $\Delta$ is an integer. Consequently, $\sum_{i=0}^{18} m_i\theta_i^3$ is divisible by $q^{18}$ as a polynomial in $q$. Hence, to compute the sum, it suffices to compute the monomials of degree at least $18$ of the polynomials $m_i\theta_i^{3},i\in \llbracket0,11\rrbracket$ and sum them to obtain the value of $\delta$. Denote by $\eta_q = (q-1)(q+1)(q^2+q+1)$. Since the sum is also divisible by $\eta_q$, the sum is divisible by $\eta_qq^{18}$. We have 
    \[
        \begin{array}{r>{\displaystyle}l}
             \theta_0^3m_0 
             &= \eta_q(q-1)^2(q+1)^2(q^2+q+1)^5
             = \eta_q o(q^{18}),\\[0.1cm]
             \theta_1^3m_1 
             &= \eta_q(q^2+q+1)\left(2q^5 + 2q^4 - 2q^2 - 2q - 1\right)^3 
             = \eta_q o(q^{18}), \\[0.1cm]
             \theta_2^3m_2 
             &=   q(q-1)(q+1)^3(q^2+q+1)(q+3)(q^4 + q^3 - q^2 - q - 1)^3 \\
             &\qquad = \eta_q \left( 30q^{18} + 9q^{19} + q^{20}  + o(q^{18})  \right),\\[0.1cm]
             \theta_3^3m_3
             &=  q^3(q-1)^2(q+1)(q^2+q+1)^3(q^3 - q - 1)^3\\
             &\qquad = \eta_q \left( -10q^{18} -q^{19} + 2q^{20} + q^{21}  + o(q^{18}) \right), \\[0.1cm]
             \theta_4^3m_4 
             &=  \frac 12q^3(q-1)(q+1)^2(q^2+q+1)(3q^4 - 2q^2 - 2q - 1)^3  \\
             &\qquad = \eta_q \left( -\frac{27}{2}q^{18} + 27q^{19} -\frac{27}{2}q^{20}  + o(q^{18}) \right), \\[0.1cm]
             \theta_5^3m_5 
             &= q(q-1)^2(q+1)^2(q^2+q+1)(2q^2+1)(2q^4 - 2q^2 - 2q - 1)^3 \\
             &\qquad = \eta_q \left( -120q^{18} - 56q^{19} + 16q^{20} + 16q^{21}   + o(q^{18}) \right), \\[0.1cm]
             \theta_6^3m_6
             &=  q^6(q-1)^2(q+1)(q^2+q+1)(q^4 + q^3 - 2q^2 - 2q - 1)^3\\ 
             &\qquad = \eta_q \left( 45q^{18}  -9q^{19}  -21q^{20}  -4q^{21} +3q^{22} + q^{23}  + o(q^{18}) \right), \\[0.1cm]
             \theta_7^3m_7
             &= \frac 12 q^3(q-1)^3(q^2+q+1)(4q+3)(2q+1)(q^4 - 2q^2 - 2q - 1)^3\\   
             &\qquad = \eta_q \left( \frac{171}{2} q^{18} -8q^{19} -\frac{57}{2}q^{20}  -3q^{21} + 4q^{22} + o(q^{18})     \right),\\[0.1cm]
             \theta_8^3m_8
             &=  \frac 16 q^7(q-1)^3(q+1)^2(q^2+q+1)(2q^3 - 2q^2 - 2q - 1)^3 \\ 
             &\qquad = \eta_q \left( -\frac{38}3q^{18} + \frac{28}{3}q^{19} + \frac{34}3q^{20} -\frac{8}{3}q^{21} -4q^{22} +\frac43q^{23}     + o(q^{18})\right) \\[0.1cm]
             \theta_9^3m_9 
             &= q^6(q-1)^3(q+1)^2(q^2+q+1)(q^3 - 2q^2 - 2q - 1)^3   \\
             &\qquad = \eta_q \left(  -5q^{18} +  24 q^{19} +  4 q^{20}  -6q^{21} + q^{22}  + o(q^{18}) \right), \\[0.1cm]
             \theta_{10}^3m_2 
             &= \frac 12 q^4(q-1)^4(q+1)(q^2+q+1)(q^3+2q+2)(-2q^2 - 2q - 1)^3 \\ 
             &\qquad = \eta_q \left( -2q^{18} -4q^{19} - 4q^{20}  + o(q^{18}) \right), \\[0.1cm]
             \theta_{11}^3m_2 
             &=  \frac 13 q^7(q-1)^5(q+1)^2(-q^3 - 2q^2 - 2q - 1)^3  \\ 
             &\displaystyle \qquad = \eta_q \left( \frac{2}{3}q^{18} + \frac{14}{3}q^{19} +\frac{8}{3}q^{20} -\frac{1}{3}q^{21} -q^{22} -\frac{1}{3}q^{23} + o(q^{18})  \right). \\
        \end{array}
    \]
    Hence, we have $\sum_{i = 0}^{11} m_i\theta_i^3 = \eta_q(-2q^{18} -4q^{19} -3q^{20} +q^{21} + 3q^{22} + 2 q^{23} + o(q^{18}))$, and thus
    \[
        \Delta = \eta_q (-2 -4q -3q^2 +q^3 + 3q^4 + 2 q^5) = (q-1)(q+1)(q^2+q+1)^2(2q^3 + q^2 - 2q -2).
    \]
    Note that the value of $\Delta$ computed here is also valid for $q=2$ since $\sum_{i=0}^{11}m_i\theta_i^{3} = \sum_{i=0}^{8}m_i'(\theta_i')^3$ with the notations of Remark \ref{rk:eigenvaluescayleyqis2}. Assume now that $q \geq 3$. Since $\theta_0 = (q-1)(q+1)(q^2+q+1)^2$, $\Delta = \theta_0(2q^3 + q^2 - 2q-2)$, and $\theta_{11} = -(q+1)(q^2+q+1)$ we have 
    \begin{align*}
        \theta_0 + \theta_{11} - \frac{\Delta}{\theta_0}
        &= (q-1)(q+1)(q^2+q+1)^2 - (q+1)(q^2+q+1) - (2q^3 + q^2 - 2q-2)\\
        &= q(q^5+2q^4+2q^3-3q^2-5q-2)
    \end{align*}
    and since $\theta_{11}+1 = -q(q^2 +2q+2)$, we have 
    \begin{align*}
        \frac{\theta_0^2 + \theta_0\theta_{11} - \Delta}{-\theta_0( \theta_{11}+1)} 
        &= \frac{\theta_0 + \theta_{11} - \frac{\Delta}{\theta_0}}{-( \theta_{11}+1)}\\ 
        &= \frac{q^5+2q^4+2q^3-3q^2-5q-2}{q^2 +2q+2}\\
        &= (q^3 - 3) + \frac{q+4}{q^2+2q+2}.
    \end{align*}
 From this we immediately see that this value is between $\theta_9 = q^3 - 2q^2 - 2q -1$ and $\theta_8 = 2q^3 - 2q^2 - 2q -1$ which concludes the proof.
\end{proof}
\begin{Prop}\label{prop:alpha3for233}
    Let $\alpha_3$ be the $3$-independence number of $\Cay(\F_q^{2\times  3 \times  3},\S_{2\times  3 \times  3})$. The following holds. 
    \begin{itemize}
        \item If $q = 2$, then we have $\displaystyle \alpha_3 \leq  2^{8}\frac{3780}{3213}$.
        \item If $q \geq 3$, then we have $\displaystyle \alpha_3 \leq q^9 \left( 2 -\frac{7 q^5  - 14 q^4  - 4q^3 + 5q^2 + 12q -4}{q^6 + 4q^5 + 8q^4 + 5q^3 - 2q^2 - 6q + 2}\right)$.
    \end{itemize}
\end{Prop}
\begin{proof}
    By \cite[Theorem 11]{AbiadCoutinhoFiol2019kindependancenumber}, with $\theta_0$ the largest eigenvalue of the graph, with $\theta_{11}$ the smallest, with $\Delta = \max_{u \in \F_{q}^{2 \times 3 \times 3}} (A^3)_{(u,u)}$, with the largest eigenvalue such that $\theta_s \geq \Delta$, and $\theta_{s+1}$ the largest eigenvalue less than $\theta_s$, we have
    \[
        \alpha_3 \leq q^{18} \frac{\Delta - \theta_0(\theta_s + \theta_{s+1} + \theta_{11}) - \theta_{s}\theta_{s+1}\theta_{11}}{(\theta_0 - \theta_s)(\theta_0-\theta_{s+1})(\theta_0 - \theta_{11})}.
    \]
    For $q = 2$, observe with Lemma~\ref{lemma:computationdelta3indepnb} that $\Delta = 2058$. Then, by Remark \ref{rk:eigenvaluescayleyqis2}, we have 
    \[
        -\frac{\theta_0'^2 + \theta_0'\theta_8' - \Delta}{\theta_0(\theta_8+1)} = \frac{21609 - 3087 - 2058}{2940} = \frac{28}{5}.
    \]
    Therefore,  we have 
    \[
        \alpha_3 \leq 2^{18}\frac{\Delta - \theta_0'(\theta_4' + \theta_5' + \theta_8') - \theta_4'\theta_5'\theta_8'}{(\theta_0'-\theta_4')(\theta_0'-\theta_5')(\theta_0'-\theta_8')} = 2^{18}\frac{2058 - 147\times (-7) - (-693)}{136\times 144 \times 168} = 2^{8}\frac{3780}{3213},
    \]
    which proves the statement for $q = 2$. Now for $q = 3$, with Lemma~\ref{lemma:computationdelta3indepnb}, since $\theta_0 = (q-1)(q+1)(q^2+q+1)^2$ and $\Delta = \theta_0(2q^3 + q^2 - 2q-2)$, we have 
    \begin{align*}
        &(\theta_0 - \theta_8)(\theta_0 - \theta_9)(\theta_0 - \theta_{11}) \\
        &= (q^6 + 2q^5 + 2q^4 - 2q^3)(q^6 + 2q^5 + 2q^4 -q^3) (q^6 + 2q^5 + 2q^4 +q^3)\\
        &= q^9 (q^9 + 6 q^8 + 18 q^7 + 30 q^6 + 28 q^5 + 8 q^4 - 9 q^3 - 10 q^2 - 2 q + 2).
    \end{align*}
    Moreover, we have 
    \begin{align*}
        \Delta - \theta_0(\theta_8 + \theta_9 + \theta_{11})
        &= \theta_0\left(2q^3 + q^2 -2q -2 - \left(2q^3 - 6q^2 - 6q - 3\right) \right)\\
        &= 7 q^8 + 18 q^7 + 23 q^6 + 10 q^5 - 12 q^4 - 22 q^3 - 17 q^2 - 6 q - 1,
    \end{align*}
    and 
    \begin{align*}
        -\theta_8\theta_9\theta_{11}
        &= 2 q^9 - 2 q^8 - 10 q^7 -9 q^6 + 8 q^5 + 28 q^4 +30 q^3 + 18 q^2 + 6 q + 1.
    \end{align*}
    Therefore, we obtain the bound
    \begin{align*}
        \alpha_3 &\leq q^9\frac{2q^9 +5q^8 + 8q^7+ 14q^6 + 18q^5 +16q^4 + 8q^3 + q^2}{q^9 + 6 q^8 + 18 q^7 + 30 q^6 + 28 q^5 + 8 q^4 - 9 q^3 - 10 q^2 - 2 q + 2}\\
        &= q^9 \left( 2 -\frac{7 q^5  - 14 q^4  - 4q^3 + 5q^2 + 12q -4}{q^6 + 4q^5 + 8q^4 + 5q^3 - 2q^2 - 6q + 2}\right).
    \end{align*}
\end{proof}

\begin{Corol}
    Let $q \in \{2,3,4\}$. There is no $(2\times 3 \times 3,M,4)_q$ tensor code that meets the Singleton-like bound. More precisely, we have $M < q^9$.
\end{Corol}
\begin{proof}
    The Singleton-like bound for $(2\times 3 \times 3,M,3)_{q}$ codes states that $M \leq 2^{3(6-4+1)} = 2^{9}$. By Proposition~\ref{prop:alpha3for233}, for $q = 2$, since $\frac{3780}{3213}< 2$, we have $M\leq \alpha_3 < 2^9$. Moreover, for $q =3$ and $q=4$, the value of $\displaystyle \left( 2 -\frac{7 q^5  - 14 q^4  - 4q^3 + 5q^2 + 12q -4}{q^6 + 4q^5 + 8q^4 + 5q^3 - 2q^2 - 6q + 2}\right)$ is respectively $\frac{1017}{1225}$ and $\frac{1688}{1751}$, which concludes the proof.
\end{proof}

The computations of the proof above show that the Singleton-like bound is sharper than the ratio-type bound on $(2\times 3\times 3,M,4)_q$ codes for each $q \geq 5$. The same result can be observed numerically; see Table \ref{tab:RatioType2345}.

\section{Final remarks}
\label{sec:bounds}

In the previous section, the particular cases of the ratio-type bounds given in Theorem~\ref{thm:ratiobound23} allowed us to deduce bounds on $k$-independence number of the graph $\Cay(\F_q^{2\times 3 \times 3},\S_{2\times 3 \times 3})$ for $k = 2$ and $k=3$, that is to say on tensor codes in $\F_q^{2 \times 3 \times 3}$ of minimum distance $3$ and $4$; see Theorem~\ref{thm:sumup233}. The upper bounds presented are sharper than the Singleton-like bound (\ref{eq:singletonlike}), as well as its improved version (\ref{eq:improvedsingletonlike}) since the codes considered have minimum distance strictly less than $6$. 

In \cite{AbiadCoutinhoFiol2019kindependancenumber}, the authors suggested linear programming (LP) and mixed integer linear programming (MILP) algorithms to compute the optimal bounds for the inertia-type bounds and the ratio-type bounds given in Theorem~\ref{thm:inertiaratio}. In particular, Algorithms \cite[(17)]{AbiadCoutinhoFiol2019kindependancenumber} and \cite[(18)]{AbiadCoutinhoFiol2019kindependancenumber} are,  respectively, LP and MILP algorithms that compute the optimal solution for the bounds given in (\ref{eq:ratiotype}) and (\ref{eq:intertiatype}) for $k$-partially walk-regular graphs. The Sage code of these algorithms can be found in \cite{Peters2024BoundingCardinality}. In this section, we will present the numerical computation of those bounds for low parameters.

Let us start by detailing how to compute the spectrum of a multilinear forms graph with a given set of parameters. Corollary~\ref{corol:eigenvaluesmultilinearformscountingspaces} shows that the spectrum of $\Cay(\F_q^{\mathbf{n}},\S_{\mathbf n})$ is determined by the numbers
\[
    N_{a,k} = \left|\{ \CCC \leq \F_q^{{\mathbf n}(j)} \ : \ \dim \CCC = k, A_1(\CCC) = a  \}\right|,
\]
for a choice of $j \in \llbracket 1,m\rrbracket$, and for $k \geq n_1\cdots n_{j-1}n_{j+1} \cdots n_m - n_j$. These numbers can then be computed by fixing a choice of $j$ and listing all subspaces of $\F_q^{{\mathbf n}(j)}$. The computation of the spectrum of the graph $\Cay(\F_q^{\mathbf{n}},\S_{\mathbf n})$ using the computation of the numbers $N_{a,k}$ will require enumerating 
\[
    \sum_{k = n_1\cdots n_{j-1}n_{j+1} \cdots n_m - n_j}^{n_1\cdots n_{j-1}n_{j+1} \cdots n_m } \gbc{k}{n_1\cdots n_{j-1}n_{j+1} \cdots n_m }{q} = \sum_{c = 0}^{n_j}  \gbc{c}{n_1\cdots n_{j-1}n_{j+1} \cdots n_m }{q}
\]
subspaces of $\F_{q}^{\mathbf n(j)}$ and computing the cardinality of their intersection with $\S_{\mathbf n(j)}$. 

\begin{Rq}\label{rem:rk1}
    It has been established in \cite{HASTAD1990644} that the computation of the rank of a $3$-tensor, and by extension of an $m$-order tensor for $m\geq 3$, is NP-complete. However, Proposition~\ref{prop:slicespdimlowerboundstrank} shows that a tensor has rank one if and only if all of its slice spaces are one-dimensional. Therefore, checking if a tensor has rank one or not has complexity 
    polynomial in $m$ and in the entries of $\mathbf n$. Therefore, the computation of the intersection of $\S_{\mathbf n(j)}$ and a given subspace $V$ has complexity $O(|V|)$ in terms of operations in $\F_q$.
\end{Rq}

For small parameters, we compute the numbers $N_{a,k}$ with the enumeration mentioned above, and deduce the spectrum of the corresponding multilinear forms graphs. We give in Table \ref{tab:RatioType2345} the ratio-type bounds. We then compare it to the bounds mentioned in Section \ref{sec:rankmetriccodesandtensorcodes}, that is the Singleton-like bound (Theorem~\ref{thm:singletonlike}), its improved version (Theorem~\ref{thm:improvedsingletonlike}), and the sphere-packing bounds for codes of minimum distance $3$ and $5$ (Corollaries \ref{corol:spherepackingboundwithd3} and \ref{corol:spherepackingboundwithd5m3}). We use the colour \legendbox{lightgray} to indicate that the ratio-type bound in the Table \ref{tab:RatioType2345} is sharper than the corresponding one in Table \ref{tab:bestbetween}.

\begin{table}[ht!]   \centering
    \begin{tabular}{|c|c||c|c|c|c|c|c|c|}\hline
        $q$&$\mathbf{n}$&$d=3$&$d=4$&$d=5$&$d=6$&$d=7$&$d=8$&$d=9$\\\hline\hline
2 & (2,2,2) & 3.17 & {\footnotesize n/a} & {\footnotesize n/a} & {\footnotesize n/a} & {\footnotesize n/a} & {\footnotesize n/a} & {\footnotesize n/a}\\\hline
2 & (3,2,2) & 6.00 & 3.59 & {\footnotesize n/a} & {\footnotesize n/a} & {\footnotesize n/a} & {\footnotesize n/a} & {\footnotesize n/a}\\\hline
2 & (4,2,2) & 8.00 & 5.05 & {\footnotesize n/a} & {\footnotesize n/a} & {\footnotesize n/a} & {\footnotesize n/a} & {\footnotesize n/a}\\\hline
2 & (3,3,2) &  \cellcolor{lightgray}{10.63} &  \cellcolor{lightgray}{8.24} & 5.43 & {\footnotesize n/a} & {\footnotesize n/a} & {\footnotesize n/a} & {\footnotesize n/a}\\\hline
2 & (4,3,2) &  \cellcolor{lightgray}{15.46} & 12.13 & 8.87 & 6.25 & {\footnotesize n/a} & {\footnotesize n/a} & {\footnotesize n/a}\\\hline
2 & (5,3,2) & 20.00 & 15.73 & 12.02 & 8.72 & {\footnotesize n/a} & {\footnotesize n/a} & {\footnotesize n/a}\\\hline
2 & (6,3,2) & 24.00 & 19.22 & 14.85 & 11.45 & {\footnotesize n/a} & {\footnotesize n/a} & {\footnotesize n/a}\\\hline
2 & (4,4,2) &  \cellcolor{lightgray}{22.50} &  \cellcolor{lightgray}{18.99} & 14.95 & 12.04 & 8.97 & {\footnotesize n/a} & {\footnotesize n/a}\\\hline
2 & (5,4,2) &  \cellcolor{lightgray}{29.28} &  \cellcolor{lightgray}{24.85} & 20.55 & 16.75 & 13.40 & 10.24 & {\footnotesize n/a}\\\hline
2 & (6,4,2) & 36.01 & 30.62 & 25.77 & 21.18 & 17.34 & 13.93 & {\footnotesize n/a}\\\hline
2 & (7,4,2) & 42.01 & 36.06 & 30.55 & 25.97 & 21.50 & 17.62 & {\footnotesize n/a}\\\hline
2 & (3,3,3) & 18.58 &  \cellcolor{lightgray}{15.90} & 11.87 & 9.70 & 6.71 & {\footnotesize n/a} & {\footnotesize n/a}\\\hline
2 & (4,3,3) & 26.48 &  \cellcolor{lightgray}{22.80} & 18.57 &  \cellcolor{lightgray}{15.39} & 12.04 & 9.41 & 5.21\\\hline
2 & (5,3,3) &  \cellcolor{lightgray}{34.17} &  \cellcolor{lightgray}{29.57} & 25.16 & 21.11 & 17.33 & 13.87 & 8.49\\\hline
2 & (6,3,3) & 42.00 & 36.42 & 31.33 & 26.37 & 22.11 & 18.20 & 11.85\\\hline
2 & (7,3,3) & 49.00 & 42.75 & 37.01 & 31.98 & 27.08 & 22.95 & 16.89\\\hline
2 & (2,2,2,2) &  \cellcolor{lightgray}{9.57} &  \cellcolor{lightgray}{7.95} &  {5.09} &  \cellcolor{lightgray}{3.81} &  \cellcolor{lightgray}{2.00} &  \cellcolor{lightgray}{1.00} & {\footnotesize n/a}\\\hline
2 & (3,2,2,2) &  \cellcolor{lightgray}{16.34} &  \cellcolor{lightgray}{13.75} &  {10.49} &  \cellcolor{lightgray}{8.33} &  \cellcolor{lightgray}{5.89} & 4.17 & {\footnotesize n/a}\\\hline
2 & (4,2,2,2) &  \cellcolor{lightgray}{23.11} &  \cellcolor{lightgray}{19.49} & 16.06 & 12.97 & 10.01 & 7.60 & {\footnotesize n/a}\\\hline
2 & (5,2,2,2) & 30.01 & 25.37 & 21.20 & 17.13 & 13.63 & 10.77 & {\footnotesize n/a}\\\hline
3 & (2,2,2) &  \cellcolor{lightgray}{3.43} & {\footnotesize n/a} & {\footnotesize n/a} & {\footnotesize n/a} & {\footnotesize n/a} & {\footnotesize n/a} & {\footnotesize n/a}\\\hline
3 & (3,2,2) & 6.00 & 3.83 & {\footnotesize n/a} & {\footnotesize n/a} & {\footnotesize n/a} & {\footnotesize n/a} & {\footnotesize n/a}\\\hline
3 & (4,2,2) & 8.00 & 5.57 & {\footnotesize n/a} & {\footnotesize n/a} & {\footnotesize n/a} & {\footnotesize n/a} & {\footnotesize n/a}\\\hline
3 & (3,3,2) &  \cellcolor{lightgray}{11.39} &  \cellcolor{lightgray}{8.84} & 6.34 & {\footnotesize n/a} & {\footnotesize n/a} & {\footnotesize n/a} & {\footnotesize n/a}\\\hline
3 & (4,3,2) & 16.00 & 12.61 & 9.69 & 7.05 & {\footnotesize n/a} & {\footnotesize n/a} & {\footnotesize n/a}\\\hline
3 & (5,3,2) & 20.00 & 16.23 & 13.07 & 9.82 & {\footnotesize n/a} & {\footnotesize n/a} & {\footnotesize n/a}\\\hline
3 & (6,3,2) & 24.01 & 20.06 & 16.68 & 12.70 & {\footnotesize n/a} & {\footnotesize n/a} & {\footnotesize n/a}\\\hline
4 & (2,2,2) &  \cellcolor{lightgray}{3.59} & {\footnotesize n/a} & {\footnotesize n/a} & {\footnotesize n/a} & {\footnotesize n/a} & {\footnotesize n/a} & {\footnotesize n/a}\\\hline
4 & (3,2,2) & 6.00 & 3.96 & {\footnotesize n/a} & {\footnotesize n/a} & {\footnotesize n/a} & {\footnotesize n/a} & {\footnotesize n/a}\\\hline
4 & (4,2,2) & 8.00 & 5.80 & {\footnotesize n/a} & {\footnotesize n/a} & {\footnotesize n/a} & {\footnotesize n/a} & {\footnotesize n/a}\\\hline
4 & (3,3,2) &  \cellcolor{lightgray}{11.55} &  \cellcolor{lightgray}{8.98} & 6.59 & {\footnotesize n/a} & {\footnotesize n/a} & {\footnotesize n/a} & {\footnotesize n/a}\\\hline
4 & (4,3,2) & 16.00 & 12.73 & 10.06 & 7.40 & {\footnotesize n/a} & {\footnotesize n/a} & {\footnotesize n/a}\\\hline
4 & (5,3,2) & 20.01 & 16.49 & 13.73 & 10.29 & {\footnotesize n/a} & {\footnotesize n/a} & {\footnotesize n/a}\\\hline
4 & (6,3,2) & 24.01 & 20.42 & 17.60 & 13.26 & {\footnotesize n/a} & {\footnotesize n/a} & {\footnotesize n/a}\\\hline
5 & (2,2,2) &  \cellcolor{lightgray}{3.68} & {\footnotesize n/a} & {\footnotesize n/a} & {\footnotesize n/a} & {\footnotesize n/a} & {\footnotesize n/a} & {\footnotesize n/a}\\\hline
5 & (3,2,2) & 6.00 & 4.03 & {\footnotesize n/a} & {\footnotesize n/a} & {\footnotesize n/a} & {\footnotesize n/a} & {\footnotesize n/a}\\\hline
5 & (4,2,2) & 8.00 & 5.92 & {\footnotesize n/a} & {\footnotesize n/a} & {\footnotesize n/a} & {\footnotesize n/a} & {\footnotesize n/a}\\\hline
5 & (3,3,2) & \cellcolor{lightgray}{11.65} & 9.05 & 6.58 & {\footnotesize n/a} & {\footnotesize n/a} & {\footnotesize n/a} & {\footnotesize n/a}\\\hline
5 & (4,3,2) & 16.00 & 12.80 & 10.29 & 7.60 & {\footnotesize n/a} & {\footnotesize n/a} & {\footnotesize n/a}\\\hline
5 & (5,3,2) & 20.01 & 16.64 & 14.10 & 10.55 & {\footnotesize n/a} & {\footnotesize n/a} & {\footnotesize n/a}\\\hline
\end{tabular}
\caption{Ratio-type upper bounds on $\log_q(M)$ for $(\mathbf{n},M,d)_q$ codes for $q \in \{2,3,4,5\}$, rounded up to the next 0.01, and coloured in \legendbox{lightgray} if it is a strict improvement of the value in Table \ref{tab:bestbetween}.}\label{tab:RatioType2345}
\end{table}

A similar computation using the inertia-type bound for this list of examples shows that it is systematically worse than the Singleton-like bound. The code for these computations, based on the code in \cite{Peters2024BoundingCardinality}, may be found in the {Github repository} \url{https://github.com/lucienfrancois/EigenvaluesMultilinearFormsGraph}.

The following empirical observation can be stated as follows: for any set of parameters $q$ and $\mathbf n$, the ratio-type bound seems to be a strict improvement of the Singleton-bound for codes of any minimum distance upper-bounded by a threshold function of $q$ and $\mathbf n$. Therefore, while the improvement on the Singleton-like bound (Theorem~\ref{thm:improvedsingletonlike}) is sharper than the Singleton-like bound for large minimum distance codes, the ratio-type bound is for small minimum distance codes. The strict improvements on the other bounds also seem to be more often for shapes of tensors closer to cubes and hypercubes.

To conclude these remarks, let us discuss about the actual largest tensor codes of given minimum distance and small parameters. The ratio-type bound and the sphere-packing bound for $(2\times 2 \times 2,M,3)_2$ codes, which are both strict improvements of the Singleton-like bound ($M \leq 16$), give $M \leq 9$. However, one can computationally check that the largest $(2\times 2 \times 2,M,3)_2$ code has cardinality $M = 4$. Recall that, from Howell's result (Proposition \ref{prop:howellbound}) that the maximum rank of a tensor in $\F_q^{2\times 2 \times 2}$ is $3$, hence a code in $\F_q^{2\times 2 \times 2}$ of minimum distance $3$ is a code of constant distance $3$. Therefore, such a code corresponds to a clique in the graph $\Cay(\F_2^{2\times 2 \times 2},S)$, where $S$ is the set of rank $3$ tensors. Proceeding this way, one can check that there are exactly $10,624$ such codes of cardinality $4$, and all such codes are affine. For instance, the code $\CCC = \Span(x_1,x_2) \subseteq \F_2^{2 \times 2 \times 2}$ with  
\[
    \begin{array}{clll}
        &x_1 := e_1 \otimes e_1 \otimes e_1 &+\ e_1 \otimes e_2 \otimes e_2 &+\  (e_1 + e_2) \otimes e_1 \otimes e_2,\\[0.15cm]
        \text{and}& x_2 := e_1 \otimes e_1 \otimes e_2 &+\ e_2 \otimes e_1 \otimes e_1  &+\ e_2 \otimes e_2 \otimes e_2
    \end{array}
\]
is a $[2\times 2 \times 2,2,3]_2$ linear tensor-code, i.e. a $(2\times 2 \times 2,4,3)_2$ code, where $e_i = e_i^{(2)}$ for $i\in \{1,2\}$. Regarding the $(3\times 2 \times 2,M,3)_2$ codes, since the actual maximum rank in the space $\F_2^{3\times 2 \times 2}$ is $3$, one can proceed similarly to find the maximum cardinality of such a code. However, the same method proves to be computationally too expensive if done extensively. A partial search through cliques of the graph $\Cay(\F_2^{3\times 2 \times 2},S)$, where $S$ is the set of rank $3$ tensors, exhibits a large list of $[3\times 2 \times 2,5,3]_2$ linear tensor-codes, i.e. $(3\times 2 \times 2,32,3)_2$ tensor-codes. For instance, the code $\CCC = \Span(x_1,x_2,x_3,x_4,x_5) \subseteq \F_2^{3 \times 2 \times 2}$ with
\[
    \begin{array}{clll}
        &x_1 := (e_1^{(3)} + e_3^{(3)}) \otimes e_1 \otimes e_1 
        &+\ e_2^{(3)} \otimes e_1 \otimes e_2 
        &+\ e_3^{(3)} \otimes e_2 \otimes e_2,\\[0.15cm]
        &x_2 := (e_1^{(3)} + e_3^{(3)}) \otimes e_1 \otimes e_2 
        &+\ (e_1^{(3)} + e_3^{(3)}) \otimes e_2 \otimes e_1 
        &+\ e_3^{(3)} \otimes e_1 \otimes e_2,\\[0.15cm]
        &x_3 := (e_1^{(3)} + e_2^{(3)} + e_3^{(3)}) \otimes e_2 \otimes e_1 
        &+\ e_2^{(3)} \otimes e_1 \otimes e_2 
        &+\ e_3^{(3)} \otimes (e_1 + e_2) \otimes e_2,\\[0.15cm]
        &x_4 := (e_1^{(3)} + e_2^{(3)}) \otimes e_2 \otimes e_2 
        &+\ (e_2^{(3)} + e_3^{(3)}) \otimes e_2 \otimes e_1 
        &+\ e_3^{(3)} \otimes e_1 \otimes e_2,\\[0.15cm] 
        \text{and}& x_5 := (e_2^{(3)} + e_3^{(3)}) \otimes e_1 \otimes e_1 
        &+\ e_2^{(3)} \otimes e_2 \otimes (e_1 + e_2) 
        &+\ e_3^{(3)} \otimes e_1 \otimes e_2
    \end{array}
\]
is such a code. This approches, yet does not reach, the bound $\dim(\CCC)\leq 6$ of Table \ref{thm:ratiobound23}.

\section*{Acknowledgements}
{This work has emanated from research conducted with the financial support of the European Union MSCA Doctoral Networks, HORIZON-MSCA-2021-DN-01, Project 101072316.} We would like to thank Prof. John Sheekey for his comments and help on the computation of tensor codes of maximum size.

\newpage
\bibliographystyle{siam}
\bibliography{BiblioCodes_arXiv}

\appendix

\section{Stabiliser sizes of subspaces of 2x3 matrices}
\label{appendix:stabilisersizes}
In this section, we provide the details of the stabiliser sizes computations presented in Lemma~\ref{lemma:sizeofstabilisers}. Each computation uses elementary linear algebra techniques to describe the stabiliser and count their cardinality. For each matrix $M \in \F_q^{m \times n}$, let us denote by $\row_i(M) \in \F_q^{n}$ its $i$-th row and $\col_j(M) \in \F_q^{m}$ its $j$-th column. 

\begin{proof}[Proof of Lemma~\ref{lemma:sizeofstabilisers}] Let $P \in \GL_2(\F_q)$ and $Q \in \GL_3(\F_q)$. Denote respectively by $\{e_1^{(2)},e_2^{(2)}\}$ and by $\{e_1^{(3)},e_2^{(3)},e_3^{(3)}\}$ the canonical bases of $\F_q^2$ and $\F_q^3$. 
\begin{itemize}
    \item \underline{$o_2$:} Denote by $A = \left[\begin{smallmatrix}
        1&0&0\\0&0&0
        \end{smallmatrix}\right]$, by $B = 
        \left[\begin{smallmatrix}
        0&1&0\\0&0&0
        \end{smallmatrix}\right]$, and by $\CCC = \langle A,B\rangle$. We have 
        \[
            P\CCC Q = \CCC \iff \col_1(P) \in \langle e_1^{(2)}\rangle \text{ and } \langle \row_1(Q),\row_2(Q)\rangle =\langle e_1^{(3)},e_2^{(3)}\rangle.
         \]
        Indeed, assume that $P\CCC Q = \CCC$. Then we have $PAQ = \col_1(P)\otimes \row_1(Q)$ and $PBQ = \col_1(P) \otimes \row_2(Q)$, thus $\col_1(P) \in \langle e_1^{(2)}\rangle$ and $\{\row_1(Q),\row_2(Q)\} \in \langle e_1^{(3)},e_2^{(3)}\rangle$. The converse is immediate. Hence there are $(q-1)(q^2-q)$ choices for $P$ and $(q^2-1)(q^2-q)(q^3-q)$ choices for $Q^\top$, therefore
        \[
            |\Stab(\CCC)| 
            = (q-1)(q^2-1)(q^2-q)^2(q^3-q^2) 
            = q^4(q-1)^5(q+1).
        \]

        \item \underline{$o_4^T$:} Denote by $A = \left[\begin{smallmatrix}
        1&0&0\\0&0&0
        \end{smallmatrix}\right]$, by $B = 
        \left[\begin{smallmatrix}
        0&0&0\\1&0&0
        \end{smallmatrix}\right]$, and by $\CCC = \langle A,B\rangle$. A similar argument shows that we have 
        \[
            P\CCC Q = \CCC \iff 
            \row_1(Q) \in \langle e_1^{(3)} \rangle.
        \]
        Hence, there are $(q^2-1)(q^2-q)$ choices for $P$ and $(q-1)(q^3-q)(q^3-q^2)$ choices for $Q$, therefore
        \[
            |\Stab(\CCC)| 
            = (q-1)(q^2-1)(q^2-q)(q^3-q)(q^3-q^2)
            = q^4(q-1)^5(q+1)^2.
        \]    

        \item \underline{$o_5$:} Denote by $A = \left[
        \begin{smallmatrix}
            1&0&0\\0&0&0
        \end{smallmatrix}\right]$, by $B = \left[
        \begin{smallmatrix}
            0&0&0\\0&1&0
        \end{smallmatrix}\right]$, and by $\CCC = \langle A,B\rangle$. Note that the only bases of $\CCC$ composed of two rank one matrices contain one collinear to $A$ and one collinear to $B$, and that all rank one matrices of $\CCC$ are collinear to $A$ or $B$. Therefore, we have $P\CCC Q = \CCC$ if and only if either $(PAQ \in \langle A\rangle \text{ and } PBQ \in \langle B\rangle)$ or $(PAQ \in \langle B\rangle \text{ and } PBQ \in \langle A\rangle)$, and thus we have $ P\CCC Q = \CCC$ if and only if 
        \[
        \begin{array}{l}
            \left(\col_1(P) \in \langle e_1^{(2)} \rangle \text{, } \col_2(P) \in \langle e_2^{(2)}\rangle \text{, } \row_1(Q) \in \langle e_1^{(3)}\rangle \text{, and} \row_2(Q) \in \langle e_2^{(3)} \rangle \right)\\
            \text{or } \left(\col_1(P) \in \langle e_2^{(2)} \rangle \text{, } \col_2(P) \in \langle e_1^{(2)} \rangle \text{, } \row_1(Q) \in \langle e_2^{(3)} \rangle \text{, and } \row_2(Q) \in \langle e_1^{(3)}\rangle \right).
        \end{array}
        \]
        Hence, in each configuration there are $(q-1)^2$ choices for $P$, $(q-1)^2(q^3-q^2)$ for $Q$ and both configuration cannot happen at the same time, so 
        \[
            |\Stab(\CCC)| 
            = 2(q-1)^4(q^3-q^2)
            = 2q^2(q-1)^5.
        \]

        \item \underline{$o_6$:} Denote by $A = \left[
        \begin{smallmatrix}
            1&0&0\\0&1&0
        \end{smallmatrix}\right]$, by $B = \left[
        \begin{smallmatrix}
            0&0&0\\1&0&0
        \end{smallmatrix}\right]$, and by $\CCC = \langle A,B\rangle$. Then we have 
        \[
            P\CCC Q = \CCC \iff P_{1,2} = Q_{1,2} = Q_{1,3} = Q_{2,3} = 0 \text{ and } Q_{2,2} = Q_{1,1}P_{1,1}P_{2,2}^{-1}.
        \]
        Indeed, assume that $P\CCC Q = \CCC$. Since $PBQ = \col_2(P) \otimes \row_1(Q)$ and since the only matrices of rank one in $\CCC$ are collinear with $B$, then $\col_2(P) \in \langle e_2^{(2)} \rangle $ and $\row_1(Q) \in \langle e_1^{(3)}\rangle$, i.e. $P_{1,2} = Q_{1,2} = Q_{1,3} = 0$. Moreover, since the last column on any matrix of $\CCC$ is zero, we have $Q_{2,3} = 0$. Finally, since $PAQ  \in \CCC$ we have $P_{1,1}Q_{1,1} = P_{2,2}Q_{2,2}$. The converse is immediate. There are then $(q-1)(q^2-q)$ choices for $P$ and for each choice of $P$ there are $(q-1)$ choices for the last column of $P$, $(q-1)q^2$ choices for the first, and $q$ choices for the second one. Therefore,
        \[
            |\Stab(\CCC)| 
            = (q-1)(q^2-q)q^3(q-1)^2
            = q^4(q-1)^4.
        \]

        \item \underline{$o_7$:}  Denote by $A = \left[
        \begin{smallmatrix}
            1&0&0\\0&0&1
        \end{smallmatrix}\right]$, by $B = \left[
        \begin{smallmatrix}
            0&1&0\\0&0&0
        \end{smallmatrix}\right]$, and by $\CCC = \langle A,B\rangle$. A similar argument shows that we have
        \[
            P\CCC Q = \CCC \iff P_{2,1}, = Q_{2,1} = Q_{2,3} = Q_{3,1} = Q_{3,2} = 0 \text{ and } (P_{1,2},P_{2,2}) = P_{1,1}Q_{3,3}^{-1}(-Q_{1,3},Q_{1,1}).
        \]
        Hence, there are $q^2(q-1)^3$ choices for $Q$ and for each choice of $Q$, there are $(q-1)$ choices of $P$, therefore
        \[
            |\Stab(\CCC)| 
            = q^2(q-1)^4.
        \]

        \item \underline{$o_{11}$:}  Denote by $A = \left[
        \begin{smallmatrix}
            1&0&0\\0&1&0
        \end{smallmatrix}\right]$, by $B = \left[
        \begin{smallmatrix}
            0&1&0\\0&0&1
        \end{smallmatrix}\right]$, and by $\CCC = \langle A,B\rangle$. Then for each $P \in \GL_2(\F_q)$, there exists a unique $Q \in \GL_3(\F_q)$ up to scalar multiplication such that $P\CCC Q = \CCC$. Indeed, note that the right stabiliser of $\CCC$ is trivial, i.e. $\left\{ Q \in \GL_3(\F_q) \ : \ \CCC Q = \CCC \right\} = \F_q^\times I_3$. Therefore, for each $P \in \GL_2(\F_q)$ and $Q,Q' \in \GL_3(\F_q)$, if $P\CCC Q = P\CCC Q' = \CCC$, then $\CCC Q^{-1}Q' = \CCC$ and hence $Q' \in \F_q^\times Q$, which proves the wanted property. We will now show that for each $P \in \GL_2(\F_q)$, there exists a $Q \in \GL_3(\F_q)$ with $P\CCC Q = \CCC$, proving in turn that 
        \[
            |\Stab(\CCC)| 
            = (q-1)(q^2-1)(q^2-q)
            = q(q-1)^3(q+1).
        \] 
        Let $P \in \GL_2(\F_q)$. We will proceed as follows to construct the appropriate matrices $Q$ using the linear dependencies between the vectors $(P_{1,1},P_{1,2},0)$, $(P_{2,1},P_{2,2},0)$ , $(0,P_{1,1},P_{1,2})$, and $(0,P_{2,1},P_{2,2})$.
        \begin{itemize}
            \item If $P_{1,2} = 0$, then the invertible matrix $Q \in \GL_3(\F_q)$ such that
            \[
                (P_{1,1},0,0)Q = e_1^{(3)}, \  (P_{2,1},P_{2,2},0)Q = e_2^{(3)}, \text{ and }(0,P_{2,1},P_{2,2})Q = P_{2,2}^{-1}(0,-P_{2,1},P_{1,1}),
            \]
            satisfies $P\CCC Q = \CCC$.
            \item If $P_{1,1} = 0$, then the invertible matrix $Q \in \GL_3(\F_q)$ such that
            \[
                (P_{2,1},P_{2,2},0)Q = -P_{2,1}^{-1}(0,-P_{2,2},P_{1,2}), \  (0,0,P_{1,2})Q =  e_1^{(3)}, \text{ and }(0,P_{2,1},P_{2,2})Q =  e_2^{(3)},
            \]
            satisfies $P\CCC Q = \CCC$.
            \item Else, note that 
            \[
                \det\left(\begin{bmatrix}
                    P_{1,1} & P_{1,2} & 0 \\ P_{2,1} & P_{2,2} & 0 \\ 0 & P_{1,1} & P_{1,2}
                \end{bmatrix}\right) = P_{1,1}P_{1,2}P_{2,2} - P_{1,2}^2 P_{2,1} = P_{1,2}\det(P),
            \]
            thus the sequence $\{(P_{1,1},P_{1,2},0), (P_{2,1},P_{2,2},0),(0,P_{1,1},P_{1,2})\}$ is a basis of $\F_{q}^3$ and there exist $a,b,c \in \F_q$ such that 
            \[
                (0,P_{2,1},P_{2,2}) = a(P_{1,1},P_{1,2},0) + b(P_{2,1},P_{2,2},0) + c(0,P_{1,1},P_{1,2}).
            \]
            Since $P$ is invertible, we have $(a,b) \neq (0,0)$. Since $P_{1,1} \neq 0$ we have $b \neq 0$. Consequently, we have $a \neq -bc$, else we would have $-bcP_{11} + bP_{2,1} = 0$ as well as $-bcP_{1,2} + bP_{2,2} + cP_{1,1} =P_{2,1}$, which would imply that the rows of $P$ are collinear which is impossible. Thus the sequence $\{(-c,1,0),(0,-c,1),(a,b,0)\}$ is a basis of $\F_q^3$ and the invertible matrix $Q \in \GL_3(\F_q)$ such that 
            \[
                (P_{1,1},P_{1,2},0)Q = (-c,1,0), \  (P_{2,1},P_{2,2},0)Q = (0,-c,1), \text{ and }(0,P_{1,1},P_{1,2})Q = (a,b,0),
            \]
            satisfies $P\CCC Q = \CCC$.
        \end{itemize}

        \item \underline{$o_{3}$:} Denote by $A = \left[\begin{smallmatrix}
        1&0&0\\0&0&0
        \end{smallmatrix}\right]$, by $B =
        \left[\begin{smallmatrix}
        0&1&0\\0&0&0
        \end{smallmatrix}\right]$, by $C = 
        \left[\begin{smallmatrix}
        0&0&1\\0&0&0
        \end{smallmatrix}\right]$, and by $\CCC = \langle A,B,C\rangle$. With a similar argument as the one for the orbit $o_2$ above, we have 
        \[
            P\CCC Q = \CCC \iff \row_1(P) \in \langle e_1^{(2)}\rangle.
        \]
        Therefore, we have
        \[
            |\Stab(\CCC)| 
            = (q-1)(q^2-q)(q^3-1)(q^3-q)(q^3-q^2)
            = q^4(q-1)^5(q+1)(q^2+q+1).
        \]

        \item \underline{$o_{7}^T$:} Denote by $A = \left[\begin{smallmatrix}
        1&0&0\\0&0&0
        \end{smallmatrix}\right]$, by $B =
        \left[\begin{smallmatrix}
        0&0&0\\1&0&0
        \end{smallmatrix}\right]$, by $C = 
        \left[\begin{smallmatrix}
        0&0&0\\0&1&0
        \end{smallmatrix}\right]$, and by $\CCC = \langle A,B,C\rangle$. Then we have  
        \[
            P\CCC Q = \CCC \iff P_{1,2} = Q_{1,2} = Q_{1,3} = Q_{2,3} = 0.
        \]
        Assume that $P\CCC Q = \CCC$. We have $Q_{1,3} = Q_{2,3} = 0$ since the third columns of $A$, $B$ and $C$ are zero. Moreover, since $PAQ$, $PBQ$, and $PCQ$ have rank one, one of them has only its first column non-zero. Hence, either $\row_1(Q) \in \langle e_1^{(3)} \rangle$ or $\row_2(Q) \in  \langle e_1^{(3)}\rangle$. One can check that $\row_2(Q) \in  \langle e_1^{(3)}\rangle$ is impossible, so $\row_1(Q) \in \langle e_1^{(3)}\rangle$, and we have the wanted property. The converse is immediate. Therefore, we have
        \[
            |\Stab(\CCC)| 
            = (q-1)(q^2-q)(q-1)^3q^3
            = q^4(q-1)^5.
        \]

        \item \underline{$o_{8}$:} Denote by $A = \left[\begin{smallmatrix}
        1&0&0\\0&0&0
        \end{smallmatrix}\right]$, by $B =
        \left[\begin{smallmatrix}
        0&0&0\\0&1&0
        \end{smallmatrix}\right]$, by $C = 
        \left[\begin{smallmatrix}
        0&0&0\\0&0&1
        \end{smallmatrix}\right]$, and by $\CCC = \langle A,B,C\rangle$. Then a similar argument shows that
        \[
            P\CCC Q = \CCC \iff P_{1,2} = P_{2,1} = Q_{1,2} = Q_{1,3} = Q_{2,1} = Q_{3,1}= 0.
        \]
        Therefore we have
        \[
            |\Stab(\CCC)| 
            = (q-1)^2(q-1)(q^2-1)(q^2-q)
            = q(q-1)^5(q+1).
        \]

        \item \underline{$o_{9}$:} Denote by $A = \left[\begin{smallmatrix}
        1&0&0\\0&0&1
        \end{smallmatrix}\right]$, by $B =
        \left[\begin{smallmatrix}
        0&0&0\\1&0&0
        \end{smallmatrix}\right]$, by $C = 
        \left[\begin{smallmatrix}
        0&0&0\\0&1&0
        \end{smallmatrix}\right]$, and by $\CCC = \langle A,B,C\rangle$. Then we have 
        \[
            P\CCC Q = \CCC \iff P_{1,2} = Q_{1,2} = Q_{1,3} = Q_{2,3} = 0 \text{ and } Q_{3,3}= Q_{1,1}P_{1,1}P_{2,2}^{-1}.
        \]
        Assume that $P\CCC Q = \CCC$. The rank one matrices in $\CCC$ are non-trivial linear combinations of $B$ and $C$, hence $Q_{1,3} = Q_{2,3} = 0$. Moreover, the form of the last column of $PAQ$ implies that $e_2^{(2)}$ is an eigenvector of $P$, and thus we have $P_{1,2} = Q_{1,2} = 0$. Consequently we have $P_{1,1}Q_{1,1} = P_{2,2}Q_{3,3}$. The converse is immediate. Therefore, we have 
        \[
            |\Stab(\CCC)| 
            = (q-1)(q^2-q) (q-1)^2q^3
            = q^4(q-1)^4.
        \]

        \item \underline{$o_{11}^T$:} Denote by $A = \left[\begin{smallmatrix}
        1&0&0\\0&0&0
        \end{smallmatrix}\right]$, by $B =
        \left[\begin{smallmatrix}
        0&1&0\\1&0&0
        \end{smallmatrix}\right]$, by $C = 
        \left[\begin{smallmatrix}
        0&0&0\\0&1&0
        \end{smallmatrix}\right]$, and by $\CCC = \langle A,B,C\rangle$. Since $\CCC$ is the space of $2\times 3$ matrices whose left $2\times 2$ submatrix is symmetric, we have 
        \[
            P\CCC Q = \CCC \iff Q_{1,3} = Q_{2,3} = 0 \text{ and } P \in \left\langle \tilde Q^\top\right\rangle,
        \]
        where $\tilde Q$ is the $2\times 2$ upper-left submatrix of $Q$. Therefore we have
        \[
            |\Stab(\CCC)| 
            = (q^2-1)(q^2-q) q^2(q-1)(q-1)
            = q^3(q-1)^4(q+1).
        \]

        \item \underline{$o_{12}$:} Denote by $A = \left[\begin{smallmatrix}
        1&0&0\\0&0&1
        \end{smallmatrix}\right]$, by $B =
        \left[\begin{smallmatrix}
        0&1&0\\0&0&0
        \end{smallmatrix}\right]$, by $C = 
        \left[\begin{smallmatrix}
        0&0&0\\0&1&0
        \end{smallmatrix}\right]$, and by $\CCC = \langle A,B,C\rangle$. Then we have $P\CCC Q = \CCC$ if and only if 
        \[
             Q_{2,1} = Q_{2,3} = 0 \text{ , } (Q_{1,1},Q_{3,1}) \in \row_2(P)^\bot \text{ , and } \left[\begin{smallmatrix}
                Q_{1,3} \\ Q_{3,3}
            \end{smallmatrix} \right] = P^{-1}\left[\begin{smallmatrix}
                0 \\ P_{1,1}Q_{1,1} + P_{1,2}Q_{3,1}
            \end{smallmatrix} \right].
        \]
        Assume that $P\CCC Q = \CCC$. Since the rank one matrices in $\CCC$ are non-trivial linear combinations of $B$ and $C$, we have $Q_{2,1} = Q_{2,3} = 0$. Moreover, since $PAQ \in \CCC$, then we have $(P_{2,1},P_{2,2}) (Q_{1,1},Q_{3,1})^\top = 0$ as well as $P\left[\begin{smallmatrix}
        Q_{1,3} \\ Q_{3,3}
        \end{smallmatrix}\right] = \left[\begin{smallmatrix}
            0 \\ P_{1,1}Q_{1,1} + P_{1,2}Q_{3,1}
        \end{smallmatrix}\right]$. The converse is immediate. Therefore we have 
        \[
            |\Stab(\CCC)| 
            = (q^2-1)(q^2-q) q^2(q-1)^2
            = q^3(q-1)^4(q+1).
        \]

        \item \underline{$o_{13}$:} Denote by $A = \left[\begin{smallmatrix}
        1&0&0\\0&1&0
        \end{smallmatrix}\right]$, by $B =
        \left[\begin{smallmatrix}
        0&1&0\\0&0&0
        \end{smallmatrix}\right]$, by $C = 
        \left[\begin{smallmatrix}
        0&0&0\\0&0&1
        \end{smallmatrix}\right]$, and by $\CCC = \langle A,B,C\rangle$. Then we have $P\CCC Q = \CCC$ if and only if 
        \[
             P_{2,1} = P_{1,2} = Q_{1,3} = Q_{2,1} = Q_{2,3} = Q_{3,1} = Q_{3,2} = 0 \text{ and } Q_{2,2} = Q_{1,1}P_{1,1}P_{2,2}^{-1}.
        \]
        Assume that $P\CCC Q = \CCC$. Since the rank one matrices in $\CCC$ are either collinear to $B$ or to $C$, we have $(PBQ,PCQ) \in \langle B \rangle \times \langle C\rangle$ or $(PBQ,PCQ) \in \langle C \rangle \times \langle B\rangle$. Moreover, since $PAQ \in \CCC$, we have $P\left[\begin{smallmatrix}
        Q_{1,1} \\ Q_{2,1}
        \end{smallmatrix}\right] \in \langle e_1^{(2)}\rangle$, $P\left[\begin{smallmatrix}
        Q_{1,3} \\ Q_{2,3}
        \end{smallmatrix}\right] \in \langle e_2^{(2)}\rangle$ and $P_{1,1}Q_{1,1} + P_{1,2}Q_{2,1} = P_{2,1} Q_{1,2} + P_{2,2}Q_{2,2}$. Once can check that $(PBQ,PCQ) \in \langle C \rangle \times \langle B\rangle$ is then impossible, therefore, $P = \diag(P_{1,1},P_{2,2})$ and $Q_{2,1} = Q_{2,3} = Q_{3,1} = Q_{3,2}  = 0$. Since $(Q_{1,3},0) \in \langle e_1^{(2)}\rangle$, we have $Q_{1,3} = 0$, and we have $P_{1,1}Q_{1,1} = P_{2,2}Q_{2,2}$. Therefore we have
        \[
            |\Stab(\CCC)| 
            = q(q-1)^4.
        \]

        \item \underline{$o_{14}$:} Denote by $A_1 = \left[\begin{smallmatrix}
        1&0&0\\0&0&0
        \end{smallmatrix}\right]$, by $A_2 =
        \left[\begin{smallmatrix}
        0&1&0\\0&1&0
        \end{smallmatrix}\right]$, by $A_3= 
        \left[\begin{smallmatrix}
        0&0&0\\0&0&1
        \end{smallmatrix}\right]$, and by $\CCC = \langle A_1,A_2,A_3\rangle$. Then
        \[
            P\CCC Q = \CCC \implies \exists \sigma \in S_3, \forall i,j \in \llbracket 1,3\rrbracket : Q_{i,j} \neq 0 \implies j = \sigma(i),
        \]
        and that for every such matrix $Q$, there exists a unique up to scalar multiplication matrix $P \in \GL_2(\F_q)$ such that $P\CCC Q = \CCC$. Assume that $P\CCC Q = \CCC$. Since the rank one matrices in $\CCC$ are collinear to one of the three generators, then there exists a permutation $\sigma \in S_3$ such that $PA_iQ \langle A_{\sigma(i)}\rangle$ for each $i \in \llbracket 1,3\rrbracket$. Since $A_{\sigma(i)}$ has only its $\sigma(i)$-th column non-zero, we have 
        \[
            P\left[\begin{smallmatrix}
                Q_{1,\sigma(2)} \\ 0
            \end{smallmatrix} \right]
            =
            P\left[\begin{smallmatrix}
                Q_{1,\sigma(3)} \\ 0
            \end{smallmatrix} \right]
            =
            P\left[\begin{smallmatrix}
                Q_{2,\sigma(1)} \\  Q_{2,\sigma(1)}
            \end{smallmatrix} \right]
            =
            P\left[\begin{smallmatrix}
                Q_{2,\sigma(3)} \\  Q_{2,\sigma(3)}
            \end{smallmatrix} \right]
            =
            P\left[\begin{smallmatrix}
                0 \\  Q_{3,\sigma(1)}
            \end{smallmatrix} \right]
            =
            P\left[\begin{smallmatrix}
                0 \\  Q_{3,\sigma(2)}
            \end{smallmatrix} \right]
            =
            0,
        \]
        which gives the wanted shape of $Q$. Now, if we denote by $x_1 = \left[\begin{smallmatrix}
                1\\0
            \end{smallmatrix} \right]$,
             by $x_2 = \left[\begin{smallmatrix}
                1\\1
            \end{smallmatrix} \right]$, and by $x_3 = \left[\begin{smallmatrix}
                0\\1
            \end{smallmatrix} \right]$, there exists scalars $a_1,a_2,a_3 \in \F_q^\times$ such that 
        \[
            \forall i \in \llbracket 1,3\rrbracket : P x_i = a_iQ_{i,\sigma(i)}^{-1} x_{\sigma(i)}.
        \]
        Since $x_1 - x_2 + x_3 = 0$, the scalars $a_1,a_2,a_3$ are solution of the linear system with two equations 
        \[
            a_1Q_{1,\sigma(1)}^{-1} x_{\sigma(1)} - a_2Q_{2,\sigma(2)}^{-1} x_{\sigma(2)} + a_3Q_{3,\sigma(3)}^{-1} x_{\sigma(3)} = 0, 
        \]
        which has exactly $q-1$ non-zero solutions, defining $P$ up to scalar multiplication. Therefore, we have 
        \[
            |\Stab(\CCC)| 
            = 6 (q-1) (q-1)^3
            = 6(q-1)^4.
        \]  

        \item \underline{$o_{16}$:} Denote by $A = \left[\begin{smallmatrix}
        1&0&0\\0&1&0
        \end{smallmatrix}\right]$, by $B =
        \left[\begin{smallmatrix}
        0&1&0\\0&0&1
        \end{smallmatrix}\right]$, by $C = 
        \left[\begin{smallmatrix}
        0&0&1\\0&0&0
        \end{smallmatrix}\right]$, and by $\CCC = \langle A,B,C\rangle$. With the same reasoning as above, we have 
        \[
        P\CCC Q = \CCC \iff P_{2,1} = Q_{2,1} = Q_{3,1} = Q_{3,2} = 0 \text{ and } \left\{\begin{array}{l}
             Q_{2,2} = P_{1,1}^{-1}P_{2,2}Q_{3,3}\\
             Q_{1,1} =P_{1,1}^{-1}P_{2,2}Q_{2,2}\\
             Q_{1,2} = P_{1,1}^{-1}P_{2,2}(Q_{2,3}- P_{1,2}P_{1,1}^{-1}Q_{3,3})
        \end{array} \right..
        \]
        Therefore, the matrix $Q$ is function of its third column and the matrix $P$, so we have
        \[
            |\Stab(\CCC)| 
            = (q-1)(q^2-q) (q-1)q^2 
            = q^3(q-1)^3
        \]  
\end{itemize}
\end{proof}

\end{document}